\documentclass{amsart}

\usepackage{amsmath}
\usepackage{amsfonts}
\usepackage{amssymb}
\usepackage{graphicx}
\usepackage{amssymb,latexsym,amsmath,graphics,setspace,amscd,verbatim}
\usepackage{epsfig,epsf,amsthm,amsfonts}
\usepackage[active]{srcltx}

\theoremstyle{plain}
\newtheorem{lemma}{Lemma}[section]
\newtheorem{theorem}[lemma]{Theorem}
\newtheorem{propn}[lemma]{Proposition}
\newtheorem{cor}[lemma]{Corollary}
\newtheorem{claim}[lemma]{Claim}

\theoremstyle{definition}
\newtheorem{defn}[lemma]{Definition}
\newtheorem{remark}[lemma]{Remark}
\newtheorem{notation}[lemma]{Notation}

\theoremstyle{remark}

\newcommand{\norma}[1]{\left|{#1}\right|}
\newcommand{\C}{\mathbb{C}}
\newcommand{\R}{\mathbb{R}}
\newcommand{\Z}{\mathbb{Z}}

\newcommand{\D}{\mathbb{D}}
\newcommand{\est}{e^{2\pi(s+it)}}
\newcommand{\util}{\tilde{u}}
\newcommand{\wtil}{\tilde{w}}
\newcommand{\vtil}{\tilde{v}}
\newcommand{\jtil}{\tilde{J}}
\newcommand{\jcal}{\mathcal{J}}

\newcommand{\Dcal}{\mathcal{D}}
\newcommand{\Lcal}{\mathcal{L}}
\newcommand{\Mcal}{\mathcal{M}}
\newcommand{\interior}[1]{\mathring{{#1}}}
\newcommand{\escp}[2]{\left<{#1},{#2}\right>}
\newcommand{\wind}{\text{wind}}
\newcommand{\cl}[1]{\overline{{#1}}}

\newcommand{\mob}{\text{M\"ob}}
\newcommand{\p}{\mathcal{P}}
\newcommand{\W}{\mathcal{W}}

\newcommand{\Ucal}{\mathcal{U}}
\newcommand{\OO}{\mathcal{O}}

\newcommand{\cov}{\text{cov}}
\newcommand{\Sp}{\text{Sp}}

\newcommand{\tr}{\text{tr }}
\newcommand{\maslov}{\text{Maslov}}

\begin{document}

\title[FAST FINITE-ENERGY PLANES]{FAST FINITE-ENERGY PLANES IN SYMPLECTIZATIONS\\AND APPLICATIONS}
\author[Umberto Hryniewicz]{Umberto Hryniewicz}
\address[Umberto Hryniewicz]{Departamento de Matem\'atica Aplicada, IM-UFRJ, Rio de Janeiro, Brazil.}
\email{umberto@labma.ufrj.br}
\date{June 24, 2011 (preliminary version)}

\begin{abstract}
We define the notion of fast finite-energy planes in the symplectization of a closed $3$-dimensional energy level $M$ of contact type. We use them to construct special open book decompositions of $M$ when the contact structure is tight and induced by a (non-degenerate) dynamically convex contact form. The obtained open books have disk-like pages that are global surfaces of section for the Hamiltonian dynamics. Let $S \subset \R^4$ be the boundary of a smooth, strictly convex, non-degenerate and bounded domain. We show that a necessary and sufficient condition for a closed Hamiltonian orbit $P\subset S$ to be the boundary of a disk-like global surface of section for the Hamiltonian dynamics is that $P$ is unknotted and has self-linking number $-1$.
\end{abstract}

\maketitle

\tableofcontents

\section{Introduction}\label{intro}

We intend to give a systematic treatment to the procedure of constructing global surfaces of section for the Hamiltonian dynamics on strictly convex $3$-dimensional energy levels.

\begin{defn}
A global surface of section for a vector field $X$ on a $3$-manifold $M$ is a compact embedded surface $\Sigma \hookrightarrow M$ satisfying:
\begin{enumerate}
 \item $X$ is transversal to $\Sigma\setminus \partial\Sigma$ and $\partial \Sigma$ consists of periodic orbits of $X$.
 \item For every $x \in M \setminus \partial \Sigma$ one finds sequences $t^\pm_n \rightarrow \pm\infty$ such that $\phi_{t^\pm_n}(x) \in \Sigma$.
\end{enumerate}
Here $\phi_t$ denotes the flow of $X$.
\end{defn}

In~\cite{convex} Hofer, Wysocki and Zehnder studied the Hamiltonian dynamics on a bounded and strictly convex regular level $S = H^{-1}(1) \subset \R^4$, where $H : \R^4 \rightarrow \R$ is a smooth Hamiltonian. If $z = (q_1,p_1,q_2,p_2)$ are linear coordinates in $\R^4$ equipped with its canonical symplectic form
\begin{equation}
 \omega_0 = dq_1 \wedge dp_1 + dq_2 \wedge dp_2,
\end{equation}
then Hamilton's equations can be rewritten as
\[
 \dot z = X_H(z)
\]
where $X_H$ is the so-called Hamiltonian vector field. It is uniquely determined by
\[
 i_{X_H}\omega_0 = dH
\]
and its flow preserves level-sets of $H$. They consider the case when $S = \partial K$ for some bounded, smooth and strictly convex domain $K \subset \R^4$. If $G$ is some other Hamiltonian realizing $S$ as a regular energy level then $\R X_G|_S = \R X_H|_S$, that is, Hamiltonian dynamics depends only on $S$ and $\omega_0$, up to time-reparametrization. This can be checked by inspection, or by noting that $$ \R X_H(z) = \left( T_zS \right) ^{\omega_0} \ \forall z\in S, $$ where $\left( T_zS \right) ^{\omega_0} := \left\{ v \in \R^4 : \omega_0(v,u) = 0, \ \forall u \in T_zS \right\}$ is the $\omega_0$-symplectic orthogonal of $T_zS$. The line bundle $(TS)^{\omega_0}$ is called the characteristic line field.

The study of Hamiltonian dynamics on such strictly convex hypersurfaces is now a classical subject. In 1978 P. Rabinowitz \cite{rab1} and A. Weinstein \cite{weinstein} proved existence of periodic orbits on these energy levels. In~\cite{convex} Hofer, Wysocki and Zehnder proved the following remarkable result.

\begin{theorem}[Hofer, Wysocki and Zehnder]\label{sectionconvex}
Let $S \subset \R^4$ be a bounded, smooth and strictly convex hypersurface. Then $S$ carries an unknotted periodic Hamiltonian orbit $P_0$ bounding a disk-like global surface of section $\Dcal$ for the Hamiltonian dynamics.
\end{theorem}

The Poincar\'e return map to $\interior{\Dcal}$ preserves a smooth area-form, with total area $\int_{\Dcal}\omega_0<\infty$. Brouwer's translation theorem provides a second periodic orbit $P_1$, geometrically distinct of $P_0$. It corresponds to a fixed point of the first return map to $\interior{\Dcal}$. One can, as described in~\cite{convex}, apply results of J. Franks \cite{franks} on periodic points of area-preserving diffeomorphisms of the open annulus to obtain the following important corollary.

\begin{cor}[Hofer, Wysocki and Zehnder]
Hamiltonian dynamics on a bounded, smooth, strictly convex energy level inside $\R^4$ has either $2$ or $\infty$-many periodic orbits.
\end{cor}

Theorem~\ref{sectionconvex} immediately prompts the following question: What are the necessary and sufficient conditions for a periodic Hamiltonian orbit to bound an embedded disk-like global surface of section?

Our main result answers this question when $S$ is non-degenerate, that is, when $1$ is not a transverse eigenvalue of the linearized Poincar\'e return map of every closed orbit. This is a $C^\infty$-generic condition on $S$. Our answer is stated in terms of a certain contact-topological invariant, called the {\it self-linking number}, which we now describe.

A $1$-form $\lambda$ on a $3$-manifold $M$ is a contact form if $\lambda\wedge d\lambda$ never vanishes. The associated contact structure is the hyperplane distribution
\begin{equation}\label{kernel}
 \xi = \ker \lambda.
\end{equation}
A co-oriented contact $3$-manifold is a pair $(M,\xi)$ such that (\ref{kernel}) holds for some contact form $\lambda$. We call $\lambda$ tight if $\xi$ is a tight contact structure, see Subsection~\ref{tightstr}. The associated Reeb vector field $R$ is defined implicitly by
\begin{equation}\label{reebvector}
 \begin{array}{ccc}
   i_Rd\lambda = 0 & \text{ and } & i_R\lambda = 1.
 \end{array}
\end{equation}

\begin{defn}[Self-linking Number]\label{selflink}
Let $L \subset M$ be a knot transverse to $\xi$, and let $\Sigma \hookrightarrow M$ be a Seifert surface\footnote{By a Seifert surface for $L$ we mean an orientable embedded connected compact surface $\Sigma \hookrightarrow M$ such that $L = \partial \Sigma$.} for $L$. Since the bundle $\xi|_\Sigma$ carries the symplectic bilinear form $d\lambda$, there exists a smooth non-vanishing section $Z$ of $\xi|_\Sigma$, which can be used to slightly perturb $L$ to another transverse knot $L_\epsilon = \{ \exp_x (\epsilon Z_x) : x\in L \}$. Here $\exp$ is any exponential map. A choice of orientation for $\Sigma$ induces orientations of $L$ and of $L_\epsilon$. The {\it self-linking number} is defined as the oriented intersection number
\begin{equation}\label{defselflink0}
 sl(L,\Sigma) := L_\epsilon \cdot \Sigma \in \Z,
\end{equation}
where $M$ is oriented by $\lambda\wedge d\lambda$. It is independent of $\Sigma$ when $c_1(\xi) \in H^2(M)$ vanishes.
\end{defn}

Recall that $\omega_0$ has a special primitive $\lambda_0 = \frac{1}{2} \sum_{k=1}^2 q_kdp_k - p_kdq_k$. We assume $0$ is in the bounded component $K$ of $\R^4\setminus S$, so that $\lambda_0|_S$ is a contact form. If $R$ is the associated Reeb field on $S$ then $\R R = \R X_H$. Our main result can be stated as follows.

\begin{theorem}\label{thmA}
Let $S \subset \R^4$ be a non-degenerate, bounded, smooth and strictly convex hypersurface. A necessary and sufficient condition for a periodic Hamiltonian orbit $P\subset S$ to be the boundary of a disk-like global surface of section for the Hamiltonian dynamics is that $P$ is unknotted and has self-linking number $-1$.
\end{theorem}

Necessity is an easy computation, see Proposition~\ref{necessity}. The assumption of $S$ being non-degenerate is rather technical and will be removed in~\cite{hry}. The reader acquainted with the work of Hofer, Wysocki and Zehnder will notice that it can be removed by arguments found in~\cite{convex}. In~\cite{hry} we shall also prove that a periodic orbit associated to any fixed point of Poincar\'e's first return map to the global surface of section obtained from Theorem~\ref{sectionconvex} has self-linking number $-1$. As a consequence, it also bounds a disk-like global surface of section.

Sufficiency in Theorem~\ref{thmA} will follow from a more general result, which we now describe. Motivated by~\cite{convex}, we consider {\it systems of global surfaces of section} organized in the form of open book decompositions.

\begin{defn}\label{adapted}
An open book decomposition of an oriented $3$-manifold $M$ is a pair $(L,p)$ where $L$ is an oriented link in $M$, and $p:M\setminus L \rightarrow S^1$ is a fibration such that each fiber $p^{-1}(\theta)$ is the interior of an oriented compact embedded surface $S_\theta \hookrightarrow M$ satisfying $\partial S_\theta = L$ (including orientations). $L$ is called the binding and the fibers are called pages. It is said to be adapted to the dynamics of a vector field $X$ if $L$ consists of periodic orbits, $X$ orients $L$ positively, the pages are global surfaces of section and the orientation of $M$ together with $X$ induce the orientation of the pages.
\end{defn}

We are interested in constructing open book decompositions adapted to the dynamics of Reeb vector fields. The study of this problem was initiated by Hofer, Wysocki and Zehnder in~\cite{char1,char2}. Their proofs are based on the theory of pseudo-holomorphic curves in symplectizations, introduced by Hofer in~\cite{93}. In the constructions they use disk-filling methods, bubbling-off analysis and their own perturbation theory. This work is the first of three articles extending their results.

Consider a contact form $\lambda$ on a closed $3$-manifold $M$. A periodic Reeb orbit will be denoted by $P=(x,T)$, where $x:\R\rightarrow M$ solves $\dot x = R\circ x$ and $T>0$ is a period. It is called simply covered when $T$ is the minimal positive period of $x$, and unknotted if it is simply covered and $x(\R)$ is the unknot. When $\Sigma$ is a Seifert surface for the transverse knot $x(\R)$ we write $sl(P,\Sigma)$ instead of $sl(x(\R),\Sigma)$. Since the Reeb flow $\phi_t$ preserves $\lambda$ we have well-defined $d\lambda$-symplectic maps $d\phi_t(x(t_0)) : \xi_{x(t_0)} \rightarrow \xi|_{x(t_0+t)}$. The periodic orbit $P$ is non-degenerate if $1$ is not an eigenvalue of $d\phi_T(x(t_0))|_{\xi}$, $\forall t_0\in\R$. When every $P$ is non-degenerate one says $\lambda$ is non-degenerate. See Subsection~\ref{orbits} for a more precise discussion.

There is an important invariant of the linearized dynamics along a periodic orbit originally introduced in~\cite{conley}, called the Conley-Zehnder index, which we now describe. Let $\Sp(n)$ be the symplectic linear group in dimension $2n$, and denote
\[
 \Sigma^* = \left\{ \varphi \in C^\infty([0,1],Sp(1)) : \varphi(0) = I \text{ and } \det \left[\varphi(1) - I\right] \not= 0 \right\}.
\]
In~\cite{fols} we find the following axiomatic characterization of the Conley-Zehnder index in the case $n=1$. For other discussions on this topic see~\cite{rob} or~\cite{salzehn}.

\begin{theorem}\label{axiomscz}
There exists a unique surjective map $\mu : \Sigma^* \rightarrow \Z$ characterized by the following axioms:
\begin{enumerate}
 \item {\bf Homotopy:} If $\varphi_s$ is a homotopy of arcs in $\Sigma^*$ then $\mu\left(\varphi_s\right)$ is constant.
 \item {\bf Maslov index:} If $\psi:\left(\R/\Z,0\right) \rightarrow \left(Sp(1),I\right)$ is a loop  and $\varphi\in\Sigma^*$ then $\mu(\psi\varphi) = 2\text{Maslov}(\psi) + \mu(\varphi)$.
 \item {\bf Invertibility:} If $\varphi \in \Sigma^*$ and $\varphi^{-1}(t) := \varphi(t)^{-1}$ then $\mu\left(\varphi^{-1}\right) = - \mu(\varphi)$.
 \item {\bf Normalization:} $\mu\left(t\mapsto e^{i\pi t}\right) = 1$.
\end{enumerate}
\end{theorem}

Let $P=(x,T)$ be a non-degenerate periodic Reeb orbit and denote by $\xi_P$ the bundle $x_T^*\xi\to\R/\Z$, where $x_T(t) = x(Tt)$. Consider the set $\mathcal{S}_P$ of homotopy classes of smooth $d\lambda$-symplectic trivializations of $\xi_P$, and fix $\beta \in \mathcal{S}_P$. Any trivialization in class $\beta$ can be used to represent the linear maps $d\phi_{Tt}:\xi_{x(0)} \rightarrow \xi_{x(Tt)}$ by some path $\varphi\in\Sigma^*$. We write $\mu_{CZ}(P,\beta) = \mu(\varphi)$. By axiom (1) above this is independent of the particular trivialization in class $\beta$. When $P$ is contractible we shall say $c_1(\xi)$ vanishes along $P$ if the $2$-sphere obtained by gluing any two disk-maps spanning the map $e^{i2\pi t} \mapsto x(Tt)$ lies in the kernel of $c_1(\xi)$. In this case, there exists a special class $\beta_P \in \mathcal{S}_P$ induced by some, and hence any, such disk-map. We can define $\mu_{CZ}(P) := \mu_{CZ}(P,\beta_P)$, see Subsection~\ref{orbits}.

In~\cite{93} Hofer introduced special almost complex structures in the symplectization $\R\times M$ and considered solutions of the associated Cauchy-Riemann equations in order to study the dynamics of Reeb vector fields. As a consequence of his ground-breaking work, the three-dimensional Weinstein Conjecture\footnote{In~\cite{taubes} C. Taubes confirmed the three-dimensional Weinstein Conjecture using Seiberg-Witten theory.} was confirmed in many cases, including all contact forms on $S^3$. We need to recall a few concepts from~\cite{93} in order to discuss our results.

A complex structure $J$ on $\xi$ is $d\lambda$-compatible if $d\lambda(\cdot,J\cdot)$ is a metric on $\xi$. We write $\jcal(\xi,d\lambda|_\xi)$ for the set of such complex structures. Following~\cite{93}, every $J \in \jcal(\xi,d\lambda|_\xi)$ induces an almost complex structure $\jtil$ on $\R\times M$ by
\begin{equation}\label{almcpxstr}
 \begin{array}{ccc}
   \jtil \cdot \partial_a=R & \text{ and } & \jtil|_\xi \equiv J,
 \end{array}
\end{equation}
where $a$ denotes the $\R$-component. A finite-energy plane is a $\jtil$-holomorphic map $\util : (\C,i) \to (\R\times M,\jtil)$ with positive and finite Hofer energy, see Subsection~\ref{fesurfaces} for the precise definitions. The following result is central in~\cite{char2}, see also~\cite{char1}.

\begin{theorem}[Hofer, Wysocki and Zehnder]\label{chars3}
Let $\lambda$ be a tight contact form on a closed $3$-manifold $M$ and assume $c_1\left(\xi\right)$ vanishes. Let $P_0 = (x_0,T_0)$ be an unknotted non-degenerate closed Reeb orbit satisfying $\mu_{CZ}(P_0) = 3$ and $sl(P_0) = -1$. Suppose every contractible orbit $P = (x,T)$ with $T\leq T_0$ is non-degenerate and satisfies $\mu_{CZ}(P)\geq 3$. Then, for generic choice of $J \in \jcal(\xi,d\lambda|_\xi)$, there exists an embedded finite-energy $\jtil$-holomorphic plane $\util_0$ in the sympelctization $\R\times M$ asymptotic to $P_0$ at $\infty$. Its projection onto $M$ does not intersect $x_0(\R)$ and is only one page of an open book decomposition adapted to the Reeb dynamics. The binding is $P_0$ and the pages are open disks. In particular, $M = S^3$ and $\xi$ is its unique (up to contactmorphism) tight contact structure.
\end{theorem}

As explained before, Brouwer's translation theorem and results of J. Franks from~\cite{franks} provide an important corollary.

\begin{cor}[Hofer, Wysocki and Zehnder]\label{corchars3}
Under the assumptions of Theorem~\ref{chars3} the Reeb dynamics has either $2$ or $\infty$-many periodic orbits.
\end{cor}

Here we shall define a class of pseudo-holomorphic curves suitable for this construction. They generalize the curves used in~\cite{char1}, \cite{char2} and~\cite{convex}. 

\begin{defn}{(Fast planes)}\label{fastplane}
A finite-energy plane $\util:\C\to\R\times M$ is said to be fast if $\infty$ is a non-degenerate puncture, $\wind_\infty (\util) = 1$ and $\cov(\util) = 1$.
\end{defn}

Let us briefly describe the invariants $\wind_\infty$ and $\cov$, originally introduced in~\cite{props2}. If we write $\util = (a,u) \in \R\times M$ and assume $\lambda$ is non-degenerate then the loops $t \mapsto u(\est)$ converge to $x(Tt+c)$ in $C^\infty(\R/\Z,M)$ as $s\to+\infty$, where $P=(x,T)$ is some closed Reeb orbit and $c\in\R$. This follows from Theorem~\ref{thm93} below and the definition of non-degenerate punctures given in Subsection~\ref{nondegasympbehavior}, see~\cite{props1}. In this case one says $\util$ is asymptotic to $P$. Condition $\cov(\util) = 1$ means that $P$ is simply covered. The identity $\wind_\infty (\util) = 1$ holds if, and only if, $u:\C\to M$ is an immersion transverse to the Reeb vector field, see Lemma~\ref{gauss}.

The term ``fast'' used above alludes to the following alternative interpretation of the identity $\wind_\infty(\util)=1$. The map $\util$ is $\jtil$-holomorphic. This means that $s\mapsto u(\est)$ can be thought as a gradient trajectory of the action functional converging to one of its critical points $P$. One has the corresponding Hessian $A_P$ for the action at $P$, which is a self-adjoint operator on a suitable Hilbert space of sections of $\xi|_P$. Its spectrum is real, discrete, accumulates only at $\pm\infty$ and consists of eigenvalues. It can be ordered according to the windings of the corresponding eigensections. All this is proved in~\cite{props2}, see Subsection~\ref{asympopprops}. Condition $\wind_\infty(\util)=1$ tells us that $s\mapsto u(\est)$ is a trajectory on the fastest piece of the stable manifold of $P$, keeping a ``maximally weighted'' Fredholm index of $\util$ $\geq 1$. This fast decay is the source of compactness properties of fast planes, as showed in Section~\ref{compactness}. We are ready for our second statement.

\begin{theorem}\label{openbook}
Let $\lambda$ be a tight contact form on a closed $3$-manifold $M$. Assume the following holds for every contractible periodic Reeb orbit $\hat P$:
\begin{itemize}
 \item[(i)] $\hat P$ is non-degenerate,
 \item[(ii)] $c_1(\xi)$ vanishes along $\hat P$ and $\mu_{CZ}(\hat P) \geq 3$.
\end{itemize}
Let $P = (x,T)$ be a simply covered periodic Reeb orbit. A necessary and sufficient condition for $P$ to be the binding of an open book decomposition adapted to the Reeb dynamics with disk-like pages is that $P$ is unknotted and $sl(P,disk) = -1$. When $P$ fulfills these conditions then $\forall \ l\geq 1$ there exists $J \in \mathcal J(\xi,d\lambda|_\xi)$ and a $C^l$-map
\[
 \util = (a,u) : S^1 \times \C \rightarrow \R \times M
\]
satisfying:
\begin{enumerate}
 \item Each $\util(\vartheta,\cdot)$ is an embedded fast finite-energy $\jtil$-holomorphic plane asymptotic to $P$.
 \item $u(\vartheta,\C) \cap x(\R) = \emptyset \ \forall \vartheta\in S^1$ and the map $u : S^1 \times \C \rightarrow M \setminus x(\R)$ is an orientation preserving $C^l$-diffeomorphism.
 \item Each $\cl{u(\vartheta,\C)}$ is a global surface of section for the Reeb flow.
\end{enumerate}
\end{theorem}

The self-linking number $sl(P,disk)$ is computed using any embedded disk for the unknot $x(\R)$, and is independent of this choice since $c_1(\xi)$ vanishes along $P$. Necessity is given by Proposition~\ref{necessity}. Sufficiency in Theorem~\ref{thmA} follows from the above statement, since $\ker \lambda_0|_S$ is tight. The proof of sufficiency in Theorem~\ref{openbook} can be found in Section~\ref{cons}, for a sketch see Subsection~\ref{sketchopenbook}. Following Hofer, Wysocki and Zehnder, there is an important corollary.

\begin{cor}\label{2infty}
Suppose $M$ and $\lambda$ satisfy the conditions of Theorem~\ref{openbook}. If there exists an unknotted periodic Reeb orbit $P$ with $sl(P,disk)=-1$ then $M\simeq S^3$, $\xi$ is contactomorphic to the positive tight contact structure on $S^3$ and the associated Reeb flow has either $2$ or $\infty$-many closed orbits.
\end{cor}

\noindent {\bf Organization of the article.} In Section~\ref{outline} we outline the proofs of Theorem~\ref{thmA} and Theorem~\ref{openbook}. In Subsection~\ref{necessity} we prove necessity in Theorem~\ref{openbook}. The proof of sufficiency requires three analytical tools: compactness, perturbation theory and existence for fast planes. They are explained in Subsections~\ref{outlinecompactness}, \ref{outlinefredholm} and~\ref{existence} respectively. In Subsection~\ref{sketchopenbook} we outline the proof of sufficiency in Theorem~\ref{openbook}, and explain how Theorem~\ref{thmA} follows from Theorem~\ref{openbook}. In Section~\ref{basicdefns} we recall the standard definitions from the theory of finite-energy surfaces in symplectizations. Section~\ref{compactness} is devoted to our compactness result, Theorem~\ref{main3}. In Section~\ref{sectionexistence} we prove our existence result for fast planes, Theorem~\ref{existfast}. In Section~\ref{fredholm} we describe the perturbation theory, Theorem~\ref{lemmafredholm}, where some technical lemmas are postponed to the Appendix. In Section~\ref{cons} we prove sufficiency in Theorem~\ref{openbook}. \\

\noindent {\bf Acknowledgements.} {\it We thank P. Albers, B. Bramham, D. Jane, J. Koiller, L. Macarini, A. Momin, C. Niche and C. Wendl for many helpful discussions, and P. Salom\~ao for his unconditional support and all the mathematical help. We are particularly grateful to the referee for many suggestions that improved the article significantly. We would especially like to thank Professor H. Hofer for proposing such beautiful problems, for all his generosity, and for his mathematical advice during the years the author spent at the Courant Institute as a student.}

\section{Outline of main arguments}\label{outline}

In this section we prove necessity in theorems~\ref{thmA} and~\ref{openbook}, and then outline the proof of
sufficiency.

\subsection{Proof of necessity in Theorem~\ref{openbook}}\label{necessity}

\begin{propn}\label{propnec}
Let $\lambda$ be a contact form on an oriented $3$-manifold $M$ satisfying $\lambda \wedge d\lambda
>0$, and let $R$ be the associated Reeb vector field. Suppose $P$ is an unknotted periodic Reeb
orbit. If there exists an embedded disk $\Dcal \subset M$ satisfying $\partial \Dcal = P$ and $\R R_p \cap T_p\Dcal
= \{0\}, \ \forall p \in \interior{\Dcal}$, then $sl(P,\Dcal) = -1$.
\end{propn}

\begin{proof}
Let $\varphi : \D \rightarrow M$ be a smooth embedding such that $\Dcal := \varphi(\D)$ satisfies $\partial \Dcal =
P$ and $\R R_p \cap T_p\Dcal = \{0\}, \ \forall p \in \interior{\Dcal}$. Here $\D \subset \R^2$ is the closed unit
disk, which we equip with euclidean coordinates $(x,y)$. We orient $\Dcal$ so that $\lambda|_{\partial \Dcal=P} >
0$, and assume $\varphi$ is orientation preserving when $\D$ is equipped with its standard orientation. Orient $\xi$
by $d\lambda|_{\xi}$ and $M$ by $\lambda\wedge d\lambda$. Let $\pi : TM \rightarrow \xi$ denote the projection along
$\R R$. Since $R$ is never tangent to $\interior{\Dcal}$, $\pi : T\interior{\Dcal} \rightarrow
\xi|_{\interior{\Dcal}}$ is an isomorphism. We claim it is orientation preserving. In fact, let $\sigma$ be a
positive smooth area form on $\Dcal$, and let $f:\Dcal \rightarrow \R$ be defined by $d\lambda|_{T\Dcal} = f\sigma$.
We know $f$ does not vanish on $\interior{\Dcal}$ since $f(p) = 0 \Leftrightarrow \R R_p \subset T_p\Dcal$. Our
choice of orientation of $\Dcal$ gives $$ \int_\Dcal d\lambda = \int_{\partial \Dcal} \lambda > 0, $$ implying $f>0$
on $\interior{\Dcal}$. Fix $p \in \interior{\Dcal}$ and $u,v \in T_p\Dcal$ such that $\sigma(u,v) > 0$. Then $$
d\lambda(\pi \cdot u,\pi\cdot v) = d\lambda (u,v) = f(p) \sigma(u,v) > 0, $$ proving our claim. The bundle map $\pi
\cdot d\varphi : T\interior{\D} \rightarrow \xi|_{\interior{\Dcal}}$ is orientation preserving since so are
$\varphi$ and $\pi : T\interior{\Dcal} \rightarrow \xi|_{\interior{\Dcal}}$. This will be important in what follows.
Consider the smooth section $W := \pi \cdot \left( x\partial_x\varphi + y\partial_y\varphi \right)$ of
$\xi|_{\Dcal}$. It does not vanish on $\partial \Dcal = P$. Let $Z$ be a non-vanishing section of $\xi|_{\Dcal}$.
Fix any exponential map $\exp$ on $M$ and consider the transverse unknots
\[
  \begin{array}{ccc}
    P^Z_\epsilon := \{ \exp_p (\epsilon Z_p) : p \in P \} & \text{ and } & P^W_\epsilon := \{ \exp_p (\epsilon W_p) : p \in P \}
  \end{array}
\]
where $0<\epsilon\ll1$. Let $J$ be any complex structure on the bundle $\xi|_{P}$ such that
$d\lambda(p)(\cdot,J_p\cdot)$ is a positive inner-product on $\xi|_p$, $\forall p\in P$. This defines a unique
non-vanishing smooth map $f = u+iv : P \rightarrow \C\setminus\{0\}$ by $W_p = u(p) Z_p + v(p) J_p Z_p$, $p\in P$.
Let $k := \deg f/|f|$. Standard degree theory tells us that $k$ is the algebraic count of zeros of $W$ on $\Dcal$
and that $P^W_\epsilon \cdot \Dcal = P^Z_\epsilon \cdot \Dcal + k$. Thus
\[
 sl(P,\Dcal) = \left( P^W_\epsilon \cdot \Dcal \right) - k.
\]
Since $\Dcal$ is transverse to $\R R$ on $\interior{\Dcal}$, the only zero of $W$ is at the point $\varphi(0)$. We
now claim $k=1$. In fact, define $\hat J := (\pi\cdot d\varphi)^*J$. Then $\hat J$ is an almost complex structure on
$T\interior{\D}$ satisfying $\det \hat J_p = 1, \ \forall p \in \interior{\D}$. One finds a smooth path $J_t$ of
almost complex structures on $T\interior{\D}$ satisfying $\det J_t|_q = 1, \ \forall (t,q) \in [0,1] \times
\interior{\D}$, $J_0 = i$ and $J_1 = \hat J$. Consider a disk $D_\delta$ of radius $0<\delta\ll1$ centered at the
origin. On $D_\delta$ there exists a non-vanishing section $Y := \pi \cdot \partial_x\varphi$ of
$\xi|_{\varphi(D_\delta)}$. When $A$ and $B$ are two non-vanishing sections of $\xi|_{\varphi(\partial D_\delta)}$
we write $A \sim B$ if they are homotopic through non-vanishing sections. We parametrize $\partial D_\delta$ by
$\theta \mapsto \delta e^{i2\pi\theta}$, $\theta \in [0,1]$, and note that
\[
 \begin{aligned}
  \left\{ \theta \mapsto W(\varphi(\delta e^{i2\pi\theta})) \right\} & \sim \left\{ \theta \mapsto \delta^{-1} W(\varphi(\delta e^{i2\pi\theta})) = \pi \cdot d\varphi(\delta e^{i2\pi\theta}) \cdot e^{i2\pi \theta} \right\} \\
  & \sim \left\{ \theta \mapsto \pi \cdot d\varphi(\delta e^{i2\pi\theta}) \cdot \exp(\hat J 2\pi \theta) \cdot \begin{pmatrix} 1 \\ 0 \end{pmatrix} \right\} \\
  & \sim \left\{ \theta \mapsto \exp(J|_{\varphi(\delta e^{i2\pi\theta})} 2\pi \theta) \cdot \pi \cdot d\varphi(\delta e^{i2\pi\theta}) \cdot \begin{pmatrix} 1 \\ 0 \end{pmatrix} \right\} \\
  & = \left\{ \theta \mapsto \exp(J|_{\varphi(\delta e^{i2\pi\theta})} 2\pi \theta) \cdot Y(\varphi(\delta e^{i2\pi\theta})) \right\}
 \end{aligned}
\]
Write $W(\varphi(\delta e^{i2\pi\theta})) = a(\theta) Y(\varphi(\delta e^{i2\pi\theta})) + b(\theta)
J|_{\varphi(\delta e^{i2\pi\theta})} Y(\varphi(\delta e^{i2\pi\theta}))$ and define $h = a + ib$. The above
calculation shows that $\deg h/|h| = 1$. It follows from standard degree theory that $k = 1$ since $W$ has no zeros
on $\D \setminus D_{\delta}$. Now consider the normal derivative $\theta \in \R/\Z \mapsto A(\theta) :=
d\varphi(e^{i2\pi\theta})\cdot e^{i2\pi\theta}$, and the map
\[
 (\theta,t) \in \R/\Z \times [0,1] \mapsto (1-t) \pi \cdot A + tA.
\]
It provides a smooth homotopy from the vector $\theta \mapsto W(\varphi(e^{i2\pi\theta}))$ to the vector $\theta
\mapsto A(\theta)$ through non-vanishing vectors. This shows $P^W_\epsilon \cdot \Dcal = 0$ and completes the proof
that $sl(P,\Dcal) = -1$.
\end{proof}


\subsection{Compactness}\label{outlinecompactness}

Since we deal with higher Conley-Zehnder indices, we need new compactness arguments replacing those given by Hofer,
Wysocki and Zehnder, see~\cite{props2},~\cite{char1} and~\cite{char2}. Loosely speaking, we shall prove that, under
convexity assumptions on $\lambda$, breaking of Morse trajectories for the action functional does not occur in
families of unparametrized fast planes through a compact $H \subset \R\times M$. As an example, this will be the
case for non-degenerate dynamically convex contact forms on $S^3$.

Let $\lambda$ be a contact form on a closed $3$-manifold $M$, and let $P = (x,T)$ be a simply covered periodic Reeb
orbit. Assume every contractible periodic Reeb orbit $\hat P = (\hat x,\hat T)$ with $\hat T\leq T$ is
non-degenerate. Consider the set $\mathcal{A}_c$ of positive periods of contractible periodic Reeb trajectories and
define $\mathcal{A}_c^k = \{\tau \in \mathcal{A}_c : \tau \leq k\}$. Following~\cite{fols}, we define
\[
 \begin{array}{cc}
   \gamma_1=\min \mathcal{A}_c^T,  & \gamma_2=\min \left\{ |\tau_1 - \tau_2| : \tau_1 \not= \tau_2;\ \tau_1,\tau_2 \in \mathcal{A}_c^T \right\}
 \end{array}
\]
and fix a number
\begin{equation}\label{gamma}
  0 < \gamma < \min\{\gamma_1,\gamma_2\}.
\end{equation}
Recall the almost complex structure $\jtil$~\eqref{almcpxstr} associated to some $J \in \jcal(\xi,d\lambda|_\xi)$. We fix a subset $H\subset\R\times M$ and define
\begin{equation}\label{setTheta}
 \Theta(H,P,\lambda,J) \subset C^{\infty}(\C,\R\times M)
\end{equation}
by requiring that $\util\in\Theta(H,P,\lambda,J)$ if, and only if, $\util$ is a fast finite-energy
$\jtil$-holomorphic plane asymptotic to $P$, $\util(0)\in H$ and $\int_{\C\setminus\D} u^*d\lambda = \gamma$. We
define
\begin{equation}\label{setLambda}
 \Lambda(H,P,\lambda,J) \subset \Theta(H,P,\lambda,J)
\end{equation}
by requiring that $\util \in \Lambda(H,P,\lambda,J)$ if, and only if, $\util \in \Theta(H,P,\lambda,J)$ is an
embedding. Consider
\begin{gather}\label{setsL}
 \Theta^L(H,P,\lambda,J) = \{ \util = (a,u) \in \Theta(H,P,\lambda,J) : \inf a(\C) \geq -L \} \\
 \Lambda^L(H,P,\lambda,J) = \{ \util = (a,u) \in \Lambda(H,P,\lambda,J) : \inf a(\C) \geq -L \}.
\end{gather}
for a given $L>0$. When $(\lambda,J)$ are fixed we write $\Theta(H,P)$, $\Lambda(H,P)$, $\Theta^L(H,P)$ and
$\Lambda^L(H,P)$ for simplicity. Note that if $P$ is not simply covered then all these sets of fast planes are
empty. Our compactness result is as follows.

\begin{theorem}\label{main3}
Let $\lambda$ be a contact form on a closed $3$-manifold $M$, $P = (x,T)$ be a periodic Reeb orbit and $H \subset \R\times  M$ be compact. Suppose the following properties hold for every contractible periodic Reeb orbit $\hat P = (\hat x,\hat T)$ with $\hat T\leq T$:
\begin{itemize}
 \item[(i)] $\hat P$ is non-degenerate,
 \item[(ii)] $c_1(\xi)$ vanishes along $\hat P$ and $\mu_{CZ}(\hat P) \geq 3$.
\end{itemize}
Then the following assertions are true:
\begin{enumerate}
 \item $\Theta^L(H,P)$ and $\Lambda^L(H,P)$ are compact in $C^\infty_{loc}(\C,\R\times M)$.
 \item $\Theta(H,P)$ and $\Lambda(H,P)$ are compact in $C^\infty_{loc}(\C,\R\times M)$ if $H\cap (\R\times x(\R)) = \emptyset$.
\end{enumerate}
\end{theorem}

The above theorem can be rephrased in the terminology of {\bf Symplectic Field Theory} (SFT) originally introduced
by Eliashberg, Givental and Hofer in~\cite{sft}. A plane in $\Theta(H,P)$ is a stable connected smooth holomorphic
curve, in the sense of~\cite{sftcomp}. The {\bf SFT Compactness Theorem}~\cite{sftcomp} describes the
compactification of the set of such curves with \emph{a priori} bounds on energy and genus. It generalizes
the notion of Gromov convergence of pseudo-holomorphic curves in closed symplectic manifolds~\cite{gromov} to, for
example, the non-compact setting of symplectizations. The notion of a stable curve, adapted in~\cite{sftcomp} to
symplectic cobordisms, was first introduced by Kontsevich in~\cite{kontsevich}. More general than a stable connected
smooth holomorphic curve is a holomorphic building of height 1, which is a finite energy map defined on the
components of a (not necessarily stable) nodal Riemann surface, plus compatibility conditions. These are not enough
to compactify the set of smooth curves and one needs to introduce higher buildings.

In our situation, the possible limiting holomorphic buildings of a sequence of fast planes can be described as a
rooted graph, the vertices of which represent stable connected smooth holomorphic curves. An edge represents a
closed Reeb orbit which is a common limit of the curves at the corresponding vertices. This is always the case when
dealing with curves of genus 0 with one positive puncture. The edges can be oriented as going away from the root,
and this divides the graph into levels, these being the levels of the holomorphic building as described
in~\cite{sftcomp}. The proof of Theorem~\ref{main3} consists of showing this graph has exactly one vertex. In order
to accomplish this we only need to analyze the root, more precisely, we shall prove that the convexity assumptions
on $\lambda$ will discard outgoing edges.

As a final remark, assertions (1) and (2) of Theorem~\ref{main3} about the sets $\Lambda^L(H,P)$ and $\Lambda(H,P)$ follow independently from Theorem 4 of~\cite{wendl}, we explain. The elegant analysis of Wendl can be applied to \textbf{embedded} fast planes. In view of Theorem~\ref{lemmafredholm}, these are examples of ``nicely embedded'' curves in the sense of~\cite{wendl} which, in the symplectization $\R\times M$, are holomorphic curves with embedded projections onto $M$. With the appropriate asymptotic constraint on the orbit $P$, a fast plane has a vanishing ``constrained'' normal first Chern number, so that the limiting holomorphic building of a sequence of embedded fast planes is either smooth or one of its levels contains a plane with $\mu$-index equal to $2$. But these do not exist under our assumptions. However, the results of~\cite{wendl} do not cover the corresponding statements made in Theorem~\ref{main3} about the sets $\Theta^L(H,P)$ and $\Theta(H,P)$, even when $\lambda$ is dynamically convex.

A proof using the SFT-Compactness Theorem or the results of~\cite{wendl} directly would make the exposition non-elementary and highly not self-contained, forcing the introduction of a large amount of notation, and making this work less accessible to a wider public.

Also, we would like to emphasize that our arguments are independent of any transversality assumptions.

\subsection{Fredholm theory}\label{outlinefredholm}

The second analytical tool for proving Theorem~\ref{openbook} is a perturbation theory. Embedded fast planes are
always regular in a suitably defined index-$2$ Fredholm theory with exponential weights.

\begin{theorem}\label{lemmafredholm}
Let $\lambda$ be any contact form on a $3$-manifold $M$ and $\xi = \ker \lambda$ be the associated contact
structure. Fix any $J \in \mathcal J(\xi,d\lambda|_\xi)$
and suppose $\util = (a,u)$ is an embedded fast finite-energy $\jtil$-holomorphic plane asymptotic to a periodic Reeb orbit $P = (x_0,T_0)$ at $\infty$. If $\mu = \mu(\util) \geq 3$ then $u(\C) \cap x_0(\R) = \emptyset$ and $u : \C \rightarrow M \setminus x_0(\R)$ is a smooth proper embedding. Moreover, for any $l\geq1$ there exists an open ball $B_r(0)\subset\R^2$ and a $C^l$ embedding $f : \C \times B_r(0) \rightarrow \R\times M$ satisfying:
\begin{enumerate}
 \item $f(z,0)=\util(z)$.
 \item If $|\tau|<r$ then $f(\cdot,\tau)$ is an embedded fast finite-energy plane in $\R\times M$ asymptotic to
     $P$ satisfying $\mu(f(\cdot,\tau)) = \mu$.
 \item Fix $\tau_0 \in B_r(0)$ and let $\{\util_n\}$ be a sequence of embedded fast finite-energy planes
     asymptotic to $P$ satisfying $\util_n \to f(\cdot,\tau_0)$ in $C^\infty_{loc}$ and $\mu(\util_n) = \mu \
     \forall n$. Then there exist sequences $\tau_n \rightarrow \tau_0$, $A_n \rightarrow 1$ and $B_n
     \rightarrow 0$ such that $$ f(A_n z + B_n,\tau_n) = \util_n(z) \ \forall z \in \C, \ n\gg 1. $$ 
\end{enumerate}
\end{theorem}

In view of definitions~\ref{fastplane} and~\ref{behavior} the map $u:\C\to M$ provides a capping disk for $P_0$, that singles out a homotopy class $\beta_{\util} \in \mathcal{S}_{P_0}$ defined by the following property: a $d\lambda$-symplectic trivialization $\Psi$ of $\xi_{P_0}$ extends to $u^*\xi$ if, and only if, it is in class $\beta_{\util}$. In the above statement $\mu(\util) = \mu_{CZ}(P_0,\beta_{\util})$.

Note that $P_0$ is not assumed non-degenerate, instead we assume the planes have non-degenerate asymptotic behavior
in the sense of Definition~\ref{behavior}. This allows us to handle arbitrary contact forms on $S^3$,
see~\cite{hry}. We also emphasize that no assumptions on $\lambda$, like being Morse-Bott, are made. This justifies
Theorem~\ref{lemmafredholm}.

A degenerate Fredholm theory as described above was only hinted at in~\cite{convex}. The above statement does not
follow from the results of~\cite{props3} but, of course, its proof follows their arguments closely. We refer the
reader to Section~\ref{fredholm} for the proof.

\subsection{Existence of fast finite-energy planes}\label{existence}

We shall prove the following existence result of fast planes. It partially generalizes the existence statement in Theorem~\ref{chars3} since it deals with higher Conley-Zehnder indices. 

\begin{theorem}\label{existfast}
Let $\lambda$ be a tight contact form on a closed $3$-manifold $M$ such that the following properties hold for every
contractible Reeb orbit $\hat P$:
\begin{itemize}
 \item[(i)] $\hat P$ is non-degenerate,
 \item[(ii)] $c_1(\xi)$ vanishes along $\hat P$ and $\mu_{CZ}(\hat P) \geq 3$.
\end{itemize}
Suppose $P$ is an unknotted periodic Reeb orbit satisfying $sl(P,disk) = -1$. Then, for a suitable
$d\lambda$-compatible complex structure $J: \xi \rightarrow \xi$, there exists an embedded fast finite-energy
$\jtil$-holomorphic plane asymptotic to $P$ at $\infty$.
\end{theorem}

As before, the integer $sl(P,disk)$ denotes the self-linking number computed with respect to any embedded disk
spanning the unknot $x(\R)$. It is independent of this disk since $c_1(\xi)$ vanishes along $P$.

The general idea of the proof is standard, see~\cite{char2}. Since $P$ is unknotted and $sl(P)=-1$ we can find,
using arguments of Giroux \cite{giroux} and Hofer \cite{93}, an embedded disk $F \hookrightarrow M$ spanning $P$
(including orientations) such that its characteristic foliation has exactly one positive elliptic singularity $e \in
F$ with real eigenvalues. Denote $F^* = F \setminus\{e\}$. The surface $\{0\}\times F^* \hookrightarrow \R \times M$
is totally real with respect to $\jtil$ and there exists a so-called Bishop family emanating from $(0,e)$. It is a
one-dimensional family of unparametrized embedded $\jtil$-holomorphic disks with boundary on $\{0\} \times F^*$.
This family was discovered by E. Bishop in~\cite{bishop} and used by Hofer in~\cite{93} in order to establish the
Weinstein Conjecture in $S^3$ and in many other closed $3$-manifolds. It should be noted that disk-filling methods
were also used in~\cite{gromov} and in~\cite{filling} in order to show that symplectically fillable contact
structures are tight.

Each connected component of the Bishop family is an open interval. At one end the family converges to the constant
$(0,e)$. At the other end bubbling-off occurs and one can prove, using the convexity assumptions on $\lambda$, that
bubbling-off does not happen before the disks reach the boundary $\partial F = P$. As a result of this bubbling-off
analysis we have, in the language of Symplectic Field Theory (see~\cite{sftcomp}), a holomorphic building with
height $k\geq2$. Each level is a collection of smooth finite-energy surfaces in $\R \times M$. The first level is a
half trivial cylinder over $P$ and the curves on other levels have no boundary. All this is proved in~\cite{char2}.

Now we need to introduce new arguments. In~\cite{char2} the authors use the fact that the $\mu$-index of the orbit
in question is $3$. They heavily rely on the compactness argument explained in~\cite{props2} to conclude that the
second level of the stable curve consists of a single plane, hence there are no more levels. This is in great
contrast with our situation since we allow $\mu_{CZ}(P) \geq 3$. We overcome this difficulty by slightly perturbing
the boundary condition $F$ in order to ensure there are no Reeb tangencies near its boundary (of course the Reeb
vector is tangent at the boundary since it is a Reeb orbit). This allows us to reach the same conclusions as
in~\cite{char2}: $k=2$ and the second level is a single plane asymptotic to $P$. Also, this plane is embedded and
fast.

\subsection{Proofs of Theorems~\ref{thmA} and~\ref{openbook}}\label{sketchopenbook}

Necessity in Theorem~\ref{openbook} follows from Proposition~\ref{necessity}. We now turn to sufficiency. We
construct the desired open book decomposition as a consequence of compactness properties of families of fast planes.
Recall the families $\Lambda(H,P)$ in (\ref{setLambda}) and define
\begin{equation}\label{setLambdak}
 \Lambda_k(H,P) := \left\{ \vtil \in \Lambda(H,P) : \mu(\vtil) = k \right\}.
\end{equation}
One easily checks that $\Lambda_k(H,P)$ is $C^\infty_{loc}$-closed, $\forall k$. Thus, each $\Lambda_k(H,P)$ is
$C^\infty_{loc}$-compact whenever $\Lambda(H,P)$ is $C^\infty_{loc}$-compact. In Section~\ref{cons} we prove

\begin{theorem}\label{ob2}
Let $\lambda$ be a contact form on the closed $3$-manifold $M$ and let $\jtil$ be the almost complex
structure on $\R\times M$ defined by equations (\ref{almcpxstr}). Suppose there exists an embedded fast
$\jtil$-holomorphic finite-energy plane $\util_0$ asymptotic to $P=(x,T)$ with $\mu(\util_0) = k \geq 3$, and that every contractible orbit is non-degenerate. We also suppose that the set of planes $\Lambda_k(H,P)$ is $C^\infty_{loc}$-compact for every compact subset $H \subset
\R\times M$ satisfying $H \cap (\R\times x(\R)) = \emptyset$. Then for every $l\geq1$ there exists a $C^l$ map
$\util = (a,u) : S^1 \times \C \rightarrow \R \times M$ with the following properties:
\begin{enumerate}
 \item $\util(\vartheta,\cdot)$ is an embedded fast finite-energy plane asymptotic to $P$ at (the positive
     puncture) $\infty$ satisfying $\mu(\util(\vartheta,\cdot)) = k,\ \forall \vartheta \in S^1$.
 \item $u(\vartheta,\C) \cap x(\R) = \emptyset \ \forall \vartheta\in S^1$ and the map $u : S^1 \times \C
     \rightarrow M \setminus x(\R)$ is an orientation preserving $C^l$-diffeomorphism.
 \item Each $\cl{u(\vartheta,\C)}$ is a smooth global surface of section for the Reeb dynamics.
\end{enumerate}
\end{theorem}

The purpose of the above statement is to isolate the compactness properties of fast planes which allow us to
construct the desired open book decompositions. Here these compactness properties follow from convexity assumptions
on $\lambda$, see Theorem~\ref{main3}. In~\cite{pedro} we shall prove that the assumptions of Theorem~\ref{ob2} hold
under much less restrictive assumptions on $\lambda$, allowing us to investigate general Reeb flows on the tight
$3$-sphere.

Let us assume the hypotheses of Theorem~\ref{openbook}. If $P$ is an unknotted, simply covered, periodic Reeb orbit
satisfying $sl(P) = -1$ then Theorem~\ref{existfast} provides an embedded fast finite-energy plane $\util_0$
asymptotic to $P$ at $\infty$. Theorem~\ref{main3} now shows that the hypotheses of Theorem~\ref{ob2} hold.
Theorem~\ref{openbook} follows immediately.

Before proving Theorem~\ref{thmA} we briefly outline the proof of Theorem~\ref{ob2} for convenience of the reader.
Suppose $M$, $\lambda$ and $\util_0$ satisfy the hypotheses of Theorem~\ref{ob2}. If we write $\util_0 = (a_0,u_0)
\in \R \times M$ then it follows from lemmas~\ref{int2}, \ref{int3} and~\ref{int4} below that $u_0$ is a proper
embedding into $M \setminus x(\R)$. The identity $\wind_\infty(\util_0) = 1$ proves $u_0(\C)$ is transverse to the
Reeb vector field. By Theorem~\ref{lemmafredholm}, $\util_0$ is only one embedded fast plane in a small
$2$-parameter family. Let $p_0 = u_0(0) \in M$ and assume, without loss of generality, that $a_0(0) = 0$. Denoting
by $\phi_t$ the Reeb flow, we can single out a one-dimensional subfamily $\{\util^t = (a^t,u^t)\}$ by requiring
\[
 \begin{array}{cccc}
   u^t(0) = \phi_t(p_0), & a^t(0) = 0 & \text{ and } & \util^0 = \util_0.
 \end{array}
\]
The family of embedded planes $\{u^t(\C)\}$ inside $M \setminus x(\R)$ provides a smooth foliation of a neighborhood
of $u_0(\C)$. Using the compactness assumptions we continue the family $\util_t$ for all values $t\in\R$, satisfying
the above normalization conditions. By Poincar\'e recurrence we could have assumed, without loss of generality, that
$p_0 \in \omega\text{-limit}(p_0)$. Then the trajectory $\phi_t(p_0)$ will eventually come close to $p_0$. The
completeness statement in Theorem~\ref{lemmafredholm} can be used to show that the family $\util^t$ can be glued to
provide a $S^1$-family. It foliates the whole of $M\setminus x(\R)$. This $S^1$-family can be made minimal if we
require $$ (t,z) \in S^1 \times \C \mapsto u^t(z) \in M \setminus x(\R) $$ is a diffeomorphism. This provides an
open book decomposition with disk-like pages that are transverse to the Reeb field. To prove the pages are global
surfaces of section, fix $q_0 \in M \setminus x(\R)$. If $x(\R) \cap \omega\text{-limit}(q_0) = \emptyset$ then
$\phi_t(q_0)$ hits every page in forward time, by an easy compactness argument. If $x(\R) \cap
\omega\text{-limit}(q_0) \not= \emptyset$ then the condition $\mu_{CZ}(P) \geq 3$ makes the flow wind around $x(\R)$
for long enough times, forcing it to hit every page. This is proved in Section 5 of~\cite{convex}, see
Lemma~\ref{omegalimit} below. The argument is the same for negative times. This concludes the proof of Theorem~\ref{ob2}.

Theorem~\ref{thmA} follows easily from Theorem~\ref{openbook} and from the following result from~\cite{convex}.

\begin{theorem}[Hofer, Wysocki and Zehnder]
If $S\subset \R^4$ is the boundary of a bounded, smooth, strictly convex domain containing $0$ then $\lambda_0|_S$
is dynamically convex, that is, $\mu_{CZ}(P) \geq 3$ for every periodic orbit of the Reeb vector field associated to
the contact form $\lambda_0|_S$.
\end{theorem}

The arguments are immediate in view of a famous result of Bennequin asserting that $\xi_0 = \ker \lambda_0|_S
\subset TS$ is a tight contact structure.

\section{Basic definitions and facts}\label{basicdefns}

Unless otherwise stated, $M$ denotes a closed $3$-manifold, $\lambda \in \Omega^1(M)$ is a contact form and $\xi =
\ker \lambda$ is the induced contact structure.

\subsection{Periodic Reeb orbits and winding numbers}\label{orbits}

We shall identify a periodic Reeb orbit $P=(x,T)$ with the class in $C^\infty(S^1,M)/S^1$ of the loop $$ t\in \R/\Z
\simeq S^1 \mapsto x_T(t) := x(Tt). $$ Here we let $S^1$ act on the loop space by rotations on the domain. Hence, we
view the collection $\p$ of periodic Reeb orbits as a subset of $C^\infty(S^1,M)/S^1$. The geometrical image of
$P=(x,T)\in\p$ is the set $x(\R)$ and $P'=(x',T')$ is geometrically distinct of $P$ if $x(\R) \cap x'(\R) =
\emptyset$. We shall agree with the following convention: for every periodic Reeb orbit we select a point in its
geometrical image, and it will be implicit from the notation $P=(x,T)$ that $x(0)$ is the chosen point.

\begin{defn}
Consider a contractible periodic orbit $P=(x,T)$ and two continuous disk-maps $f_1,f_2:\D\rightarrow M$ spanning
$x_T$, that is, $f_j(e^{i2\pi t}) = x(Tt)$, $j=1,2$. We can define the map $f_1\#\bar f_2 : S^2 =
\C\sqcup\{\infty\}\to M$ by
\[
 z\mapsto \left\{ \begin{aligned} & f_1(z) \text{ if } |z|\leq 1 \\ & f_2(1/\bar z) \text{ if } 1\leq |z| < \infty \\ & f_2(0) \text{ if } z=\infty \end{aligned} \right.
\]
We say that $c_1(\xi)$ vanishes along $P$ if $c_1((f_1\#\bar f_2)^*\xi)=0$ for every pair $f_1,f_2$ as above. We
denote by $\p^*$ the set of contractible orbits with this property.
\end{defn}

Clearly all contractible periodic Reeb orbits belong to $\p^*$ if $c_1(\xi)$ vanishes.

\begin{notation}[Winding Numbers]
Let $(E,J)$ be a complex line bundle over $S^1$. If $A$ and $B$ are two non-vanishing sections then $A=fB$ for
unique $f:S^1 \rightarrow\C \setminus \{0\}$, identifying $i$ with $J$. We denote $\wind(A,B,J) = \deg f/|f| \in
\Z$. More loosely, we write $\wind(f)=\wind(f,1,i)=\deg f/|f|$ if $f:S^1 \rightarrow\C \setminus \{0\}$ is
continuous. This winding count does not depend on the homotopy classes (of non-vanishing sections) of $A$ or $B$,
nor does it depend on the homotopy class (of complex multiplications) of $J$. If we fix classes $\alpha$ and $\beta$
and sections $A\in\alpha$ and $B\in\beta$ then expressions like $\wind(\alpha,\beta,J)$ or $\wind(A,\beta,J)$ have
obvious meanings.
\end{notation}

\begin{remark}\label{abc}
Given $P\in\p$, recall the set $\mathcal{S}_P$ of homotopy classes of $d\lambda$-symplectic trivializations of
$\xi_P$. One can identify $\mathcal{S}_P$ with the set of homotopy classes of non-vanishing sections of $\xi_P$ in a
straightforward way. In the following we shall always assume this is done. Consequently, if $P=(x,T)\in \p^*$ then a
section $Z$ is in the special class $\beta_P$ discussed in the introduction if, and only if, for some (and hence
any) continuous map $f:\D\to M$ satisfying $f(e^{i2\pi t}) = x(Tt)$ the section $Z$ extends to a non-vanishing
section of $f^*\xi$.
\end{remark}

The following lemma, which is a trivial consequence of standard degree theory, will be stated without proof.

\begin{lemma}\label{localclass}
Suppose $P = (x,T) \in \p^*$ and $U$ is a small tubular neighborhood of $x(\R)$ in $M$. Let $Z$ be a non-vanishing
section of $\xi|_U$ such that $x_T^*Z \in \beta_P$. If $f:\D\rightarrow M$ is a continuous map such that $f(\partial
\D) \subset U$ and $t\mapsto f(e^{i2\pi t})$ is homotopic to $x_T$ in $U$ then $(f|_{\partial \D})^*Z$ extends to a
non-vanishing section of $f^*\xi$.
\end{lemma}

\subsection{Special coordinates}\label{specialcoord}

Consider a periodic Reeb orbit $P = (x_0,T_0) \in \p$ with minimal period $0<T_{min}\leq T_0$. Set $k := T_0/T_{min}
\in \Z^+$.

\begin{defn}[Martinet Tube]\label{martinettube}
Let $\R/\Z \times \R^2$ be equipped with coordinates $(\theta,x,y)$ and set $\lambda_0 := d\theta + xdy$. A Martinet
Tube around $P$ is an open neighborhood $U$ of $x_0(\R)$ and a
diffeomorphism
\begin{equation}\label{phi}
 \Psi : U \rightarrow \R/\Z \times B
\end{equation}
where $B\subset\R^2$ is an open ball centered at $0$, satisfying the following properties:
\begin{enumerate}
 \item $\Psi_*\lambda = f\lambda_0$ where $f|_{\R/\Z \times \{0\}} \equiv T_{min}$ and $df|_{\R/\Z \times \{0\}}
     \equiv 0$.
 \item $\Psi( x_0(T_{min} t) ) = (t,0,0) \ \forall t\in \R$.
\end{enumerate}
\end{defn}

\begin{remark} \label{homtriv}
There always exists a Martinet tube around any $P$, as noted in~\cite{props1}. The bundle $\xi|_{U}$ is framed by
$\partial_x$ and $-x\partial_\theta + \partial_y$. Setting $e_1 = f^{-1/2}\partial_x$ and $e_2 =
f^{-1/2}(-x\partial_\theta + \partial_y)$ then $\{e_1,e_2\}$ is $d\lambda$-symplectic. The homotopy class (of
non-vanishing sections of $\xi|_{P_{min}}$) induced by $t\in\R/\Z \mapsto \partial_x|_{(t,0)}$ can be arbitrarily
chosen. Note that $\Psi^{-1}(kt,0,0) = {x_0}_{T_0}(t) \ \forall t\in S^1$.
\end{remark}

\subsection{Dynamical Convexity}\label{cz}

Here we shall modify slightly an important definition from~\cite{convex}.

\begin{defn}
A contact form $\lambda$ is dynamically convex if $\mu_{CZ}(P)\geq3 \ \forall P\in \p^*$.
\end{defn}

\subsection{Tight contact structures}\label{tightstr}

The contact structure $\xi$ is said to be {\it tight} if there are no {\it overtwisted disks} in $M$. An embedded disk $F \subset M$ is {\it overtwisted} if $\partial F$ is a Legendrian knot and $T_xF \not= \xi|_x$, $\forall x\in \partial F$. 

\subsection{Finite-energy surfaces in symplectizations}\label{fesurfaces}

In 1985 pseudo-holomorphic curves were introduced in symplectic geometry by M. Gromov~\cite{gromov}. In 1993 they
were used by H. Hofer to study Reeb flows on contact manifolds. Let $(\Sigma,j)$ be a Riemann surface, possibly with
non-empty boundary and not necessarily compact, and $\Gamma \subset \Sigma \setminus \partial \Sigma$ be a finite
subset. The notion of finite-energy surfaces was introduced by H. Hofer in~\cite{93}.

\begin{defn}{(Finite-energy surfaces)}\label{defsymp}
A map $\tilde{u}:(\Sigma\setminus\Gamma,j)\rightarrow(\R\times {M},\tilde{J})$ is called a finite energy surface if it is pseudo-holomorphic, that is, it satisfies the non-linear Cauchy-Riemann
equations
\begin{equation}\label{cr}
d\tilde{u}\circ j=\tilde{J}\circ d\tilde{u}
\end{equation}
and also the energy condition $0<E(\tilde{u})<+\infty$. The energy $E(\tilde{u})$ is defined as follows. Set
$\Lambda := \{\phi\in C^\infty(\R,[0,1]):\phi^{\prime}\geq0\}$ and $\omega_\phi=d\lambda_\phi$ where
$\lambda_\phi\in\Omega^1(\R\times {M})$ is given by $\lambda_\phi(a,p)=\phi(a)\lambda(p)$. Finally define
\[
E(\tilde{u})=\sup_{\phi\in\Lambda}\int_{\Sigma\setminus\Gamma}\tilde{u}^*\omega_\phi.
\]
It follows from (\ref{cr}) that each integral above is non-negative. When $\Sigma = S^2$ and $\#\Gamma=1$ we call
$\util$ a finite-energy plane.
\end{defn}

Let us write $\util = (a,u) \in \R \times M$. The points of $\Gamma$ are called punctures. Let us fix a puncture
$z\in\Gamma$ and let $\varphi : (U,0) \rightarrow (\varphi(U),z)$ be a holomorphic chart of $(\Sigma,j)$ centered at
$z$. Write $\util (s,t) = \util \circ \varphi \left( e^{-2\pi(s+it)} \right)$. It follows easily from $E(\util) <
\infty$ that the limit
\begin{equation}\label{naturepuncture}
 m = \lim_{s\rightarrow+\infty} \int_{\{s\}\times S^1} u^*\lambda
\end{equation}
exists. The puncture $z$ is removable if $m=0$, positive if $m>0$ and negative if $m<0$. A removable singularity can
actually be removed, meaning that $\tilde{u}$ can be smoothly continued across the singularity, see~\cite{93}. If
$\Sigma$ is closed then a finite-energy surface must have non-removable punctures because the forms $\omega_\phi$
are exact.

Finite-energy surfaces are closely related to periodic Reeb orbits. This is the content of the following fundamental
result from~\cite{93}.

\begin{theorem}[H. Hofer]\label{thm93}
In the notation explained above, suppose $z$ is non-removable and let $\epsilon=\pm1$ be the sign of $m$ in
(\ref{naturepuncture}). Then every sequence $s_n \rightarrow +\infty$ has a subsequence $s_{n_k}$ such that the
following holds: there exists a real number $c$ and a periodic Reeb orbit $P = (x,T)$ such that $u(s_{n_k},t)
\rightarrow x(\epsilon Tt+c)$ in $C^\infty(S^1,M)$ as $k \rightarrow +\infty$.
\end{theorem}

In his seminal work~\cite{93} H. Hofer is able to partially solve the three-dimensional Weinstein conjecture using
techniques of pseudo-holomorphic curves.

\begin{remark}\label{realaction}
$\R \times M$ carries a $\R$-action given by translating the first coordinate and $\jtil$ is $\R$-invariant. If
$\util = (a,u)$ is a finite-energy surface then so is $c \cdot \util := (a+c,u)$.
\end{remark}

\subsection{Asymptotic behavior near the punctures}\label{nondegasympbehavior}

Let $\pi : TM\rightarrow \xi$ denote the projection along the Reeb direction.

\begin{defn}\label{behavior}
Let $(S,j)$, $\Gamma$ and $\util$ be as in Definition~\ref{defsymp}. Fix a non-removable puncture $z \in \Gamma$,
choose a holomorphic chart $\varphi : (U,0) \rightarrow (\varphi(U),z)$ centered at $z$ and write $\util (s,t) =
(a(s,t),u(s,t)) = \util \circ \varphi ( e^{-2\pi(s+it)} )$ for $s\gg1$. Define $m$ by (\ref{naturepuncture}) and let
$\epsilon = \pm1$ be its sign. We say that $z$ is a \textbf{non-degenerate puncture of $\util$} if there exists a
periodic Reeb orbit $P = (x,T)$ and constants $c,d\in\R$ such that
\begin{enumerate}
 \item $\sup_{t\in S^1} \norma{a(s,t) - \epsilon Ts - d} \rightarrow 0$ as $s\rightarrow +\infty$.
 \item $u(s,t) \rightarrow x(\epsilon Tt + c)$ in $C^0(S^1,M)$ as $s\rightarrow +\infty$.
 \item If $\pi \cdot du$ does not vanish identically over $S\setminus\Gamma$ then $\pi \cdot du (s,t) \not=0$
     when $s$ is large enough.
 \item If we define $\zeta(s,t)$ by $u(s,t) = \exp_{x(\epsilon Tt+c)} \zeta(s,t)$ then $\exists b>0$ such that
     $\sup_{t\in S^1} e^{bs}\norma{\zeta(s,t)} \rightarrow 0$ as $s\rightarrow +\infty$.
\end{enumerate}
In this case we say $\util$ is asymptotic to $P$ at $z$. The puncture $z$ is positive or negative according to the
sign $\epsilon$. This definition is independent of $\varphi$ and of the exponential map $\exp$.
\end{defn}

\begin{defn}\label{nondegpunc}
We say that $\util$ has non-degenerate asymptotic behavior at $z$ if $z$ is a non-degenerate puncture, and that $\util$ has non-degenerate asymptotics if this holds for every puncture.
\end{defn}

The behavior of $\util$ near a puncture is studied in~\cite{props1}. Here is a partial result.

\begin{theorem}[Hofer, Wysocki and Zehnder]\label{p1}
Let $\util$, $z$ and $\varphi$ be as in Theorem~\ref{thm93}. If a Reeb orbit $P$ obtained by
Theorem~\ref{thm93} is non-degenerate then $\util$ has non-degenerate asymptotic behavior at $z$. In particular,
$\util$ is asymptotic to $P$ at $z$.
\end{theorem}

\subsection{Algebraic invariants}\label{alginvs}

In~\cite{props2} a number of algebraic invariants of finite-energy surfaces were introduced. In the next two definitions we fix a finite-energy plane $\util : \C = S^2\setminus\{\infty\} \rightarrow \R\times M$ with non-degenerate asymptotics and consider its asymptotic limit $P=(x,T)$ at $\infty$.

\begin{defn}[Covering Number]
We define $\cov(\util) := T/T_{min} \in \Z^+$ where
$T_{min}>0$ is the minimal period of $x$.
\end{defn}

\begin{defn}[$\mu$-index for planes]
Writing $\util = (a,u)$ then $u$ provides a capping disk for $P$ and induces a class $\beta_{\util} \in \mathcal{S}_P$. Define $\mu(\util) = \mu_{CZ}(P,\beta_{\util})$.
\end{defn}

Let $\util$ and $\vtil$ be finite-energy planes with the same asymptotic limit $P$. The identity $\mu(\util) =
\mu(\vtil) + 2\wind(\beta_{\vtil},\beta_{\util},J)$ proves that $\beta_{\util} = \beta_{\vtil} \Leftrightarrow
\mu(\util) = \mu(\vtil)$.

\begin{remark}\label{locallycr}
If $\util = (a,u)$ defined on $\left(S,j\right)$ is a $\jtil$-holomorphic map then $\pi\cdot du$ satisfies the
``perturbed'' Cauchy-Riemann equation $\pi\cdot du \circ j = \jtil \circ \pi\cdot du$. By the similarity principle,
see~\cite{mcdsal}, $\pi \cdot du \equiv 0$ on connected components of $S$ where the zero set of $\pi \cdot du$ has a
limit point.
\end{remark}

For the next two definitions we fix a closed Riemann surface $(\Sigma,j)$, a finite set $\Gamma \subset \Sigma$ and
a finite-energy surface $\util = (a,u) : (\Sigma\setminus\Gamma,j) \rightarrow (\R\times M,\jtil)$ with
non-degenerate asymptotics. We assume that $\pi\cdot du$ does not vanish identically and that $\Gamma$ consists of
non-removable punctures.

\begin{defn}[$\wind_\infty$]
Split $\Gamma = \Gamma^+ \sqcup \Gamma^-$ where $\Gamma^+$ is the set of positive punctures and $\Gamma^-$ is the
set of negative punctures. The bundle $u^*\xi$ is trivializable since $\Gamma\not=\emptyset$. Let $\alpha$ be a
homotopy class of non-vanishing sections of $u^*\xi$. Choose a non-vanishing section $\sigma$ in class $\alpha$ and
write $\util(s,t)$ around a puncture $z\in\Gamma$ as in Definition~\ref{behavior}. If $\epsilon=\pm1$ is the sign of
$z$ then define
\[
 \text{wind}_\infty(\util,\alpha,z) = \lim_{s\rightarrow+\infty} \text{wind}(t\mapsto \pi u_s(s,\epsilon t),t\mapsto \sigma(s,\epsilon t),J)
\]
where $\sigma(s,t)=\sigma(\varphi(e^{-2\pi(s+it)}))$. The invariant wind$_\infty(\util)$ is defined in~\cite{props2}
by
\[
 \text{wind}_\infty(\util)=\sum_{z\in\Gamma^+}\text{wind}_\infty(\util,\alpha,z)-\sum_{z\in\Gamma^-}\text{wind}_\infty(\util,\alpha,z).
\]
Each $\text{wind}_\infty(\util,\alpha,z)$ depends on the choice of $\alpha$ but is independent of $\varphi$. By standard degree theory $\text{wind}_\infty(\util)$ is independent of $\alpha$.
\end{defn}

\begin{defn}[$\wind_\pi$]
The bundle $\mathcal{E} = Hom_\C (T(\Sigma \setminus \Gamma), u^*\xi)$ is a complex line bundle and the section
$\pi\cdot du$ satisfies a perturbed Cauchy-Riemann equation, see Remark~\ref{locallycr}. Thus its zeros are isolated
and count positively when computing the intersection number with the zero section of $\mathcal{E}$. As a consequence
of Definition~\ref{behavior} the number of zeros is finite. Following~\cite{props2} we define
\begin{equation}
 \text{wind}_\pi(\util) = \text{algebraic count of zeros of }\pi\cdot du
\end{equation}
where the zeros are counted with multiplicities. The inequality $\wind_\pi(\util)\geq 0$ can be seen as a linearized
version of positivity of self-intersections.
\end{defn}

The Gauss-Bonnet formula proves the following lemma, as shown in~\cite{props2}.

\begin{lemma}\label{gauss}
$\text{wind}_\infty(\util)=\text{wind}_\pi(\util)-\#\Gamma+\chi(\Sigma)$.
\end{lemma}

\begin{remark}
$\pi\cdot du$ does not vanish identically if $\util$ is a finite-energy plane with non-degenerate asymptotics, this
follows from results of~\cite{props2} (see Lemma~\ref{zerodlambda} below). In this case $\wind_\infty(\util)\geq1$.
\end{remark}


\section{Compactness}\label{compactness}

This section is devoted to the proof of Theorem~\ref{main3}. We fix the Riemannian metric $g^0$ on $\R\times M$
given by
\begin{equation}\label{norm0}
 g^0 = da\otimes da + \lambda \otimes \lambda + d\lambda(\cdot,J\cdot)
\end{equation}
where $J:\xi\to\xi$ is $d\lambda$-compatible. All norms of maps or objects in $\R\times M$ are taken with respect to
the metric $g^0$.


\subsection{Asymptotic operators and their spectral properties}\label{asympopprops}

We endow $\R^2$ with its standard euclidean structure $\left< \cdot,\cdot \right>$, inducing a Hilbert space
structure on $L^2(S^1,\R^2)$. If $\varphi : [0,1] \rightarrow \Sp(1)$ is a smooth path and
$S:=-J_0\varphi^\prime\varphi^{-1}$ then $S^T=S$. We identify $S^1 = \R/\Z$ and consider the unbounded self-adjoint
operator
\[
\begin{array}{ccc}
  L_S : W^{1,2} \subset L^2 \rightarrow L^2, &  & L_S(e) = -J_0 \dot e - Se.
\end{array}
\]
$L_S$ has compact resolvent and discrete real spectrum $\sigma(L_S)$ accumulating only at $\pm\infty$. Each point of
the spectrum is an eigenvalue with the same (finite) algebraic and geometric multiplicities, see~\cite{kato}. This
is so because $L_S$ is homotopic to $-J_0\partial_t$ through compact symmetric perturbations. See~\cite{props2} for
more details.

If $\varphi(0) = I$ then $\varphi \in \Sigma^*$ if, and only if, $0\not\in\sigma(L_S)$. For any $\delta\in\R$ denote
by $\nu^{neg}_\delta<\delta$ and $\nu^{pos}_\delta>\delta$ the special eigenvalues
\[
  \begin{array}{ccc}
    \nu^{pos}_\delta = \min \{ \nu \in \sigma(L_S) :  \nu > \delta \}, &  & \nu^{neg}_\delta = \max \{ \nu \in \sigma(L_S) :  \nu < \delta \}.
  \end{array}
\]
To each non-zero eigenvector $e$ of $L_S$ one can consider the winding number $\wind(e)$. It is proved
in~\cite{props2} that:
\begin{enumerate}
 \item If $e_1,e_2$ are two non-zero eigenvectors of $L_S$ such that $L_Se_j=\nu e_j$ for $j=1,2$ then
     $\wind(e_1)=\wind(e_2)$.
 \item If $L_Se_j=\nu_je_j$ for $j=1,2$ and $\nu_1\leq\nu_2$ then $\wind(e_1)\leq\wind(e_2)$.
\end{enumerate}
Thus one has a well-defined winding $\wind(\nu)$ associated to an eigenvalue $\nu$ of $L_S$ satisfying
$\nu_1\leq\nu_2 \Rightarrow \wind(\nu_1)\leq\wind(\nu_2)$. It is also proved in~\cite{props2} that for every
$k\in\Z$ there are exactly two eigenvalues (counting multiplicities) with winding equal to $k$.
Following~\cite{props2}, we have a well-defined integer
\begin{equation}\label{genmuindex}
 \tilde \mu^\delta\left(L_S\right) = 2\wind\left(\nu^{neg}_\delta\right) + \frac{1}{2} \left( 1 + (-1)^{b_\delta} \right)
\end{equation}
where $b_\delta$ is the number of eigenvalues $\nu < \delta$ such that $\wind(\nu) = \wind(\nu^{neg}_\delta)$,
counting multiplicities.

\begin{lemma}[Hofer, Wysocki and Zehnder]
$\tilde \mu^0 : \varphi \in \Sigma^* \mapsto \tilde \mu^0 (\varphi) \in \Z$ satisfies the axioms of
Theorem~\ref{axiomscz}.
\end{lemma}

\begin{remark}
The above lemma provides an extension of the index to symplectic paths $\varphi$ that are not in $\Sigma^*$. By the
spectral properties of $L_S$, if $\delta$ is not an eigenvalue then the term $\frac{1}{2} \left( 1 + (-1)^{b_\delta}
\right)$ is equal to $\wind\left(\nu^{pos}_\delta\right) - \wind\left(\nu^{neg}_\delta\right)$.
\end{remark}

\begin{defn}[Asymptotic operators]\label{defnasympop}
Fix a $d\lambda$-compatible complex structure $J$ on $\xi$ and let $P=(x,T)$ be a periodic Reeb orbit. Then
the metric $d\lambda(\cdot,J\cdot)$ induces a Hilbert space structure on $L^2\left(\xi|_P\right)$. Choose a
symmetric connection $\nabla$ on $TM$. The unbounded self-adjoint operator
\begin{gather*}
 A_P : W^{1,2}\left(\xi|_P\right) \subset L^2\left(\xi|_P\right) \rightarrow L^2\left(\xi|_P\right) \\
 \eta \mapsto -J\nabla_t\eta + TJ\nabla_\eta R
\end{gather*}
is independent of $\nabla$ ($\nabla_t$ denotes the covariant derivative along the curve $t\mapsto x_T(t)$). $A_P$ is
the so-called asymptotic operator at $P$.
\end{defn}

The linear flow generated by $\nabla_t\eta = T\nabla_\eta R$ is $d\phi_{Tt}|_{x(0)}$ restricted to $\xi$. Choose a
$d\lambda$-symplectic frame $\sigma = \{e_1,e_2\}$ for $\xi|_P$, represent the linear maps $d\phi_{Tt}|_{x(0)}$ by a
smooth path $\varphi:[0,1] \rightarrow \Sp(1)$, $\varphi(0) = I$, and the multiplication $x_T^*J$ by a smooth path
$t \mapsto J(t)$. Then $J(t) \in \Sp(1)$ and $-J_0J(t)$ is a positive symmetric matrix. The matrix $S =
-J(t)\varphi^\prime\varphi^{-1}$ is symmetric with respect to the inner-product $\left< \cdot,-J_0J(t)\cdot \right>$
on $\R^2$ and $1$-periodic (since so is $\varphi^\prime\varphi^{-1}$). The operator
\begin{equation}\label{localrepasympop}
 L_S : e(t) \mapsto -J(t) e^\prime(t) - S(t)e(t)
\end{equation}
represents $A_P$ in the frame $\sigma$. If $\sigma$ is $(d\lambda,J)$-unitary ($d\lambda(e_1,e_2) \equiv 1$ and
$Je_1=e_2$), then $A_P$ is presented as $L_S = -J_0\partial_t - S(t)$ with $S^T=S$. Thus $A_P$ has all the spectral
properties explained before.

\begin{notation}
With respect to a symplectic frame $\sigma = \{e_1,e_2\}$ for $\xi|_P$ the eigenvectors and eigenvalues of $A_P$
have well-defined winding numbers. These, of course, depend on the homotopy class $\beta \in \mathcal{S}_P$ of the
section $t \mapsto e_1(t)$ and will be denoted by $(\nu,\beta) \in \Z$. They are comparable via the formula
$(\nu,\beta_1) = (\nu,\beta_0) + \wind(\beta_0,\beta_1,J)$. For any $\delta\in\R$ we define
\[
 \mu_{CZ}^\delta(P,\beta) = 2\left(\nu^{neg}_\delta,\beta\right) + \frac{1}{2} \left( 1 + (-1)^{b_\delta} \right)
\]
where $b_\delta$ is the number of eigenvalues $\nu < \delta$ such that $(\nu,\beta) = (\nu^{neg}_\delta,\beta)$,
counting multiplicities. If $P$ is non-degenerate then
\[
 \mu_{CZ}(P,\beta) = 2\left(\nu^{neg}_0,\beta\right) + (\nu^{pos}_0,\beta) - (\nu^{neg}_0,\beta).
\]
\end{notation}

The following lemma is an easy consequence of the definitions.

\begin{lemma}\label{dynconvex}
Suppose $P\in\p^*$ satisfies $\mu_{CZ}(P) \geq 3$. Then $(\nu,\beta_P)\geq2$ for every eigenvalue $\nu\geq0$ of
$A_P$.
\end{lemma}

\subsection{Finite-energy spheres with vanishing $d\lambda$-energy}

\begin{lemma}[Hofer, Wysocki and Zehnder]\label{zerodlambda}
Suppose $\vtil=(d,v):\C\setminus\Gamma \rightarrow \R\times M$ is a finite-energy sphere satisfying $\pi\cdot dv
\equiv 0$. Suppose further that $\infty$ is its unique positive puncture. There exists a non-constant polynomial
$p:\C\rightarrow\C$ and a periodic orbit $\hat P=(\hat x,\hat T)$ such that $p^{-1}(0)=\Gamma$ and $\vtil=f_{(\hat
x,\hat T)}\circ p$ where
\[
  \begin{array}{ccc}
    f_{(\hat x,\hat T)}:\C\setminus\{0\}\rightarrow\R\times M, &  & f_{(\hat x,\hat T)}(\est) = (\hat Ts,x(\hat Tt)).
  \end{array}
\]
\end{lemma}


\subsection{Bubbling-off points}

The basic tool for the bubbling-off analysis is the following lemma. In the statement below norms are taken with
respect to $g^0$ (\ref{norm0}) and the euclidean metric on $\C$.

\begin{lemma}\label{beforeclaim}
Let $\Gamma \subset \C$ be finite and $U_n \subset \C\setminus\Gamma$ be an increasing sequence of open sets such
that $\cup_n U_n = \C\setminus\Gamma$. Let $\util_n = (a_n,u_n) : (U_n,i) \rightarrow (\R \times M,\jtil)$ be a
sequence of $\jtil$-holomorphic maps satisfying $\sup_n E(\util_n) = C < \infty$, and $z_n \in U_n$ be a sequence
such that $\norma{d\util_n(z_n)} \rightarrow +\infty$. If $z_n$ stays bounded away from $\Gamma \sqcup \{\infty\}$,
or if there exist $m$ and $\rho>0$ such that $\C \setminus B_\rho(0) \subset U_m$ and $z_n$ stays bounded away from
$\Gamma$, then the following holds: $\forall0<s<1$ there exist subsequences $\{\util_{n_j}\}$ and $\{z_{n_j}\}$,
sequences $z^\prime_j \in \C$ and $r_j(s) \in\R$, and a contractible periodic Reeb orbit $\hat P = (\hat x , \hat
T)$ such that $\norma{z_{n_j}-z_j^\prime} \rightarrow 0$, $r_j(s) \rightarrow 0^+$, $\hat T \leq C$ and
\[
 \limsup_{j\rightarrow+\infty} \int_{\norma{z-z_j^\prime}\leq r_j(s)} u_{n_j}^*d\lambda \geq s\hat T.
\]
\end{lemma}

We do not include a proof here since it is standard.

\begin{cor}\label{claim}
Assume $\lambda$ and $P$ satisfy the hypotheses of Theorem~\ref{main3}. Suppose
$\{\util_n=(a_n,u_n)\}\subset\Theta(H,P)$ and $\{z_n\}\subset\C$ are sequences such that
$\norma{d\util_n(z_n)}\rightarrow+\infty$. Then $\limsup|z_n|\leq1$ and for any
$1<s<\gamma^{-1}\min\{\gamma_1,\gamma_2\}$ there exist subsequences $\{\util_{n_j}\}$ and $\{z_{n_j}\}$, a sequence
$r_j(s)\rightarrow0^+$ and a sequence $z_j^\prime$ such that $\norma{z_{n_j}-z_j^\prime} \rightarrow 0$ and
\[
 \limsup_{j\rightarrow+\infty} \int_{\norma{z-z_j^\prime}\leq r_j(s)}u_{n_j}^*d\lambda \geq s\gamma.
\]
\end{cor}

\begin{cor}
Assume $\lambda$ and $P$ satisfy the hypotheses of Theorem~\ref{main3}. If $\{\util_n\}\subset\Theta(H,P)$ and
$z^1,\dots,z^J$ are distinct points of $\C$ satisfying
\[
 \forall \ 1\leq l\leq J \ \exists \{z^l_n\} \text{ such that } z^l_n\rightarrow z^l \text{ and } |d\util_n(z^l_n)|\rightarrow+\infty
\]
then $\{z^1,\dots,z^J\}\subset\D$ and $J\leq T/\gamma$.
\end{cor}

\begin{proof}
Write $\util_n=(a_n,u_n)$. The conclusion follows easily from the previous lemma since $\int_\C u_n^*d\lambda = T$
for all $n$.
\end{proof}

\begin{cor}\label{corclaim}
If $\{\util_n\}\subset\Theta(H,P)$ then one can find a subsequence $\{\util_{n_j}\}$ and a finite set
$\Gamma\subset\D$ such that $\{|d\util_{n_j}|\}$ is uniformly bounded on compact subsets of $\C\setminus\Gamma$.
\end{cor}

\subsection{Special cylinders with small energy}\label{specialcylinderssmallenergy}

Fix any contact form $\lambda$ on the closed $3$-manifold $M$. Suppose $Q \subset \p$ satisfies the following
condition: if $\hat P = (\hat x,\hat T) \in Q$, $\tilde P = (\tilde x,\tilde T) \in \p$ and ${\hat x}_{\hat T}$ is
$C^0$-homotopic to ${\tilde x}_{\tilde T}$ then $\tilde P \in Q$. We denote $Q^C = \{ \hat P = (\hat x,\hat T) \in Q
: \hat T \leq C \}$ for some fixed $C>0$. We assume every $P \in Q^C$ is non-degenerate, so that $Q^C$ is finite. We follow~\cite{fols} and fix a number
\begin{equation}\label{gammaagain}
 \begin{aligned}
  &0 < e < \min\{a_1,a_2\}, \text{ where } a_1 = \min \{\hat T : \hat P=(\hat x,\hat T) \in Q^C \} \text{ and} \\
  &a_2 = \min \{|\hat T - \tilde T| : \hat P=(\hat x,\hat T),\tilde P=(\tilde x,\tilde T) \in Q^C \text{ and } \hat T \not= \tilde T \}.
 \end{aligned}
\end{equation}
Choose also an arbitrary $S^1$-invariant open neighborhood $\W$ of $Q^C$ in the loop space $C^\infty(S^1,M)$. The
proof of the following lemma is found all over the literature, however its statement is not. We shall not give the
arguments here since they are almost identical to the ones given to prove Lemma 4.9 from~\cite{fols}, see
also~\cite{long}.

\begin{lemma}\label{longcyl}
Let $\eta > 0$ be fixed. Suppose that, in addition to the assumptions made above, every contractible $\tilde P =
(\tilde x,\tilde T)$ with $\tilde T \leq C$ belongs to $Q$. Then $\exists h>0$ with the following significance. If $\util = (a,u) : [r,+\infty) \times S^1 \rightarrow \R \times M$ is a finite-energy cylinder satisfying
\begin{enumerate}
 \item $E(\util) \leq C$ and $\int_{[r,+\infty) \times S^1} u^*d\lambda \leq e$.
 \item $\int_{\{s\}\times S^1} u^*\lambda \geq \eta \ \forall s \geq r$.
 \item $\exists\hat P \in Q^C$ and $y \in \hat P$ such that $\lim_{s\rightarrow+\infty} u(s,\cdot) = y \text{ in
     } C^\infty(S^1,M)$.
\end{enumerate}
then $s \geq r+h \Rightarrow u(s,\cdot) \in \W$.
\end{lemma}

It is not hard to check that, under the assumptions of Theorem~\ref{main3}, $Q = \p^*$ and $C=T$ satisfy the
hypotheses of Lemma~\ref{longcyl}.

\subsection{An auxiliary lemma}\label{auxlemma}

This subsection is independent of the previous discussion. Our goal is to prove Lemma~\ref{notdyn} below. Let
\[
 \wtil = (d,w) : \C\setminus\hat{\Gamma}\rightarrow \R\times M
\]
be a finite-energy sphere with non-degenerate asymptotics, where $\hat\Gamma\subset\C$ is finite. Suppose $\Gamma =
\hat\Gamma\cup\{\infty\}$ consists of non-removable punctures. Denote $z_1 = \infty$ and write $\hat\Gamma =
\{z_2,\dots,z_N\}$. We find periodic orbits $\{P_j = (x_j,T_j)\}_{j=1\dots N}$ such that $\wtil$ is asymptotic to
$P_j$ at $z_j$, according to Definition~\ref{behavior}.

Assume $P_j \in \p^* \ \forall j$ and that $z_2,\dots,z_N$ are negative punctures. There are distinguished homotopy classes $\beta_j = \beta_{P_j} \in \mathcal{S}_{P_j}$ induced by capping disks for the maps ${x_j}_{T_j}$, as explained in Remark~\ref{abc}, and we choose $Z_j$ in class $\beta_j$. By Theorem~\ref{p1} there exist sections of $w^*\xi$ defined near the punctures $z_j$, still denoted $Z_j$, such that the following holds: if $\psi$ is a holomorphic chart satisfying $\psi(0) = z_j$ then $Z_j(\psi(e^{2\pi(s+it)})) \to Z_j(x_j(T_jt+c))$ uniformly in $t$ as $s\to-\infty$, for some $c\in\R$.


\begin{lemma}\label{uniqueextension}
The sections $Z_j$ (defined only near the punctures $z_j$) extend to a non-vanishing section $B$ of $w^*\xi$. 
\end{lemma}

The proof follows from standard degree theory, we only sketch it here.

\begin{proof}[Sketch of proof]
We can glue capping disks $D_j$ for $P_j$ along the punctures $z_j$, for $j=2\dots,N$, with the surface $w$ to
obtain a capping disk
\[
 D_1 = D_2 \# \dots \# D_N \# w
\]
for $P_1$. This follows from the asymptotic behavior described in Definition~\ref{behavior}. The sections
$\{Z_j\}_{j\geq2}$ extend to a section $\sigma$ of $\xi|_{D_1}$. We used that $Z_j|_{P_j} \in \beta_{P_j} \ \forall
j\geq 2$ and that the class $\beta_{P_j}$ has special properties described in Subsection~\ref{orbits}. Since
$Z_1|_{P_1} \in \beta_{P_1}$ then $\sigma$ does not wind with respect to $Z_1$ near $\partial D_1$ and,
consequently, can be patched with $Z_1$.
\end{proof}

\begin{lemma}\label{negpun}
Fix $1\leq j\leq N$. Suppose the $P_j$ are non-degenerate. If $z_j$ is a negative puncture of $\wtil$,
$\mu_{CZ}(P_j) \geq 3$ and $\int_{\C\setminus\hat\Gamma} w^*d\lambda > 0$ then $\wind_\infty(\wtil,\delta,z_j) \geq
2$  where $\delta$ is the homotopy class of the section $B$ given by Lemma~\ref{uniqueextension}.
\end{lemma}

The proof requires the following very non-trivial theorem proved in~\cite{props1}.

\begin{theorem}[Hofer, Wysocki and Zehnder]\label{p12}
Suppose $P_j$ is non-degenerate and choose a holomorphic chart $\varphi:(V,0)\rightarrow(\varphi(V),z_j)$ centered
at $z_j$. Write $w(s,t) = w\circ\varphi(e^{-2\pi(s+it)})$ if $z_j$ is a positive puncture or $w(s,t) = w\circ
\varphi(e^{2\pi(s+it)})$ if $z_j$ is a negative puncture. By rotating the chart $\varphi$ we can assume that
$w(s,t)\rightarrow x_j(T_jt)$ as $|s|\rightarrow+\infty$ in $C^\infty$. Then either $\pi\cdot dw$ vanishes
identically or
\begin{enumerate}
 \item If $z_j$ is positive then $\exists f(s,t)\in\R\setminus\{0\}$ smooth such that
 \[
  \lim_{s\rightarrow+\infty} f(s,t) \pi\cdot\partial_sw = e(t)\text{ in } C^\infty(S^1,\xi)
 \]
 where $e$ is an eigenvector of $A_{P_j}$ associated to an eigenvalue $\nu\leq\nu^{neg}$.
 \item If $z_j$ is negative then $\exists f(s,t)\in\R\setminus\{0\}$ smooth such that
 \[
  \lim_{s\rightarrow-\infty} f(s,t) \pi\cdot\partial_sw = e(t)\text{ in }C^\infty(S^1,\xi)
 \]
 where $e$ is an eigenvector of $A_{P_j}$ associated to an eigenvalue $\nu\geq\nu^{pos}$.
\end{enumerate}
\end{theorem}

In Section~\ref{fredholm} we will generalize the above theorem, replacing the assumption that $P_j$ is
non-degenerate by the assumption that $\wtil$ has non-degenerate asymptotic behavior at $z_j$.

\begin{proof}[Proof of Lemma~\ref{negpun}]
Let $\varphi$ be as in the statement of Theorem~\ref{p12}. Write $Z_j(s,t) = Z_j \circ \varphi\left(\est\right)$ and
$B(s,t) = B\circ \varphi\left(\est\right)$, where $B$ is the non-vanishing section given by
Lemma~\ref{uniqueextension}. We compute
\[
 \begin{aligned}
  \wind_\infty(\wtil,\delta,z) &= \lim_{s\rightarrow-\infty} \wind(\pi\cdot\partial_sw(s,t),B(s,t),J) \\
  &=\lim_{s\rightarrow-\infty} \wind(\pi\cdot\partial_sw(s,t),Z_j(s,t),J) \\
  &=\wind(e(t+c),Z_j(x(Tt+c)),J) =(\nu,\beta_{P_j})
 \end{aligned}
\]
where $\nu$ is a positive eigenvalue of $A_{P_j}$ and $c\in\R$. The inequality $\mu_{CZ}(P_j) \geq 3$ implies
$(\nu,\beta_{P_j}) \geq 2$ in view of Lemma~\ref{dynconvex}.
\end{proof}

\begin{lemma}\label{notdyn}
Let $\lambda$ be a contact form on $M$, inducing the contact structure $\xi = \ker \lambda$, and let $J : \xi
\rightarrow \xi$ be a $d\lambda$-compatible complex structure. Suppose
\[
 \wtil = (d,w) : \C\setminus\hat{\Gamma}\rightarrow \R\times M
\]
is a finite-energy sphere with non-degenerate asymptotics, where $\emptyset\not=\hat\Gamma\subset\C$ is finite.
Suppose also that $\Gamma = \hat\Gamma\cup\{\infty\}$ consists of non-removable punctures. Write $z_1 = \infty$ and
$\hat\Gamma = \{z_2,\dots,z_N\}$. Assume $\wtil$ is asymptotic to $P_j$ at $z_j$ according to
Definition~\ref{behavior}, where $\{P_j = (x_j,T_j)\}_{j=1\dots N}$ are periodic Reeb orbits. Assume also
\begin{enumerate}
 \item $P_j \in \p^*$ and $P_j$ is non-degenerate $\forall j$.
 \item $z_1$ is a positive puncture and $z_2,\dots,z_N$ are negative punctures.
\end{enumerate}
Let $\delta$ be the homotopy class of the section $B$ given by Lemma~\ref{uniqueextension}. If $\wtil$ satisfies
$\int_{\C\setminus\hat{\Gamma}} w^*d\lambda>0$ and $\wind_\infty(\wtil,\delta,z_1) \leq 1$ then $\mu_{CZ}(P_j) < 3$
for some $j\geq 2$.
\end{lemma}

\begin{proof}
Let us assume that $\mu_{CZ}(P_j) \geq 3 \ \forall j\geq 2$. Using Lemma~\ref{negpun} we have the following chain of
inequalities
\[
 \begin{aligned}
  0 \leq\text{wind}_\pi&(\wtil) = \text{wind}_\infty(\wtil)+\#\hat{\Gamma}+1-\chi(S^2) \\
  &\leq 1 - \sum_{j\geq2} \text{wind}_\infty(\wtil,\delta,z_j) + 1 + \#\hat{\Gamma} - 2 \\
  &\leq-2\#\hat{\Gamma}+\#\hat{\Gamma}=-\#\hat{\Gamma}
 \end{aligned}
\]
proving that $\hat{\Gamma}=\emptyset$. This contradiction concludes the argument.
\end{proof}

As the proof demonstrates, the above lemma follows essentially from the inequality $\wind_\pi\geq 0$. This should be
seen as some kind of linearized version of positivity of self-intersections, and is violated when the asymptotic
winding is $\leq 1$ at the positive puncture and $\geq 2$ at the negative punctures. One should also note that this
very simple argument is independent of any transversality results.


\subsection{Bubbling-off analysis}

We start with a technical lemma.

\begin{lemma}\label{1round}
Assume $\lambda$ and $P=(x,T)$ satisfy the hypotheses of Theorem~\ref{main3} and recall the number $\gamma$ in
(\ref{gamma}). Suppose $r_n\rightarrow0^+$, $e>0$ and $\{\vtil_n = (d_n,v_n):\C\rightarrow \R \times M\}$ is a
sequence of finite-energy planes. Suppose further that
\begin{enumerate}
 \item $E(\vtil_n)\leq T$ and $\int_{\{\norma{z}>r_n\}} v_n^*d\lambda \leq \gamma$.
 \item $\int_{|z|=\rho} v_n^*\lambda \geq e \ \forall \rho>r_n$.
 \item $\exists N>0$ such that $\vtil_n(\C) \subset \left[ -N , +\infty \right) \times M \ \forall n$.
\end{enumerate}
Then $\lim_{n\rightarrow+\infty} \inf_K d_n = +\infty$ for every compact set $K\subset\C\setminus\{0\}$. 
\end{lemma}

\begin{proof}
We claim $\norma{d\vtil_n}$ is bounded on compact subsets of $\C\setminus\{0\}$. If not we may assume $\exists \zeta_n \rightarrow \zeta^* \not= 0$ satisfying $\norma{d\vtil_n(\zeta_n)} \rightarrow +\infty$. By Lemma~\ref{beforeclaim} we can further assume that for every $0<s<1$ there exists $\rho_n \rightarrow 0^+$ and a contractible Reeb orbit $\tilde P = (\tilde x,\tilde T)$ such that $\tilde T \leq T$ and $\limsup_{n\rightarrow+\infty} \int_{B_{\rho_n}(\zeta_n)} v_n^*d\lambda \geq s\tilde T$. Choosing $s>\gamma/\gamma_1$ we obtain a contradiction to (1).

We proceed indirectly. Suppose $\exists\{z_n\}\subset\C$ such that $z_n\rightarrow z^*\not=0$ and $\{d_n(z_n)\}$ is
bounded. Hence the sequence $\vtil_n(z_n)$ is compactly contained in $\R\times M$. We proved above that
$\{\norma{d\vtil_n}\}$ is $C^0_{loc}$-bounded on $\C\setminus\{0\}$. Thus $\{\vtil_n\}$ is $C^1_{loc}$-bounded on
$\C\setminus\{0\}$. Elliptic estimates provide $C^\infty_{loc}$-bounds and a subsequence $\{\vtil_{n_j}\}$
converging to a $\jtil$-holomorphic cylinder $f:\C\setminus\{0\} \rightarrow \R\times M$ in $C^\infty_{loc}$. It
satisfies $E(f)\leq T$ and $\int_{\C\setminus\{0\}} g^*d\lambda \leq \gamma$. We write $f=(h,g)$ and estimate
\begin{equation}\label{elocal}
 \int_{|z|=\rho} g^*\lambda = \lim_{j\rightarrow+\infty} \int_{|z|=\rho} v_{n_j}^*\lambda \geq e > 0 \ \forall \rho > 0.
\end{equation}
Hence $E(f)>0$ and $0$ is not a removable puncture. Moreover, $h\geq -N$ on $\C\setminus\{0\}$ since $\vtil_n(\C)
\subset \left[-N,+\infty\right)\times M \ \forall n$. Consequently $0$ is a positive puncture. Let us write
$g\left(e^{-2\pi(s+it)}\right) = g(s,t)$. By Theorem~\ref{thm93} there exists a periodic orbit $\tilde P=(\tilde
x,\tilde T)$ and a sequence $s_k\rightarrow +\infty$ such that
\[
 \lim_{k\rightarrow+\infty} g(s_k,t) = \tilde x(\tilde Tt+c) \text{ in } C^\infty(S^1,M).
\]
for some $c\in \R$. Thus $\lim_{k\rightarrow+\infty} \int_{\norma{z}=e^{-2\pi s_k}} g^*\lambda = -\tilde T < 0$
contradicting (\ref{elocal}).
\end{proof}

\begin{lemma}\label{round1}
Assume $\lambda$ and $P=(x,T)$ satisfy the hypotheses of Theorem~\ref{main3} and let $H\subset \R\times M$ be a
compact set. Suppose $\{\util_n = (a_n,u_n)\}\subset\Theta(H,P)$. Then one can find a subsequence of
$\{\util_{n_j}\}$, a finite set $\Gamma\subset\D$ and a smooth $\jtil$-holomorphic map
$\wtil:\C\setminus\Gamma\rightarrow \R\times M$ satisfying the following properties.
\begin{enumerate}
 \item Defining $\wtil_n:\C\rightarrow \R\times M$ by $\wtil_n(z) = (a_n(z)-a_n(2),u_n(z))$ then $\wtil_{n_j}
     \rightarrow \wtil$ in $C^\infty_{loc}(\C\setminus\Gamma,\R\times M)$ and $0<E(\wtil)\leq T$.
 \item $\exists A \in [-\infty,+\infty)$ such that $\inf_\C a_{n_j} \to A$.
 \item If $A = -\infty$ then $\#\Gamma = 1$, $\pi\cdot dw \equiv0$ and $H\cap\R\times x(\R)\not=\emptyset$.
 \item If $A > -\infty$ then $\Gamma = \emptyset$, $a_n(2) \rightarrow c\in\R$ and $\util_n\rightarrow \util :=
     c\cdot\wtil$ in $C^\infty_{loc}$. Moreover, $\util \in \Theta^A(H,P)$.
\end{enumerate}
\end{lemma}

\begin{proof}
By Corollary~\ref{corclaim} we find a subsequence $\{\util_{n_j}\} \subset \{\util_n\}$ and a finite set
$\Gamma\subset\D$ such that
\begin{itemize}
 \item[(i)] $|d\util_{n_j}|$ is $C^0_{loc}$-bounded on $\C\setminus\Gamma$.
 \item[(ii)] $\forall z^* \in \Gamma \ \exists z_j \rightarrow z^*$ such that $\norma{d\util_{n_j}(z_j)}
     \rightarrow +\infty$.
\end{itemize}
Define $\wtil_n(z)=(a_n(z)-a_n(2),u_n(z))$ and write $\wtil_n=(b_n,w_n)$. Note that
$|d\wtil_n(z)|=|d\util_n(z)|\forall z\in \C$ since the metric $g^0$ is $\R$-invariant. This proves that
$\wtil_{n_j}$ is $C^1_{loc}$-bounded on $\C\setminus\Gamma$ since $\wtil_{n_j}(2)\subset\{0\}\times M$. Elliptic
estimates provide $C^\infty_{loc}$-bounds. Thus we can assume, without loss of generality, that we can find a smooth
$\jtil$-holomorphic map
\[
 \wtil = (b,w) : \C\setminus\Gamma\rightarrow \R\times M
\]
and a subsequence $\{\wtil_{n_j}\}$ such that $\wtil_{n_j}\rightarrow\wtil$ in $C^\infty_{loc}(\C\setminus\Gamma)$. Clearly $E(\wtil)\leq \sup_j E(\wtil_{n_j}) = T$. We split the remaining arguments into a few steps. \\

\noindent \textbf{STEP 1:} $\wtil$ is not constant, all punctures in $\Gamma$ are negative and $\infty$ is the
unique positive puncture.

\begin{proof}[Proof of STEP 1]
If $\Gamma=\emptyset$ then $\int_\D w^*d\lambda = \lim_j \int_\D u_{n_j}^*d\lambda = T - \gamma > 0$, proving that
$\wtil$ is not constant. Suppose $\wtil$ is constant and $\Gamma\not=\emptyset$. Fix $z^*\in\Gamma\subset\D$ and
$1<s<\gamma^{-1}\min\{\gamma_1,\gamma_2\}$. By Corollary~\ref{claim} we assume, without loss of generality, that we
can find sequences $z_{n_j}\rightarrow z^*$ and $r_j(s)\rightarrow 0^+$ such that
\[
 \limsup_{j\rightarrow+\infty} \int_{|z-z_{n_j}|\leq r_j(s)} u_{n_j}^*d\lambda \geq s\gamma.
\]
Set $B_j = B_{r_j(s)}(z_{n_j})$. If $\wtil$ is constant we can estimate
\[
 0 = \int_{|z|=2} w^*\lambda = \lim_j \int_{|z|\leq2} w_{n_j}^*d\lambda \geq \limsup_j \int_{B_j} u_{n_j}^*d\lambda = s\gamma > 0.
\]
This contradiction shows $\wtil$ is not constant. Assume again $\Gamma\not=\emptyset$, fix $z^*\in\Gamma$ and
suppose, by contradiction, that it is a positive puncture. By Theorem~\ref{thm93} we find $r>0$ small and a periodic
Reeb orbit $P^*=(x^*,T^*)$ such that
\[
 -\int_{|z-z^*|=r }w^*\lambda \geq \frac{T^*}{2} > 0.
\]
However
\[
 -\int_{|z-z^*|=r }w^*\lambda = -\lim_j \int_{|z-z^*|=r} w_{n_j}^*\lambda = -\lim_j \int_{|z-z^*|\leq r}u_{n_j}^*d\lambda \leq 0.
\]
This contradiction shows that $z^*$ is a negative puncture. If $\infty$ is also a negative puncture then
$E(\wtil)<0$, which is impossible.
\end{proof}

%

We fix a $S^1$-invariant neighborhood $\W$ of the (discrete) set of $S^1$-orbits
\[
 \{ y \in C^\infty (S^1,M) : y \in \tilde P = (\tilde x, \tilde T) \in \p^* \text{ with } \tilde T \leq T \}
\]
in $C^\infty(S^1,M)$ with the following property: whenever $\hat P = (\hat{x},\hat{T}),\tilde{P} = (\tilde{x},\tilde{T})$ are two contractible Reeb orbits with $\max\{\hat T,\tilde T\} \leq T$ and $\hat T \not= \tilde T$ then no connected component of $\W$ contains the two loops $\hat{x}_{\hat{T}}$ and $\tilde{x}_{\tilde{T}}$ simultaneously. This can be done by our assumptions on $\lambda$ and $P$. Let $\W_1$ denote the component containing $x_T$. \\

\noindent \textbf{STEP 2:} There exist $h>0$ and $j_0\in\Z^+$ such that $\{t \mapsto u_{n_j}(re^{i2\pi t})\} \in \W_1$ for every $j\geq j_0$ and $r\geq 2e^h$.

\begin{proof}[Proof of STEP 2]
We can estimate
\begin{enumerate}
 \item $\limsup_j \int_{2\leq|z|\leq R} u_{n_j}^*d\lambda \leq \gamma\ \forall R\geq2$.
 \item $\int_{|z|=2} u_{n_j}^*\lambda = T - \gamma \ \forall j$.
\end{enumerate}
Applying Lemma~\ref{longcyl} to $\W$, $e=\gamma$, $\eta = T - \gamma$ and the cylinders
\[
 (s,t) \in [(2\pi)^{-1}\log 2,+\infty)\times S^1 \mapsto \util_{n_j} \left( \est \right)
\]
we obtain $h>0$ and $j_0$ such that the loop $t \mapsto u_{n_j}(re^{i2\pi t})$ is in $\W$ whenever $j\geq j_0$ and
$r\geq 2e^h$. By the path-connectedness of $\W_1$ these loops can not leave $\W_1$.
\end{proof}

\noindent \textbf{STEP 3:} The curve $\wtil$ has non-degenerate asymptotics and is asymptotic to $P$ at the puncture
$\infty$.

\begin{proof}[Proof of STEP 3]
We only deal with $\infty$, the other punctures are subject to analogous arguments. Use Theorem~\ref{thm93} to find
$c_+ \in \R$, $r_k\rightarrow +\infty$ and an orbit $P^+ = (x^+,T^+)$ such that
\[
 w \left( r_k e^{i2\pi t} \right) \rightarrow x^+(T^+t+c_+) \text{ in } C^\infty(S^1,M).
\]
Fix $k$ large. Then for $j$ large the loop $u_{n_j} \left( r_k e^{i2\pi t} \right)$ is $C^\infty$ close to $w \left(
r_k e^{i2\pi t} \right)$ and homotopic to $x(Tt)$. It follows that $P^+$ is contractible. Clearly $T^+ \leq E(\wtil)
\leq T$. Thus, by the assumptions of Theorem~\ref{main3}, $P^+$ is non-degenerate and belongs to $\p^*$. By
Theorem~\ref{p1} $\wtil$ has non-degenerate asymptotic behavior at the puncture $\infty$ and the associated
asymptotic Reeb orbit is $P^+$. It follows from STEP 2 that $P^+ = P$.
\end{proof}

By our assumptions on $\lambda$ and $P$, the asymptotic limits at the punctures $z\in \Gamma$ are non-degenerate orbits in $\p^*$, with periods $\leq T$ and indices $\mu_{CZ} \geq 3$. \\



\noindent \textbf{STEP 4:} If $\int_{\C\setminus\Gamma} w^*d\lambda>0$ then $\Gamma = \emptyset$ and
$\wind_\infty(\wtil) = 1$.

\begin{proof}[Proof of STEP 4]
Write $z_1 = \infty$ and $\Gamma = \{z_2,\dots,z_N\}$. Suppose $\wtil$ is asymptotic to $P_j=(x_j,T_j)$ at the
puncture $z_j$. As remarked above, $P_1 = P$, $\{P_1,\dots,P_N\}\subset \p^*$ and $P_j$ is non-degenerate $\forall
j$. Moreover, $\max_j T_j \leq T$ and $\mu_{CZ}(P_j)\geq 3 \ \forall j$. Let $U_1$ be an open neighborhood of
$x(\R)$ and $Z_1$ be a non-vanishing section of $\xi|_{U_1}$ satisfying $Z_1|_{P} \in \beta_{P}$. Let $B$ be
the non-vanishing section of $w^*\xi$ given by Lemma~\ref{uniqueextension} and $\delta$ be its homotopy class. By
Theorem~\ref{p1} we can find $R_0\gg2$ so that $r\geq R_0$ implies $\pi\cdot \partial_r w(re^{i2\pi t})\not=0$ and
$w(re^{i2\pi t}) \in U_1$. Let $h>0$ be given by STEP 2 and suppose $\W_1$ is small enough so that
$c\in\W_1\Rightarrow c(S^1)\subset U_1$. If we assume $R_0 \geq 2e^h$ then $w_{n_j}(re^{i2\pi t}) \in U_1$ whenever
$r\geq R_0$ and $j$ is large. Perhaps after making $j$ larger, we can also assume
$\pi\cdot\partial_rw_{n_j}(R_0e^{i2\pi t})\not=0\ \forall t\in S^1$ because $w_{n_j}\rightarrow w$ in the sense of
$C^\infty$ on the set $\{\norma{z}=R_0\}$. Define $l_j$ and $l$ by
\[
 \begin{aligned}
  l_j &= \wind(\pi\cdot\partial_rw_{n_j}(R_0e^{i2\pi t}),Z_1\circ w_{n_j}(R_0e^{i2\pi t}),J) \\
  l &= \wind(\pi\cdot\partial_rw(R_0e^{i2\pi t}),Z_1\circ w(R_0e^{i2\pi t}),J).
 \end{aligned}
\]
Then $l=\wind_\infty(\wtil,\delta,\infty)$. Note that $Z_1\circ w_{n_j}$, only defined on $\{|z|\geq R_0\}$, extends
to a non-vanishing section of $w_{n_j}^*\xi$ by the properties of the class $\beta_{P_1} = \beta_P$. Consequently,
$\forall j \ \exists R_j \gg R_0$ such that
\[
 1 = \wind_\infty(\util_{n_j}) = \wind(\pi\cdot\partial_rw_{n_j}(R_{j}e^{i2\pi t}),Z_1\circ w_{n_j}(R_{j}e^{i2\pi t}),J).
\]
The Gauss-Bonnet formula proves
\[
 1-l_j = \wind_\infty(\util_{n_j})-l_j=\#\{\text{zeros of }\pi\cdot dw_{n_j}\text{ in }\{R_0\leq|z|\leq R_{j}\}\}\geq0.
\]
It follows that $l_j\leq1$. We know that $l_j\rightarrow l$, proving $l\leq1$. If $\Gamma \not= \emptyset$ we can
apply Lemma~\ref{notdyn} to obtain a contradiction.
\end{proof}

\noindent \textbf{STEP 5:} If $\pi\cdot dw \equiv0$ then $\#\Gamma = 1$, $\Gamma \subset
\partial\D$ and $\wtil(\C\setminus\Gamma) \subset \R\times x(\R)$.

\begin{proof}[Proof of STEP 5]
We use Lemma~\ref{zerodlambda}. Let $p$ be a polynomial of degree $k\geq 1$ satisfying $p^{-1}(0)=\Gamma$ and let
$\hat P = (\hat x,\hat T)$ be a periodic orbit such that $\wtil=f_{(\hat x,\hat T)}\circ p$. It follows from STEP 3
that $(x,T) = (\hat x,k\hat T)$. Thus $x = \hat x$ and $k=1$. In fact, if $k\geq2$ then $T$ is not the minimal
period of $x$, contradicting the fact that $P$ is simply covered. Consequently $p(z)=Az+D$ and
$\Gamma=\{-D/A\}\subset\D$ for some $A\in\C$, $A\not=0$. It follows from Lemma~\ref{zerodlambda} and STEP 3 that
$\wtil(\C\setminus\Gamma) \subset \R\times x(\R)$. If $|-D/A|<1$ then we obtain the following contradiction
\begin{equation}\label{ineqzerodw}
 T = \int_{|z|=1} w^*\lambda = \lim_{j\rightarrow+\infty} \int_{|z|=1} u_{n_j}^*\lambda = T-\gamma.
\end{equation}
\end{proof}

\noindent \textbf{STEP 6:} $A=-\infty \Leftrightarrow \Gamma \not=\emptyset$.

\begin{proof}[Proof of STEP 6]
Let $\{z_j\}$ be so that $a_{n_j}(z_j) = \inf_\C a_{n_j} \rightarrow -\infty$. Suppose, by contradiction, that
$\Gamma = \emptyset$. Then $|d\util_{n_j}|$ is bounded on compact subsets of $\C$. We claim that $|z_j| \rightarrow
+\infty$. If not we can assume, after selecting a subsequence, that we have an uniform bound
$|a_{n_j}(z_j)-a_{n_j}(0)|\leq c$ for some $c>0$. This proves $a_{n_j}(0) \rightarrow -\infty$, contradicting
$\{\util_{n_j}(0)\} \subset H$ because $H$ is compact. Now define $\vtil_j:\C\rightarrow \R\times M$ by
\[
 \vtil_j(z) = (d_j(z),v_j(z)) = (a_{n_j}(z_jz)-a_{n_j}(z_j),u_{n_j}(z_jz)).
\]
Then the planes $\{\vtil_n\}$ satisfy the hypotheses of Lemma~\ref{1round} with $C=T$, $e=T-\gamma$,
$r_j=2|z_{n_j}|^{-1}$ and $N=0$, in view of the properties of the set $\Theta(H,P)$. However,
$\{\vtil_n(1)\}\subset\{0\}\times M$ and this is in contradiction to Lemma~\ref{1round}.

Assume $\Gamma\not=\emptyset$ and suppose, by contradiction, that $A>-\infty$. First we claim that $a_{n_j}(2) \to
+\infty$. If not we can assume, after selecting a subsequence, that $\exists N<+\infty$ satisfying $\sup_j
a_{n_j}(2)<N$. Choose $z\in\Gamma$. By STEP 1 $z$ is a negative puncture. In view of the asymptotic behavior
described in Theorem~\ref{p1} and of Definition~\ref{behavior}, $\exists\zeta\in\C\setminus\Gamma$ close to $z$ such
that $b(\zeta)< A-N-1$. Since $b_{n_j}(\zeta)\rightarrow b(\zeta)$ as $j\rightarrow+\infty$ then we can estimate
$a_{n_j}(\zeta) = b_{n_j}(\zeta) + a_{n_j}(2) < A-1$ for $j$ large. This contradicts the definition of $A$, proving
$a_{n_j}(2) \rightarrow +\infty$. We now claim that $0\in \Gamma$. If not then we find a smooth curve
$c:[0,1]\rightarrow\C$ such that $c(0)=0$, $c(1)=2$ and $c\left([0,1]\right)\cap\Gamma=\emptyset$. We estimate
$|a_{n_j}(2)-a_{n_j}(0)| \leq C \times \text{length}(c) \ \forall j$ for some $C>0$,
since the metric $g^0$ is $\R$-invariant. This is a contradiction since $\{a_n(0)\}$ is a bounded sequence ($H$ is
assumed compact) and $a_{n_j}(2) \rightarrow +\infty$. Hence $0\in\Gamma$. By STEP 4 we must have $\pi\cdot dw
\equiv 0$. By STEP 5 we conclude that $0\in S^1$, absurd.
\end{proof}

In view of STEP 4 and STEP 5 the points $0$ and $2$ do not belong to $\Gamma$. This allow us to conclude that
$a_{n_j}(2)$ is bounded since so is $a_{n_j}(0)$ ($H$ is compact). Up to selection of a subsequence, we can assume
$a_{n_j}(2) \to c \in \R$. If $A>-\infty$ then $\Gamma = \emptyset$ by STEP 6, and the plane $\util=c\cdot\wtil$
must belong to $\Theta^A(H,P)$ by STEP 3 and STEP 4. If $A=-\infty$ then $\Gamma \not=\emptyset$ by STEP 6,
$\pi\cdot dw \equiv0$ by STEP 4 and $\wtil(0) = (b(0),w(0)) \in \R \times x(\R)$ by STEP 5. We know $\util(0) =
\lim_j \util_{n_j}(0) \in H$. The conclusion follows since $\util(0) = (b(0)+c,w(0))$.
%
%
\end{proof}

\subsection{End of the proof of Theorem~\ref{main3}}\label{endmain3}

Suppose $\lambda$ and $P$ satisfy the assumptions of Theorem~\ref{main3} and that $H\subset \R\times M$ is compact.
Take $\{\util_n\} \subset \Theta^L(H,P)$ and set $A_n = \inf_\C a_n$. We can assume, after selecting a subsequence,
that $A_n \rightarrow A \in [L,+\infty)$. By Lemma~\ref{round1} we find a subsequence $\{\util_{n_j}\}$ and some
$\util \in \Theta^A(H,P) \subset \Theta^L(H,P)$ such that $\util_{n_j} \rightarrow \util$ in $C^\infty_{loc}$. This
proves $\Theta^L(H,P)$ is $C^\infty_{loc}$-compact.

Assume $H \cap \R\times x(\R) = \emptyset$, consider $\{\util_n\} \subset \Theta(H,P)$ and set $A_n = \inf_\C a_n$.
If $\inf_n A_n = -\infty$ then we can assume $A_n \rightarrow -\infty$. By Lemma~\ref{round1} we conclude $H \cap
\R\times x(\R) \not= \emptyset$, a contradiction. Thus we find $L>0$ such that $\{\util_n\} \subset \Theta^L(H,P)$.
This proves $\Theta(H,P)$ is $C^\infty_{loc}$-compact.

Now suppose $\{\util_n\} \subset \Lambda^L(H,P) \subset \Theta^L(H,P)$. We already know $\exists \util \in \Theta^L(H,P)$ and a subsequence $\{\util_{n_j}\}$ such that $\util_{n_j} \rightarrow \util$ in $C^\infty_{loc}$. We must show that $\util$ is an embedding. It must be an immersion since $\wind_\pi(\util) = \wind_\infty(\util) - 1 = 0$. 
Let $\Delta$ be the diagonal in $\C \times \C$ and consider the set
\[
 D = \{ (z_1,z_2) \in \C \times \C \setminus \Delta: \util(z_1) = \util(z_2) \}.
\]
If $D$ has a limit point in $\C\times\C \setminus \Delta$ then we find, using the similarity principle as
in~\cite{mcdsal}, a polynomial $p:\C\rightarrow\C$ of degree $\geq 2$ and a $\jtil$-holomorphic map $f:\C\rightarrow
\R\times M$ such that $\util = f\circ p$. This forces zeros of $d\util$, a contradiction. Thus $D$ is closed and
discrete in $\C\times\C \setminus \Delta$. By stability and positivity of intersections of pseudo-holomorphic
immersions we find self-intersections of the maps $\util_{n_j}$ for large values of $j$ if $D \not= \emptyset$. This
is a contradiction since each $\util_n$ is an embedding. We proved $D=\emptyset$ and $\Lambda^L(H,P)$ is
$C^\infty_{loc}$-compact. The same reasoning as above shows $\Lambda(H,P)$ is $C^\infty_{loc}$-compact if $H \cap
\R\times x(\R) = \emptyset$. The proof of Theorem~\ref{main3} is complete.


\section{Existence of fast planes}\label{sectionexistence}

In this section we prove Theorem~\ref{existfast}.

\subsection{Special boundary conditions for the Bishop Family}

We will need special totally real boundary conditions for our Bishop family of $\jtil$-holomorphic disks described
in Subsection~\ref{existence}.

\begin{propn}\label{perturbdisk}
Let $P=(x,T)$ be a non-degenerate, unknotted and simply covered periodic Reeb orbit on a $3$-manifold $M$ equipped
with a contact form $\lambda$ and Reeb vector field $R$. Let $\Dcal_0$ be a smooth embedded disk satisfying
$\partial \Dcal_0 = x(\R)$. For any open neighborhood $U$ of $x(\R)$ in $M$ there exists a smooth embedded disk
$\Dcal_1$ satisfying the following properties.
\begin{enumerate}
 \item $\partial \Dcal_1 = x(\R)$, $\Dcal_1 \setminus U = \Dcal_0 \setminus U$.
 \item There exists a neighborhood $\OO$ of $\partial \Dcal_1$ in $\Dcal_1$ such that
 \begin{equation}
  R_p \not\in T_p\Dcal_1, \ \forall p \in \OO \setminus \partial \Dcal_1.
 \end{equation}
\end{enumerate}
\end{propn}

The goal of this subsection is to prove the above statement. We start with some technical lemmas. Let $sp(1)$ denote
the Lie algebra of $\Sp(1)$.

\begin{lemma}\label{techlog}
Let $\psi\in C^\infty([0,1],\Sp(1))$ satisfy $\det [I-\psi(1)] \not= 0$, suppose that $Y \in sp(1)$ satisfies $e^Y = \psi(1)$ and $t \mapsto \psi^\prime(t)\psi^{-1}(t)$ extends to a smooth $1$-periodic function on $\R$. Then $M(t) = e^{tY}\psi^{-1}(t)$ also extends to a smooth $1$-periodic function on $\R$. Moreover, if $\det Y > 0$ then $\exists \ \tilde Y \in sp(1)$ satisfying $e^{\tilde Y} = \psi(1)$, such that the smooth map $\tilde M : t\in \R/\Z \rightarrow e^{t\tilde Y}\psi^{-1}(t) \in \Sp(1)$ satisfies $\maslov(\tilde M) = 0$.
\end{lemma}

\begin{proof}
Clearly $M$ extends to a continuous $1$-periodic function. The formula $M^\prime = YM -
e^{tY}\psi^{-1}\psi^\prime\psi^{-1} = YM - M\psi^\prime\psi^{-1}$ shows that so does $M'$. An induction argument
proves that all derivatives of $M$ extend to continuous $1$-periodic functions on $\R$. Now assume $\det Y > 0$. The
eigenvalues of $Y$ are $\pm i\sqrt{\det Y}$ since $\tr Y = 0$. In this case the Jordan form of $Y$ is the matrix
$J=i\gamma$ for $\gamma = \sqrt{\det Y} \not\in 2\pi\Z$. There exists $T\in GL(2,\R)$ such that $Y=T^{-1}i\gamma T$.
Thus $\psi(1)=T^{-1}e^{i\gamma}T$. Denote by $C= \text{diagonal}(1,-1)$ the conjugation matrix and let
$k=\text{Maslov}(M)$. Define a loop of symplectic matrices by $N(t) := T^{-1}e^{-i2\pi kt}T$, $t\in \R/\Z$.

We claim that $\maslov(N) = -k$ if $\det T>0$ and $\maslov(N) = k$ if $\det T<0$. In fact, suppose $\det T<0$. Then
there exists a smooth path $s\in [0,1] \mapsto T(s) \in GL(2,\R)$ with $T(0) = T$ and $T(1) = C$. Using the homotopy
invariance of the Maslov index we compute
\[
 \maslov(N) = \maslov(t \mapsto Ce^{-i2\pi kt}C) = \maslov(t \mapsto e^{i2\pi kt}) = k.
\]
The case $\det T>0$ is similar and the claim is proved.

We continue the proof of the lemma considering the case $\det T<0$. Then $\maslov(N) = k$. Using $e^{tY} =
T^{-1}e^{it\gamma}T$ we  compute
\[
 \begin{aligned}
  0 & = \text{Maslov}(M(t))-k = \maslov (N^{-1}(t)M(t)) \\
  & = \text{Maslov}(T^{-1}e^{it(\gamma+2\pi k)}T\psi^{-1}(t))
 \end{aligned}
\]
The conclusion follows by noting that $\tilde Y := T^{-1}i(\gamma+2\pi k)T$ is another logarithm of $\psi(1)$. 
The case $\det T>0$ is treated similarly. The eigenvalues of $\tilde Y$ are $\pm i(\sqrt{\det Y} + 2\pi k) \not\in
i2\pi \Z$, and we still have $\det \tilde Y > 0$.
\end{proof}

Before starting with the proof, we fix the notation and make some initial constructions. Let $U$, $\Dcal_0$ and $P$
be as in Proposition~\ref{perturbdisk}. Perhaps after making $U$ smaller, we can find a Martinet tube
$$ \Psi : U \rightarrow \R/\Z \times B $$ as explained in Definition~\ref{martinettube}. Here $B \subset
\R^2$ is an open ball centered at the origin. In the coordinates $(\theta,x,y) \in \R/\Z \times \R^2$ the contact
form is $\Psi_*\lambda = f(d\theta + xdy)$ where the smooth function $f>0$ satisfies $f(\theta,0,0) \equiv T$ and
$df(\theta,0,0) \equiv 0$. Note that $T$ is the minimal positive period of $x$ by assumption. The map $t \mapsto
x(Tt)$ is represented by $t \mapsto (t,0,0)$. We still denote by $R$ the Reeb vector in the local coordinates
$(\theta,x,y)$. The proof of Proposition~\ref{perturbdisk} is based on the following lemma.

\begin{lemma}\label{nicestrip}
Let $\pi_0:\R/\Z\times \R^2\rightarrow\R^2$ be the projection onto the second factor. There exists $\epsilon>0$ and
a smooth embedding $h:(1-\epsilon,1]\times \R/\Z \to \R/\Z \times B$ satisfying
\begin{enumerate}
 \item $h(1,t)=(t,0,0)$.
 \item $\{h_r(r,t), h_t(r,t), R\circ h(r,t)\}$ is linearly independent if $1-\epsilon<r<1$.
 \item $\wind \left( t\mapsto d\pi_0\cdot h_r(1,t) \right) = 0$.
\end{enumerate}
\end{lemma}

\begin{proof}
We denote the Reeb flow by $\phi_t$ and assume, for simplicity and without loss of generality, that $T=1$. In the
local coordinates $(\theta,x,y)$ introduced above we have $R(\theta,0,0) = \left(1,0,0\right)$,
$\phi_t(\theta_0,0,0) = (t+\theta_0,0,0)$ and
\[
 D\phi_t(0,0,0)=\begin{pmatrix} 1 & 0 \\ 0 & \psi(t) \end{pmatrix}\text{ for some } \psi\in C^\infty\left([0,1],Sp(1)\right)
\]
with respect to the splitting $T(\R/\Z \times\R^2) = \R \oplus\R^2$. We used that the linearized Reed flow preserves
the splitting $TM = \R R \oplus \xi$. The formula
\[
 \frac{d}{dt}D\phi_t(0,0,0) = DR(\phi_t(0,0,0)) D\phi_t(0,0,0),
\]
shows that the matrix $\psi^\prime\psi^{-1}(t)$ is a smooth loop $\R/\Z \rightarrow \R^{2\times 2}$. $P$ is a
non-degenerate Reeb orbit if, and only if, $1$ is not an eigenvalue of $\psi(1)$.

Let us assume the existence of a logarithm $Y \in sp(1)$ of $\psi(1)$. By Lemma~\ref{techlog} the function
\begin{equation}\label{matrixM}
 M(t)=e^{tY}\psi^{-1}
\end{equation}
defines a loop $\R/\Z \rightarrow Sp(1)$ of class $C^\infty$. Consider the diffeomorphism
\[
 G: (t,x,y) \mapsto\left( t , M(t) \begin{pmatrix} x \\ y \end{pmatrix} \right)
\]
of $\R/\Z\times \R^2$ onto itself. Note that $G$ is smooth since so is $M$. We compute
\[
 \begin{aligned}
  & D(G_*\phi_t)(0,0,0) = DG\left(\phi_t(0,0,0)\right)\cdot D\phi_t\left(0,0,0\right)\cdot DG^{-1}(0,0,0) \\
  &= \begin{pmatrix} 1 & 0 \\ 0 & M(t) \end{pmatrix} \cdot \begin{pmatrix} 1 & 0 \\ 0 & \psi(t) \end{pmatrix} \cdot I = \begin{pmatrix} 1 & 0 \\ 0 & e^{tY} \end{pmatrix},
 \end{aligned}
\]
and prove
\[
 \begin{aligned}
  D(G_*R)(t,0,0)&=\left[\frac{d}{dt}D(G_*\phi_t)(0,0,0)\right]\left[D(G_*\phi_t)(0,0,0)\right]^{-1} \equiv \begin{pmatrix} 0 & 0 \\ 0 & Y \end{pmatrix}.
 \end{aligned}
\]
From now on we work in these new coordinates, obtained by pushing forward with $G$. We still denoted them by
$(\theta,x,y)$ without fear of ambiguity. The Reeb vector is still denoted by $R$
and its flow by $\phi_t$. The above equations imply $$ DR(\theta,0,0)\equiv \begin{pmatrix} 0 & 0 \\
0 & Y \end{pmatrix}. $$ The characteristic polynomial of $Y$ is $p(\lambda) = \lambda^2-\text{tr}(Y)\lambda +\det Y$. Since tr$(Y)=0$ its roots are
$\pm i\sqrt{\det Y}$ if $\det Y>0$, or $\pm \sqrt{-\det Y}$ if $\det Y<0$. The case $\det Y=0$ is ruled out since
$P$ is non-degenerate. Let $k=\text{Maslov}(M)$ and consider the smooth $2\pi$-periodic function
\begin{equation}\label{functiong}
  g(\vartheta) = \det\left(\begin{pmatrix}  1 \\ 0 \end{pmatrix} , Ad_{e^{-i\vartheta}}Y \begin{pmatrix}   1 \\ 0  \end{pmatrix} \right) \\ = \det\begin{pmatrix} 1 & (e^{-i\vartheta}Ye^{i\vartheta})_{11} \\ 0 & (e^{-i\vartheta}Ye^{i\vartheta})_{21} \end{pmatrix} = (e^{-i\vartheta}Ye^{i\vartheta})_{21}
\end{equation}
of the real variable $\vartheta$. We split the proof in two cases. \\

\noindent \textit{Case 1}: $\det Y>0$. By Lemma~\ref{techlog} we can assume $k=0$. It follows that
$\mu_{CZ}(e^{tY})=\mu_{CZ}(\psi(t))$. Since the Jordan form of $Y$ is the matrix
\[
 J=i\gamma=
 \begin{pmatrix}
  0 &-\gamma \\
  \gamma &0
 \end{pmatrix}, \text{ with } \gamma = \sqrt{\det Y} >0,
\]
we find $T\in GL(2,\R)$ such that $Y=T^{-1}i\gamma T$. Thus $\psi(1)=T^{-1}e^{i\gamma}T$ and
\[
g(\vartheta)=\frac{1}{\det T} \det \begin{pmatrix} Te^{i\vartheta} \begin{pmatrix} 1 \\ 0 \end{pmatrix} ,JTe^{i\vartheta} \begin{pmatrix} 1 \\ 0 \end{pmatrix} \end{pmatrix}.
\]
This proves $g(\vartheta)\not=0$ for all $\vartheta\in S^1$. Consider
\[
 \hat{h}(r,t) = \begin{pmatrix}   t \\ 1-r \\ 0  \end{pmatrix}
\]
defined for $(r,t) \in (1-\epsilon,1]\times \R/\Z$, with $\epsilon>0$ small. Then
\[
 R(\hat{h}(r,t))=  \begin{pmatrix}   1 \\ 0 \\ 0  \end{pmatrix}  +(1-r)  \begin{pmatrix}   0 \\   Y_{11} \\ Y_{21} \end{pmatrix}  +O(|1-r|^2).
\]
This implies
\[
 \begin{aligned}
 \det\left( \hat{h}_r,\hat{h}_t,R\right) & =\det \begin{pmatrix} \begin{array}{r} 0 \\ -1 \\ 0 \end{array} & \begin{array}{c} 1 \\ 0 \\ 0 \end{array} & \begin{array}{c} 1 \\ (1-r)Y_{11} \\ (1-r)Y_{21} \end{array} \end{pmatrix} +O(|1-r|^2) \\
 &=(1-r)g(0)+O(|1-r|^2)
 \end{aligned}
\]
and we conclude $\det( \hat{h}_r,\hat{h}_t,R)\not=0$ for $r<1$ close to $1$. This tells us that the map
$h=G^{-1}\circ \hat{h}$ satisfies conditions (1) and (2) if $\epsilon$ is small enough. We compute
\[
  h_r(1,t) = DG^{-1}(\hat{h}(1,t))\cdot\hat{h}_r(1,t) = -\begin{pmatrix} 0 \\ M^{-1}(t) \begin{pmatrix} 1 \\ 0 \end{pmatrix} \end{pmatrix}
\]
which proves assertion (3) since Maslov$(M)=0$. \\

\noindent {\it Case 2:} $\det Y < 0$. In this case $Y=T^{-1}JT$ with $J = \text{diagonal}(\gamma,-\gamma)$ and
$\gamma=\sqrt{-\det Y}>0$. The function $g(\vartheta)$ defined in (\ref{functiong}) is $2\pi$-periodic and we find
\[
 \vartheta_0<\vartheta_1<\vartheta_2<\vartheta_3<\vartheta_4=\vartheta_0+2\pi,
\]
with $\vartheta_0\leq0$, so that $g$ changes sign at every $\vartheta_j$. Here $\vartheta_j$ are precisely the
numbers when $Te^{i\vartheta}$ changes quadrant. Define
\[
 \hat{h}(r,t) = \begin{pmatrix} t \\ (1-r)e^{i\vartheta(t)}\begin{pmatrix} 1 \\ 0\end{pmatrix} \end{pmatrix}  
\]
for $(r,t) \in (1-\epsilon,1] \times \R/\Z$, with $\epsilon>0$ small. We are still to find the real-valued function
$\vartheta(t)$. Let us make some {\it a priori} computations:
\[
 \hat{h}_r=\begin{pmatrix} 0 \\ -e^{i\vartheta} \begin{pmatrix} 1 \\ 0 \end{pmatrix} \end{pmatrix},\ \hat{h}_t=\begin{pmatrix} 1 \\ (1-r)i\vartheta^\prime e^{i\vartheta} \begin{pmatrix} 1 \\ 0 \end{pmatrix} \end{pmatrix}
\]
and
\[
 R(\hat{h}(r,t))= \begin{pmatrix} 1 \\ 0 \\ 0 \end{pmatrix} + (1-r) \begin{pmatrix} 0 \\ Ye^{i\vartheta} \begin{pmatrix} 1 \\ 0 \end{pmatrix} \end{pmatrix} + O(|1-r|^2).
\]
We have
\[
 \begin{aligned}
  \det\left( \hat{h}_r,\hat{h}_t,R\right) & =\det \begin{pmatrix} \begin{array}{r} 0 \\ -1 \\ 0 \end{array} & \begin{array}{c} 1 \\ 0 \\ (1-r)\vartheta^\prime \end{array} & \begin{array}{c} 1 \\ (1-r)\left(Ad_{e^{-i\vartheta}}Y\right)_{11} \\ (1-r)\left(Ad_{e^{-i\vartheta}}Y\right)_{21} \end{array} \end{pmatrix} +O(|1-r|^2) \\
  &=-(1-r)\vartheta^\prime + (1-r) (e^{-i\vartheta}Ye^{i\vartheta})_{21} + O(|1-r|^2) \\
  &=-(1-r)(\vartheta^\prime-g(\vartheta))+O(|1-r|^2).
 \end{aligned}
\]
Thus it suffices to prove that $t\mapsto\vartheta(t)$ can be chosen to satisfy:
\begin{enumerate}
 \item $\vartheta(t)$ is smooth in a neighborhood of $[0,1]$, and $\vartheta(1)=\vartheta(0)+2\pi k$.
 \item $\vartheta^\prime(t) : \R/\Z \rightarrow\R$ is smooth and $\vartheta^\prime-g(\vartheta)$ does not
     vanish.
\end{enumerate}
We first handle the case $k<0$. Fix some non-empty open interval $I\subset(2\pi k,0)$ where $g\geq\sigma>0$. Let
$\eta = \sup|g|$ and $a>2\eta$. Now choose $\vartheta(t)$ so that
\begin{enumerate}
  \item $\vartheta(t)=-at$ for $t$ close to $0$, and $\vartheta(t)=-a(t-1)+2\pi k$ for $t$ close to $1$.
  \item $\vartheta^\prime(t)<0\ \forall\ t\in[0,1]$, and $\vartheta^\prime(t)\not=-a \Leftrightarrow
      \vartheta(t)\in I$.
\end{enumerate}
Then $\vartheta^\prime-g(\vartheta) < -a+\eta < -\eta$ if $\vartheta(t)\not\in I$, and $\vartheta^\prime-g(\vartheta)\leq0-\sigma=-\sigma$ if $\vartheta(t)\in I$. This proves that $\vartheta^\prime-g(\vartheta)\leq\min(-\sigma,-\eta)<0$. The case $k>0$ is treated similarly. If $k=0$ we take $\vartheta(t)\equiv\vartheta_0$ with $g(\vartheta_0)\not=0$. In all cases $\hat{h}(r,t)$ is a smooth embedding, $\hat{h}(1,t)=(t,0,0)$ and $\{\hat h_r(r,t),\hat h_t(r,t),R(\hat{h}(r,t))\}$ is a linearly independent set if $r<1$ and $\epsilon$ is fixed small enough. Moreover,
\[
 \text{wind}(t \mapsto d\pi_0 \cdot \partial_r \hat{h}(1,t))=k.
\]
Composing with $G^{-1}$ we arrive at $h=G^{-1}\circ \hat{h}$ with similar properties, but with $\text{wind}(t\mapsto
d\pi_0 \cdot h_r(1,t))=0$: the winding is corrected from $k$ to $0$ since
\[
  h_r(1,t) = -\begin{pmatrix} 0 \\ M^{-1}(t)e^{i\vartheta(t)} \begin{pmatrix} 1 \\ 0 \end{pmatrix} \end{pmatrix}
\]
and Maslov$(M^{-1})=-k$.

If $\psi(1)$ does not have a logarithm then its spectrum must be a pair of real negative numbers $a,a^{-1} \neq -1$ and one finds $T\in GL(2,\R)$ such that $- \psi(1) = T^{-1} e^J T$ with $J = \text{diagonal}(\ln(-a),-\ln(-a)\}$. We proceed as above setting $M(t) = K(t)\psi(t)^{-1}$ with $K(t) = T^{-1} e^{i\pi t}e^{tJ} T$, and defining $$ \hat h = \begin{pmatrix} t \\ (1-r)T^{-1}e^{i\vartheta(t)} \begin{pmatrix} 1 \\ 0 \end{pmatrix} \end{pmatrix} $$ where $\vartheta(t)$ satisfies $\vartheta' - [Ad_{e^{-i\vartheta}}(i\pi + e^{i\pi t}Je^{-i\pi t})]_{21}\neq 0$ and $\vartheta(1) - \vartheta(0) = 2\pi k$ with $k = \maslov(M)$.
\end{proof}

With Lemma~\ref{nicestrip} proved we can continue our constructions towards the proof of
Proposition~\ref{perturbdisk}. Let $\varphi_0 : \D \hookrightarrow M$ be a $C^\infty$-embedding satisfying
$\varphi_0(\D) = \Dcal_0$ and $\varphi_0 (e^{i2\pi t}) = x(Tt) \ \forall t\in \R$. We have
\begin{equation}\label{contido}
\varphi_0 \left( \{1-\epsilon_0\leq|z|\leq 1\} \right) \subset U
\end{equation}
for some $\epsilon_0 > 0$ small. Let $(r,t) \in [1-\epsilon_0,1]\times \R/\Z$ be polar coordinates in the annulus
$\{1-\epsilon_0 \leq |z| \leq 1\}$. We may write $\varphi_0(r,t)$ instead of $\varphi_0(re^{i2\pi t})$. In the local
coordinates $(\theta,x,y)$, $\xi|_{x(Tt)}$ is represented by $\{0\} \times \R^2$. As explained in
Remark~\ref{homtriv}, the vector field along $(\theta,0,0)$ given by $\theta \mapsto
\partial_x|_{(\theta,0,0)} = (0,1,0)|_{(\theta,0,0)}$ can be chosen to represent any given non-vanishing
smooth vector tangent to the contact structure along $t \mapsto x(Tt)$. We will assume
\begin{equation}\label{niceclass}
 \partial_x|_{(t,0,0)} = \partial_r (\Psi \circ \varphi_0) (1,t) \ \forall t\in \R/\Z.
\end{equation}
We prove a few preliminary steps. \\

\noindent {\bf STEP 1:} There exists $0<\epsilon<\epsilon_0$, a diffeomorphism $F$ of the set $(1-\epsilon,1] \times
\R/\Z$ onto itself and a smooth function $\gamma_0 : (1-\epsilon,1] \times \R/\Z \rightarrow B$ such that
\begin{enumerate}
 \item If $F = (F_1,F_2)$ then $F_1(r,t) = r$.
 \item $\Psi \circ \varphi_0 \circ F^{-1} (r,\vartheta) = (\vartheta,\gamma_0(r,\vartheta)), \ \forall (r,\vartheta) \in (1-\epsilon,1] \times \R/\Z$.
 \item $\gamma_0(r,\vartheta) \not= 0$ if $r<1$.
 \item $\partial_r \gamma_0 (1,\vartheta) \not= 0, \ \forall \vartheta \in \R/\Z$.
 \item $\wind \left( \vartheta \mapsto \partial_r \gamma_0 (1,\vartheta) \right) = 0$.
\end{enumerate}

\begin{proof}[Proof of STEP 1]
In this proof we write $\varphi_0$ instead of $\Psi \circ \varphi_0$ for simplicity. Then $\theta \circ
\varphi_0(1,t) = t, \ \forall t$. Consider the map $F(r,t) = (r,\theta \circ \varphi_0 (r,t))$ defined for $ (r,t)
\in [1-\epsilon_0,1] \times \R/\Z$. One has
\[
 DF(1,t) = \begin{pmatrix} 1 & 0 \\ d\theta \cdot \partial_r \varphi_0(1,t) & d\theta \cdot \partial_t \varphi_0(1,t) \end{pmatrix} = \begin{pmatrix} 1 & 0 \\ * & 1 \end{pmatrix}.
\]
Since $F$ maps $\{1\} \times \R/\Z$ diffeomorphically onto $\{1\} \times \R/\Z$, we can use the inverse function
theorem to find $0<\epsilon\ll\epsilon_0$ such that $F$ is a diffeomorphism of the set $(1-\epsilon,1] \times \R/\Z$
onto itself.
It follows that
\[
 \theta \circ \varphi_0 \circ F^{-1} (r,\vartheta) = \vartheta, \ \forall (r,\vartheta) \in (1-\epsilon,1] \times \R/\Z
\]
since $F \circ F^{-1} = id$. Let $\gamma_0$ be implicitly defined by
\begin{equation}\label{defngamma0}
 \varphi_0 \circ F^{-1} (r,\vartheta) = (\vartheta,\gamma_0(r,\vartheta)).
\end{equation}
The map $\gamma_0$ satisfies conditions ($3$) and ($4$) since $\varphi_0$ is an embedding. Condition ($5$) follows
trivially since
\[
 \text{image of }d\varphi_0(1,t) = \text{span} \left\{ (1,0,0),(0,\partial_r \gamma_0(1,t)) \right\}
\]
and $\wind (t\mapsto d\pi_0 \cdot \partial_r \varphi_0 (1,t)) = 0$ by (\ref{niceclass}). Here $\pi_0:\R/\Z\times
\R^2\rightarrow\R^2$ denotes the projection onto the second factor.
\end{proof}

Arguing exactly as in STEP 1 we prove \\

\noindent {\bf STEP 2:} Let $h$ be the map given by Lemma~\ref{nicestrip}, defined on $(1-\epsilon,1]\times\R/\Z$
for some $\epsilon>0$. We can assume $\epsilon<\epsilon_0$ and find a diffeomorphism $H$ of the set $(1-\epsilon,1]
\times \R/\Z$ onto itself, and a smooth function $\gamma_1 : (1-\epsilon,1] \times \R/\Z \rightarrow B$ satisfying
\begin{enumerate}
 \item If $H = (H_1,H_2)$ then $H_1(r,t) = r$.
 \item $h \circ H^{-1} (r,\vartheta) = (\vartheta,\gamma_1(r,\vartheta)), \ \forall (r,\vartheta) \in
     (1-\epsilon,1] \times \R/\Z$.
 \item $\gamma_1(r,\vartheta) \not= 0, \ \forall (r,\vartheta) \in (1-\epsilon,1) \times \R/\Z$.
 \item $\partial_r \gamma_1 (1,\vartheta) \not= 0, \ \forall \vartheta \in \R/\Z$.
 \item $\wind \left( \vartheta \mapsto \partial_r \gamma_1 (1,\vartheta) \right) = 0$. \\
\end{enumerate}

$B_\eta \subset \R^2$ will denote the open ball of radius $\eta>0$ centered at the origin. \\

\noindent {\bf STEP 3:} Let $g(r,\vartheta) := |\pi_0 \circ \Psi \circ \varphi_0(r,\vartheta)|^2$, where $\pi_0 :
\R/\Z \times \R^2 \rightarrow \R^2$ is the projection onto the second factor. $\exists \ 0<\epsilon<\epsilon_0$ and
$\eta>0$ such that $\cl{B_\eta} \subset B$ and
\begin{enumerate}
 \item $\Dcal_0 \cap \Psi^{-1}(\R/\Z \times \cl{B_\eta}) \subset \varphi_0((1-\epsilon,1]\times\R/\Z)$.
 \item $\partial_r g < 0$ on $[1-\epsilon,1)\times\R/\Z$.
\end{enumerate}

\begin{proof}[Proof of STEP 3]
Denote $\Psi \circ \varphi_0 = (\beta,\Gamma) \in \R/\Z \times B$ where $\beta$ and $\Gamma$ are defined on $\R/\Z
\times (1-\epsilon_0,1]$, so that $g(r,\vartheta) = |\Gamma(r,\vartheta)|^2$. Clearly $g(r,\vartheta) = 0$ if, and
only if, $r=1$. We also know that $\partial_r\Gamma(1,\vartheta) \not= 0$ $\forall \vartheta \in \R/\Z$. Let $C>0$
be a constant so that $|\Gamma(r,\vartheta) + (1-r)\partial_r\Gamma(1,\vartheta)| \leq C|1-r|^2$ holds for every
$(r,\vartheta) \in [1-\epsilon_0,1]\times \R/\Z$. This follows from expanding $r \mapsto \Gamma(r,\vartheta)$ up to
first order at $(1,\vartheta)$. Thus
\[
 \begin{aligned}
  \partial_r g (r,\vartheta) &= 2 \left< \Gamma (r,\vartheta) , \partial_r \Gamma (r,\vartheta) \right> \\
  &= -2(1-r) \left< \partial_r \Gamma (1,\vartheta) , \partial_r \Gamma (r,\vartheta) \right> + O(|1-r|^2)
 \end{aligned}
\]
which is clearly strictly negative if $1-\epsilon<r<1$ for some $\epsilon>0$ small. We used that $\epsilon>0$ small
enough implies $\left< \partial_r \Gamma (1,\vartheta) , \partial_r \Gamma (r,\vartheta) \right>$ is positive and
bounded away from $0$ on $\R/\Z\times(1-\epsilon,1]$. The existence of $\eta$ is easy since $\varphi_0$ is an
embedding.
\end{proof}

From now on we fix $0<\epsilon<\epsilon_0$ and $\eta>0$ such that the conclusions of Lemma~\ref{nicestrip} and of steps 1, 2 and 3 hold. \\

\noindent {\bf STEP 4:} Let $\gamma_j$ ($j=0,1$) be the maps obtained from STEP 1 and STEP 2. There exist numbers
$0<\delta_1<\delta_0\ll \epsilon_2<\epsilon$ and a smooth map $\gamma_2: (1-\epsilon_2,1] \times \R/\Z \rightarrow
B$ satisfying:
\begin{enumerate}
 \item 
 $\gamma_2 = \gamma_1$ on $(1-\delta_1,1] \times \R/\Z$ and $\gamma_2 = \gamma_0$ on $(1-\epsilon_2,1-\delta_0]
\times \R/\Z$.
 \item If $\rho_2(r,\vartheta) := |\gamma_2(r,\vartheta)|$ then $\partial_r \rho_2(r,\vartheta) < 0$ on
     $(1-\epsilon_2,1) \times \R/\Z$.
 \item $\rho_2 < \eta$ on $(1-\epsilon_2,1] \times \R/\Z$ where $\eta$ is given by STEP 3.
\end{enumerate}

\begin{proof}[Proof of STEP 4]
Define $\rho_j(r,\vartheta) := |\gamma_j(r,\vartheta)|$ on $(1-\epsilon,1]\times \R/\Z$, $j=0,1$. Note that
$\rho_j(1,\vartheta)=0 \ \forall \vartheta$. We claim $\exists \ 0<\epsilon_2<\epsilon$ such that
\begin{equation}\label{prel}
\begin{aligned}
 & \partial_r (\rho_j^2) < 0 \text{ on } (1-\epsilon_2,1) \times \R/\Z \\
 & \wind (\vartheta \mapsto \gamma_j(r,\vartheta)) = 0, \ \forall \ 1-\epsilon_2 < r < 1 \\
 & \rho_j < \eta \text{ on } (1-\epsilon_2,1] \times \R/\Z
\end{aligned}
\end{equation}
for $j=0,1$. In fact, let $C>0$ be a constant so that
\begin{equation}\label{taylorgammaj}
 \norma{\gamma_j (r,\vartheta) + (1-r) \partial_r \gamma_j(1,\vartheta)} \leq C \norma{1-r}^2; \ j=0,1.
\end{equation}
holds for every $(r,\vartheta) \in [1-\epsilon,1]\times \R/\Z$. 
Thus
\[
 \begin{aligned}
  \partial_r (\rho_j^2) (r,\vartheta) &= 2 \left< \gamma_j (r,\vartheta) , \partial_r \gamma_j (r,\vartheta) \right> \\
  &= -2(1-r) \left< \partial_r \gamma_j (1,\vartheta) , \partial_r \gamma_j (r,\vartheta) \right> + O(|1-r|^2)
 \end{aligned}
\]
which is clearly strictly negative if $0<1-r\leq\epsilon_2$ for some $\epsilon_2>0$ small. We used that $\partial_r
\gamma_j (1,\vartheta) \not= 0 \ \forall \vartheta$ for $j=0,1$. This proves the first assertion in~(\ref{prel}). By
the continuity of $\partial_r \gamma_j$ we can make $\epsilon_2$ smaller so that
\begin{equation}\label{smartwind}
  \wind (\vartheta \mapsto \partial_r \gamma_j(r,\vartheta)) = \wind(\vartheta \mapsto \partial_r \gamma_j(1,\vartheta)) = 0, \ \forall 1-\epsilon_2 < r < 1.
\end{equation}
The second assertion of (\ref{prel}) follows since we showed $\left< \gamma_j , \partial_r \gamma_j \right>$ does
not change sign when $r\in (1-\epsilon_2,1)$. The last condition of (\ref{prel}) is easy to achieve.

We now construct the map $\gamma_2$. Since $\rho_j^2(r,\vartheta) > 0$ if $r<1$, we have
$$(\rho_j^2)^{-1}(0) \cap (1-\epsilon_2,1]\times \R/\Z = \{1\} \times \R/\Z.$$ Now choose $0 < s_1
\ll s_0 \ll \epsilon_2$. It follows from (\ref{prel}) and from the implicit function theorem that
$(\rho_j^2)^{-1}(s_j^2) \not= \emptyset$, $s_j^2$ is a regular value of $\rho_j^2|(1-\epsilon_2,1]\times \R/\Z$,
$j=0,1$, and that there are unique smooth functions $r_j : \R/\Z \rightarrow (1-\epsilon_2,1)$ satisfying
\[
 (\rho_j^2)^{-1}(s_j^2) \cap (1-\epsilon_2,1]\times\R/\Z = \{ (r_j(\vartheta),\vartheta) : \vartheta \in \R/\Z \}, \ j=0,1.
\]
Since $0 < s_1 \ll s_0$ we can also assume $r_1(\vartheta) > r_0(\vartheta) \ \forall \vartheta$. Note that
\[
\begin{array}{ccc}
  r_1(\vartheta) < r \leq 1 \Rightarrow \rho_1(r,\vartheta) < s_1 & \text{ and } & 1-\epsilon_2 < r < r_0(\vartheta) \Rightarrow \rho_0(r,\vartheta) > s_0.
\end{array}
\]
We can smoothly define $\alpha_j : (1-\epsilon_2,1) \times \R/\Z \to \R/\Z$ by $\gamma_j = \rho_j
e^{i\alpha_j(r,\vartheta)}$, $j=0,1$. Note that $\gamma_j(r,\vartheta) \not= 0$ since $r<1$. Here (\ref{smartwind})
was strongly used. We choose a smooth function $f : \R \rightarrow \R$ satisfying $f \equiv 0$ on a neighborhood of
$(-\infty,s_1]$, $f \equiv 1$ on a neighborhood of $[s_0,+\infty)$, and $f^\prime \geq 0$ on $\R$. We can find a
smooth function
\[
 \rho_2 : (1-\epsilon_2,1) \times \R/\Z \rightarrow (0,+\infty)
\]
satisfying $\rho_2 = \rho_1 \text{ if } r \geq r_1(\vartheta)$, $\rho_2 = \rho_0 \text{ if } r \leq r_0(\vartheta)$
and $\partial_r \rho_2 < 0$. Define $\gamma_2 = \rho_2 e^{i2\pi(f\circ \rho_2) \alpha_0}
e^{i2\pi(1-f\circ\rho_2)\alpha_1}$. The argument is now complete.
\end{proof}

Let $\epsilon_2$ and $\gamma_2$ be given by STEP 4 and consider the map $\phi : (1-\epsilon,1] \times \R/\Z
\rightarrow \R/\Z \times B$ defined by
\[
 \phi(r,\vartheta) = \left\{
 \begin{aligned}
  & (\vartheta,\gamma_2(r,\vartheta)) \text{ if } r>1-\epsilon_2 \\
  & (\vartheta,\gamma_0(r,\vartheta)) \text{ if } 1-\epsilon<r\leq1-\epsilon_2.
 \end{aligned}
 \right.
\]
One easily checks, using STEP 4, that $\phi$ is a smooth embedding. Define
\begin{equation}\label{somesets}
 \begin{aligned}
  & A = \varphi_0 (\{z\in\D : |z| \leq 1-\epsilon\}) \\
  & B = \Psi^{-1} (\phi((1-\epsilon,1]\times\R/\Z)).
\end{aligned}
\end{equation}
We claim that $A\cap B = \emptyset$. If not then $\Psi(\varphi_0(z^*)) = \phi(r^*,\vartheta^*)$ for some $|z^*| \leq
1-\epsilon$ and some $r^* > 1-\epsilon$. Note that
\[
 \phi((1-\epsilon,1-\epsilon_2] \times \R/\Z) = \Psi(\varphi_0(\{1-\epsilon<|z|\leq1-\epsilon_2\})),
\]
in view of STEP 1 and STEP 4.
Thus we must have $r^* > 1-\epsilon_2$ because $\varphi_0$ is injective. By condition (3) in STEP 4,
$\phi(r^*,\vartheta^*) \in \Psi^{-1}(\R/\Z \times \cl{B_\eta})$. We know from STEP 3 that $\Psi^{-1}(\R/\Z \times
\cl{B_\eta}) \cap \Dcal_0 \subset \varphi_0(\{|z|>1-\epsilon\})$. This proves $\varphi_0(\{|z|\leq1-\epsilon\}) \cap
\varphi_0(\{|z|>1-\epsilon\}) \not= \emptyset$, a contradiction since $\varphi_0$ is 1-1.

Now we claim $\Dcal_1 := A\cup B$ is a smooth embedded disk spanning the orbit $P$ and satisfying conditions (1) and
(2) from Proposition~\ref{perturbdisk}. In fact, the map $F$ obtained in STEP 1 preserves $r$-slices. Then $\Dcal_1$
is a smooth embedded disk since $F$ maps $(1-\epsilon,1-\epsilon_2) \times \R/\Z$ diffeomorphically onto itself and
$\Psi^{-1} \circ \phi \circ F (r,\vartheta) = \varphi_0(re^{i2\pi\vartheta})$ for $(r,\vartheta) \in
(1-\epsilon,1-\epsilon_2) \times \R/\Z$. Condition (1) in Proposition~\ref{perturbdisk} follows from the definition
of $\epsilon_0$ and from $\epsilon<\epsilon_0$. Condition (2) follows from the properties of the map $h$ proved in
Lemma~\ref{nicestrip}. Proposition~\ref{perturbdisk} is now proved.

\subsection{The Bishop Family}

From now on we suppose $\lambda$ is a contact form on a $3$-manifold $M$ and $P$ is a non-degenerate, unknotted and simply covered periodic Reeb orbit. Following~\cite{char1}, we construct a Bishop family of $\jtil$-holomorphic disks in the symplectization $\R \times M$. We orient $x(\R)$ along the Reeb field and let $\Dcal_0 \subset M$ be an embedded disk with $\partial \Dcal_0 = x(\R)$, orientations included. By Proposition~\ref{perturbdisk} we obtain another embedded disk $\Dcal_1$ spanning the orbit $P$ with special properties near the boundary. These properties will be crucial for the proof of Theorem~\ref{existfast}. 

If $h : \D \rightarrow M$ is a smooth embedding such that $h(\D) = \Dcal_1$, we will consider the transverse unknot
$l$ and the disk $F$ given by
\[
  \begin{array}{ccc}
    l := h(\{z \in \D : |z| = 1-\epsilon\}) & \text{and} & F := h(\{z \in \D : |z| \leq 1-\epsilon\})
  \end{array}
\]
where $0<\epsilon\ll1$. We orient $l$ so that $\lambda|_{Tl} > 0$, and $F$ accordingly. If $\epsilon$ is small
enough then $sl(l,F) = sl(P,\Dcal_1)$ and $\xi|_p \not= T_p\Dcal_1, \ \forall p \in \cl{\Dcal_1 \setminus F}$.

\subsubsection{The characteristic singular foliation}

The contact structure $\xi$ induces a (singular) characteristic distribution
\begin{equation}\label{singdist}
 \xi \cap TF \subset TF.
\end{equation}
Generically these are lines since $\xi$ is maximally non-integrable, except at the so-called singular points, where
$\xi = TF$. Given a smooth function $H$ on a neighborhood of $F$, having $F$ inside a regular level set, the
equations
\begin{equation} \label{field}
  \begin{array}{cc}
    i_V\lambda = 0, & i_Vd\lambda = (i_R dH)\lambda - dH
  \end{array}
\end{equation}
define a vector field $V$ tangent to both $F$ and the contact structure $\xi|_{F}$. The zero set of $V$ is precisely
the singular set of $\xi\cap TF$. Clearly $V$ does not vanish over $\partial F=l$ since $l$ is transverse to $\xi$.
All these facts are standard, see~\cite{char1} and~\cite{char2}. The integral lines of $V$ define the so-called
characteristic singular foliation of $F$.

\subsubsection{A convenient spanning disk for $P$}

Let $dvol$ be a smooth $2$-form on $\Dcal_1$ defining the orientation induced by the Reeb vector along $\partial
\Dcal_1 = x(\R)$. We have two symplectic bundles over $F$, namely $(TF,dvol)$ and $(\xi|_F,d\lambda)$, and $V$ is a
section of both. Perhaps after changing $H$ by $-H$ in (\ref{field}), we can assume $V$ points out of $F$ at
$\partial F$. One can, as done in \cite{char1} and \cite{char2}, use the topological information of both bundles in
order to understand the zero set of $V$.

The singular distribution (\ref{singdist}) is said to be Morse-Smale if so is $V$. One can show, see~\cite{93}, that
$F$ can be $C^{\infty}$-perturbed, keeping $\partial F$ fixed, so that its characteristic distribution becomes
Morse-Smale. This perturbation can be arbitrarily $C^\infty$-small. $V$ becomes a non-degenerate section of $TF$
and, consequently, also of $\xi|_F$. We assume this is done and examine a zero $p \in F$ of $V$. Let $o$ and
$o^\prime$ be the orientations of $TF$ and $\xi|F$ induced by $dvol$ and $d\lambda$ respectively. The zero $p$ has
two  associated numbers $\epsilon$ and $\epsilon^\prime$ (both equal to $\pm 1$), namely, the intersection numbers
of $V$ with the zero sections $0_{TF}$ and $0_{\xi}$ of $TF$ and $\xi|_F$, respectively. Let $a_1$ and $a_2$ be the
two eigenvalues of the linearization $DV|_p \in GL(T_pF=\xi|_p)$. The zero $p$ is elliptic if $a_1a_2>0$, or
hyperbolic if $a_1a_2<0$. An elliptic point is nicely elliptic if the eigenvalues are real. These notions relate to
the bundle $TF$. If $p$ is elliptic then $\epsilon=1$, if $p$ is hyperbolic then $\epsilon=-1$. Now, following
\cite{93}, we relate them to the bundle $\xi|_F$. The zero $p$ is positive if $o=o^\prime$ and negative if
$o=-o^\prime$. We have $\epsilon^\prime=1$ if $DV|_p:(T_pF,o)\rightarrow(T_pF,o^\prime)$ preserves orientation,
which will be the case when $p$ is elliptic positive or hyperbolic negative. It is also clear that
$\epsilon^\prime=-1$ if $p$ is elliptic negative or hyperbolic positive. Denote
\[
\begin{aligned}
e^+ & = \text{number of positive elliptic points} \\
e^- & = \text{number of negative elliptic points} \\
h^+ & = \text{number of positive hyperbolic points} \\
h^- & = \text{number of negative hyperbolic points}.
\end{aligned}
\]
In view of the formula (\ref{defselflink0}) for the self-linking number, and of standard degree theory (see the
proof of Proposition~\ref{propnec}), one has
\begin{equation}\label{topcounts}
 \begin{aligned}
  & -sl(l,F) = \sum\epsilon^\prime=e^+-e^--h^++h^- \\
  & 1 = \chi(F) = \sum\epsilon=e^++e^--h^+-h^-
 \end{aligned}
\end{equation}
where the sums are taken over the (non-degenerate) zeros of $V$. We used the fact that $V$ points in the outward
normal direction of the disk $F$ at its boundary and hence pushes it off from $F$.

We state the following proposition which can be extracted from~\cite{char1}.

\begin{propn}[Giroux and Hofer]\label{girouxhofer}
The disk $F$ can be smoothly perturbed, keeping $l = \partial F$ fixed, so that its singular characteristic
distribution satisfies the following properties.
\begin{enumerate}
 \item It is Morse-Smale.
 \item All its elliptic points are nicely elliptic.
 \item It has no trajectories connecting an elliptic point to a hyperbolic point of the same sign.
 \item It has no closed leaves.
 \item $h^+=e^-=0$.
\end{enumerate}
This perturbation can be made $C^0$-arbitrarily small. Moreover, it can be made $C^\infty$-arbitrarily small on a
neighborhood of $\partial F$.
\end{propn}

It is crucial for the above statement that $\xi$ is tight, so that no closed leaves arise when perturbing $F$. The
proof that $h^+ = e^- = 0$ is carried out in section 3 of~\cite{char2}, more precisely, they show that (5) is
implied by (1)-(4).

From now on we assume the disk $F$ was perturbed using Proposition~\ref{girouxhofer}. The compact strip $S :=
(\Dcal_1 \setminus F) \cup \partial F$ has two boundary components, namely $x(\R)$ and $\partial F$. Recall that
$T_pS \not= \xi|_p$, for every $p\in S$. Since the perturbation using the above lemma can be arbitrarily
$C^\infty$-small near $\partial F$, we can obtain a smooth embedded disk, still denoted by $\Dcal_1$, constructed by
joining $F$ with $S$ along $\partial F$, and smoothing it out near $\partial F$. This smoothing process clearly has
support arbitrarily near $\partial F$. The singular distribution $\xi \cap T\Dcal_1$ has the same singular points as
$F$. If $sl(P,\Dcal_0) = -1$ then $sl(l,F) = -1$, by the invariance of the self-linking number under homotopy
through transverse knots. Equations (\ref{topcounts}) and the above construction imply $e^+ = 1$ and $h^- = 0$. This
follows from $e^-=h^+=0$. Summarizing

\begin{theorem}\label{convenientdisk}
Let $M$ be a closed $3$-manifold with a tight contact form $\lambda$. Let $P=(x,T)$ be a non-degenerate, unknotted,
simply covered, periodic Reeb orbit. Suppose there exists an embedded disk $\Dcal_0 \subset M$ satisfying $\partial
\Dcal_0 = x(\R)$ and $sl(P,\Dcal_0) = -1$. Then there exists an embedded disk $\Dcal_1$ spanning $P$, arbitrarily
$C^0$-close to $\Dcal_0$, such that the singular characteristic distribution $\xi \cap T\Dcal_1$ has precisely one
positive, nicely elliptic singular point. In addition, there exists a neighborhood $\OO \subset \Dcal_1$ of
$\partial \Dcal_1$ such that $R_p \not\in T_p\Dcal_1$ for every $p \in \OO \setminus
\partial\Dcal_1$, where $R$ is the Reeb vector associated to $\lambda$.
\end{theorem}

\subsubsection{One last perturbation}\label{lastpert}

Let $\Dcal_1$ be the disk given by Theorem~\ref{convenientdisk} and fix a smooth embedding $f_1:\D\rightarrow M$
such that $f_1(\D)=\Dcal_1$. There exists $\delta>0$ such that
\[
 1-\delta \leq|z| < 1 \Rightarrow \R R_{f_1(z)} \cap T_{f_1(z)}\Dcal_1 = \{0\}.
\]
Consider the set
\[
 X = \{ f \in C^\infty(\D,M) : 1-\delta \leq |z| \leq 1 \Rightarrow f(z)=f_1(z) \}.
\]
Then $X$ is closed in the complete metric space $C^\infty(\D,M)$ endowed with the $C^\infty$ topology. Hence it is
also a complete metric space. For a fixed non-trivial periodic Reeb trajectory $y:\R \rightarrow M$ we define
\[
 X_y := \{ f \in X : y(\R) \subset f(\D) \}.
\]
By the definition of $\delta$ and the properties of $\Dcal_1$
\[
 y(\R) \not= x(\R) \text{ and } f \in X_y \Rightarrow y(\R) \subset f(\{|z|<1-\delta\}).
\]
It is easy to show $X_y^c$ is open and dense in $X$ if $y(\R) \not = x(\R)$. Let us assume every contractible closed
Reeb orbit $P^\prime = (x^\prime,T^\prime)$ is non-degenerate. There are only countably many such $P^\prime$. It
follows from Baire's category theorem that
\[
 \bigcap \left\{ X_{x^\prime}^c : P^\prime=(x^\prime,T^\prime) \text{ is contractible},\ x^\prime(\R) \not= x(\R) \right\}
\]
is dense in $X$. Hence, by an arbitrarily small $C^\infty$-perturbation supported away from $\partial \Dcal_1$, we
may assume that our disk $\Dcal_1$ contains no periodic Reeb orbits other than $x(\R)$.

\subsubsection{Filling by $\jtil$-holomorphic disks}

We now recall a construction done by Hofer, Wysocki and Zehnder in~\cite{char1} and~\cite{char2}. Let $M$, $\xi$,
$\lambda$ and $P$ satisfy the hypotheses of Theorem~\ref{convenientdisk}. Let $e$ be the (unique) singular point of
the characteristic foliation of $\Dcal_1$ and denote $\Dcal_1^* = \Dcal_1 \setminus \{e\}$. Recall that $e$ is a
nicely elliptic singularity. Following \cite{char2}, consider for each $J \in \jcal(\xi,d\lambda|_\xi)$ the boundary
value problem
\begin{equation}\label{bvp}
\left\{
\begin{aligned}
 & \util = (a,u) \in C^\infty(\D,\R\times M) \\
 & \bar\partial_{\jtil}(\util) = 0 \\
 & u(S^1) \subset \Dcal_1 \setminus \{e\} \text{ and } a|_{S^1} \equiv 0 \\
 & u|_{S^1} \text{ winds once positively around } e \text{ in } \Dcal_1 \\
 & \util \text{ is an embedding} \\
 & \util : (\mathbb{D},S^1) \rightarrow (\R\times M, \{0\} \times \Dcal_1^*) \text{ is homotopic to } e
\end{aligned}
\right.
\end{equation}
and set
\begin{equation}\label{solbvp}
 \Mcal(J) := \{ \util \in C^\infty(\D,\R\times M) : \util \text{ solves (\ref{bvp})} \}.
\end{equation}
Here $\jcal(\xi,d\lambda)$ denotes the set of $d\lambda$-compatible complex structures on $\xi$ and $\jtil$
is given by (\ref{almcpxstr}).


\begin{theorem}[Hofer, Wysocki and Zehnder]\label{bishop}
Let the closed $3$-manifold $M$ be equipped with a tight contact form $\lambda$. Assume every contractible closed
Reeb orbit $\hat P$ is non-degenerate, $c_1(\xi)$ vanishes along $\hat P$ and $\mu_{CZ}(\hat P) \geq 3$. Let $P =
(x,T)$ be an unknotted, simply covered periodic Reeb orbit satisfying $sl(P,disk) = -1$. Suppose $\Dcal_1$ is an
embedded disk spanning $P$ such that its characteristic foliation has precisely one positive nicely elliptic
singular point $e$, and that $\Dcal_1 \setminus x(\R)$ contains no periodic Reeb orbits. There exists $J \in
\jcal(\xi,d\lambda|_\xi)$ for which the following holds. The disk $\Dcal_1$ can be smoothly perturbed on an
arbitrarily small neighborhood of $e$ so that $e$ still is the only (nicely elliptic) singularity of the
characteristic foliation and, moreover, there exists a smooth $1$-parameter family $\{\util^t =
(a^t,u^t)\}_{t\in(0,1)}\subset \Mcal(J)$ satisfying:
\begin{enumerate}
 \item $\util^t$ converges in $C^\infty$ to the constant map $(0,e)$ as $t\rightarrow 0$.
 \item There exists $\eta>0$ such that $$ \limsup_{t\rightarrow 1} \left[ \sup \{ |d\util^t(z)| : 1-\eta \leq
     |z| \leq 1 \} \right] < \infty. $$ Here the norms are taken with respect to euclidean metric on $\D$ and to
     any $\R$-invariant riemannian metric on $\R \times M$.
 \item Given any $t_n \rightarrow 1^-$ there exists a subsequence, still denoted $t_n$, such that if we define $$ \Gamma :=
     \{ z \in \D \mid \exists n_j\to \infty \text{ and } \{z_j\} \text{ such that } z_j \rightarrow z \text{ and } |d\util^{t_{n_j}}(z_j)|
     \rightarrow \infty \} $$ then $\#\Gamma = 1$ and $\Gamma \subset \interior{\D}$. Moreover, one can find a
     sequence $g_n \in \mob(\D)$ such that $\util^{t_n} \circ g_n \rightarrow f_P$ in
     $C^\infty_{loc}(\D\setminus\{0\},\R\times M)$ where $f_P$ is the map
     $$ f_P(z) = \left( \frac{T}{2\pi} \log|z| , x\left( \frac{T}{2\pi} \arg z \right) \right). $$ In
     particular, the loops $u^{t_n} \circ g_n (e^{i2\pi t})$ converge to $x_T$ in $C^\infty(S^1,M)$.
 \item $u^t(\D) \cap x(\R) = \emptyset, \ \forall t \in (0,1)$.
\end{enumerate}
\end{theorem}

Here $\mob(\D)$ denotes the group of holomorphic diffeomorphisms of $\D$. The arguments for proving
Theorem~\ref{bishop} are very delicate and we refer to~\cite{char2} for details.

\subsection{Obtaining the fast plane}

Suppose now we are under the assumptions of Theorem~\ref{existfast}. Let $\Dcal_1$ be the smooth embedded disk
spanning the orbit $P = (x,T)$ obtained by Theorem~\ref{convenientdisk}. As explained in Subsection~\ref{lastpert}
we can assume, in addition and without loss of generality, that $\Dcal_1$ contains no periodic Reeb trajectories
other than $x(\R)$. Consequently we can apply Theorem~\ref{bishop} to obtain a $1$-parameter family
$\{\util^t\}_{t\in(0,1)}$ of solutions of (\ref{bvp}). Select a sequence $t_n \rightarrow 1^+$ and denote
$$\util^{t_n} = \util_n = (a_n,u_n).$$ Theorem~\ref{bishop} tells us that, after selecting a subsequence, we can
assume there exists $g_n \in \mob(\D)$ such that $\util_n \circ g_n \rightarrow f_P$ in $C^\infty_{loc}
(\D\setminus\{0\},\R\times M)$. Replacing $\util_n$ by $\util_n \circ g_n$ we assume $g_n = id \ \forall n$.

In the following we denote $f_P(z) = (d,w) \in \R \times M$ and follow~\cite{fols} closely. The bubbling-off point
$0$ has a well-defined mass
\[
 m(0) = \lim_{\epsilon \downarrow 0} \lim_{n\uparrow+\infty} \int_{B_\epsilon(0)} u_n^*d\lambda.
\]
This limit exists since
\[
 m_\epsilon(0) = \lim_{n\uparrow+\infty} \int_{B_\epsilon(0)} u_n^*d\lambda = \lim_{n\uparrow+\infty} \int_{\partial B_\epsilon(0)}u_n^*\lambda = \int_{\partial B_\epsilon(0)}w^*\lambda
\]
is a continuous, non-decreasing and positive function of $\epsilon$. This shows $m(0)=T$. If we define
$\{z_n\}\subset\D$ by
\begin{equation}\label{sequencezn}
a_n(z_n)=\inf_{\D}a_n
\end{equation}
then $z_n\rightarrow0$.
Choose $\delta_n>0$ so that
\[
 \int_{B_{\delta_n}(z_n)}u_n^*d\lambda = m(0)-\frac{\gamma}{2} = T - \frac{\gamma}{2} >0
\]
where $\gamma$ is the constant (\ref{gamma}). It follows easily from the definition of $m(0)$ that $\delta_n
\rightarrow 0^+$. Now define
\[
  \vtil_n(z) = (b_n(z),v_n(z)) := \left(a_n(z_n+\delta_nz)-a_n(z_n+2\delta_n),u_n(z_n+\delta_nz)\right)
\]
for $z\in B_{\delta_n^{-1}(1-|z_n|)}(0)$. Consider
\[
 \hat \Gamma := \{ z \in \C : \exists n_j \to \infty \text{ and } \zeta_j \rightarrow z \text{ such that } |d\vtil_{n_j}(\zeta_j)| \rightarrow \infty \}.
\]
Then $\hat\Gamma \subset \D$ and, up the choice of a subsequence, we may assume $\#\hat\Gamma < \infty$. Thus $\{\vtil_n\}$ has uniform gradient bounds on compact subsets of $\C\setminus\hat{\Gamma}$. We have $C^1_{loc}$-bounds since $\vtil_n(2)\in\{0\}\times M$. Standard elliptic boot-strapping arguments give  $C^\infty_{loc}$-bounds for the sequence $\vtil_n$ on $\C\setminus\hat{\Gamma}$. A particular subsequence, again denoted $\{\vtil_n\}$, must have a $\jtil$-holomorphic limit
\begin{equation}\label{limitv}
 \vtil = (b,v) : \C\setminus\hat{\Gamma} \rightarrow \R\times M
\end{equation}
in $C^\infty_{loc}(\C\setminus\hat{\Gamma},\R\times M)$ satisfying
\begin{gather*}
 E(\vtil) \leq \limsup E(\util_n) \leq area_\lambda(\Dcal_1) < \infty,
\end{gather*}
where
\[
 area_\lambda(\Dcal_1) = \sup \left\{ \left| \int_U d\lambda \right| : U \subset \Dcal_1 \text{ is open} \right\}.
\]
The following important lemma can be extracted from~\cite{char2}. Note that very similar arguments were used to
prove Lemma~\ref{round1} above.

\begin{lemma}[Hofer, Wysocki and Zehnder]\label{propsvtil}
The map $\vtil$ (\ref{limitv}) satisfies the following properties:
\begin{enumerate}
 \item $E(\vtil)>0$ and $\int_{\C\setminus\hat\Gamma} v^*d\lambda>0$.
 \item All punctures $z\in\hat{\Gamma}$ are negative and $\infty$ is the unique positive puncture. If
     $\hat{\Gamma}\not=\emptyset$ then $0\in\hat{\Gamma}$.
 \item $\vtil$ is asymptotic to $P=(x,T)$ at the puncture $\infty$.
\end{enumerate}
\end{lemma}

\begin{proof}
We only show here that $\int_{\C\setminus\hat\Gamma} v^*d\lambda>0$ and that $0\in\hat{\Gamma}$ if $\hat{\Gamma}
\not= \emptyset$. The other properties follow from arguments easily found in the literature, see~\cite{fols} for
example. This is the so-called ``soft-rescaling''. Arguing indirectly, suppose $\hat \Gamma \not= \emptyset$ and $0
\not\in \hat\Gamma$. Then $\exists z_0 \in \C \setminus \hat \Gamma : b(z_0) < b(0) - 2$ since every puncture in
$\hat\Gamma$ is negative. Hence $b_n(z_0) < b_n(0) - 1$ if $n$ is large. Consequently $a_n(z_n+\delta_nz_0) <
a_n(z_n) = \inf_\D a_n$ for large $n$, a contradiction. This shows $0 \in \hat \Gamma$ if
$\hat\Gamma\not=\emptyset$.

Let $\pi : TM \rightarrow \xi$ be the projection along the Reeb direction. If $\pi\cdot dv \equiv 0$ then we can
apply Lemma~\ref{zerodlambda} to find a non-constant polynomial $p$ and a periodic orbit $\hat P = (\hat x,\hat T)$
such that $p^{-1}(0)=\hat \Gamma$ and $\vtil = f_{(\hat x,\hat T)}\circ p$. Here $f_{(\hat x,\hat T)}(\est) = (\hat
Ts,\hat x(\hat Tt))$. Let $k:= \deg p$. Since $\vtil$ is asymptotic to the orbit $P$ at the unique positive puncture
$\infty$, we must have $(x,T) = (\hat x,k\hat T)$. Thus $x = \hat x$ and $k=1$. In fact, if $k\geq2$ then $T$ is not
the minimal period of $x$, contradicting our assumptions on $P$. Consequently we must have $p(z)=Az+D$ and $\hat
\Gamma=\{-D/A\}\subset\D$ for some $A\in\C^*$. Thus $D = 0$ since $0 \in \hat\Gamma$. We now get the following
contradiction
\begin{equation}
 T = \int_{|z|=1} v^*\lambda = \lim_{n\rightarrow+\infty} \int_{|z-z_n|=\delta_n} u_n^*\lambda = T - \gamma/2.
\end{equation}
\end{proof}

\begin{remark}
Further bubbling-off analysis would reveal an entire bubbling-off tree. The first level of this tree has only one
vertex representing the sphere $\vtil$. All this is showed in \cite{fols} via the so-called ``soft-rescaling''. We
shall not make any use of these facts.
\end{remark}

We fix a small tubular neighborhood $U$ of $x(\R)$ and a non-vanishing section
\begin{equation}\label{choiceclass}
  \begin{array}{ccc}
    Z : U \rightarrow \xi|_U & \text{satisfying} & x_T^*Z \in \beta_P.
  \end{array}
\end{equation}
Here $\beta_P$ is the special homotopy class of non-vanishing sections of $x_T^*\xi$ discussed in Remark~\ref{abc}.
Theorem~\ref{bishop} tells us that the sequence of loops $t \mapsto u_n(e^{i2\pi t})$ converges in $C^\infty$ to
$x_T$. Thus $u_n(S^1) \subset U$ for $n\gg1$. By Lemma~\ref{localclass} the sections $t \mapsto Z\circ u_n(e^{i2\pi
t})$ extend to non-vanishing sections
\begin{equation}\label{sectionsZn}
 Z_n : \D \rightarrow u_n^*\xi.
\end{equation}
As usual, $\pi : TM \rightarrow \xi$ denotes the projection along the Reeb direction.

\begin{lemma}\label{pidunnotvanish}
If $n$ is large enough then the sections $\pi \cdot du_n$ do not vanish on $\D$.
\end{lemma}

\begin{proof}
In view of Theorem~\ref{convenientdisk} there exists a neighborhood $\OO \subset \Dcal_1$ of $\partial \Dcal_1$ such
that $R_p \not\in T_p\Dcal_1$ for every $p \in \OO \setminus
\partial\Dcal_1$. There exists $n_0 \in \Z^+$ such that $n\geq n_0$ implies $u_n(S^1) \subset \OO$.
We can, of course, assume $\OO \subset U$. From now on we consider only $n\geq n_0$. For every $z\in S^1$ the linear
map $\pi\cdot du_n(z)$ does not vanish. In fact, $\pi\cdot du_n(z)$ has rank $0$ or $2$, since it satisfies
\begin{equation}\label{crpidun}
 \pi\cdot du_n(z) \cdot i = J \cdot \pi\cdot du_n(z).
\end{equation}
Denote $\partial_\theta u_n(z) = \left. \frac{d}{d\theta} \right|_{\theta=0} u_n(e^{i2\pi \theta}z)$. The strong
maximum principle tells us $\partial_\theta u_n(z) \not= 0$. If the rank is zero then $\partial_\theta u_n(z) \in \R
R \cap T_{u_n(z)}(\Dcal_1)$. This is a contradiction with $u_n(z) \in \OO$ and shows $\text{rank} \ \pi \cdot
du_n(z) = 2$, $\forall z\in S^1$.

Let $x+iy$ be usual Euclidean coordinates on $\D$. Since the $\pi \cdot du_n$ satisfy (\ref{crpidun}) we compute
\begin{equation}\label{wind1}
 \wind(t\mapsto \pi \cdot \partial_xu_n(e^{i2\pi t}), t\mapsto \pi \cdot \partial_\theta u_n(e^{i2\pi t}), J) = \wind(1, ie^{i2\pi t},i) = -1.
\end{equation}
Recall the section $V$ of $\xi|_{\Dcal_1} \cap T\Dcal_1$ from (\ref{field}). We claim $\pi \cdot
\partial_\theta u_n(e^{i2\pi t})$ and $V(u_n(e^{i2\pi t}))$ are not parallel, for every $t\in\R$.
In fact, suppose $\exists c\in\R$ such that $\pi\cdot\partial_\theta u_n(e^{i2\pi t}) = cV(u_n(e^{i2\pi t}))$ for
some $t\in\R$. Then
\[
 \begin{aligned}
  \pi \cdot &\left(\partial_\theta u_n(e^{i2\pi t}) - cV(u_n(e^{i2\pi t}))\right)=0 \\
  & \Rightarrow \partial_\theta u_n(e^{i2\pi t}) - cV(u_n(e^{i2\pi t})) \in T_{u_n(e^{i2\pi t})}\Dcal_1 \cap \R R = \{ 0 \} \\
  & \Rightarrow \partial_\theta u_n(e^{i2\pi t}) = cV(u_n(e^{i2\pi t})) \\
  & \Rightarrow \lambda(u_n(e^{i2\pi t})) \cdot \partial_\theta u_n(e^{i2\pi t}) = 0.
 \end{aligned}
\]
However, by the strong maximum principle, we have $\lambda(u_n(e^{i2\pi t})) \cdot \partial_\theta u_n(e^{i2\pi t})
> 0$. This is a contradiction. We proved
\begin{equation}\label{wind2}
 \wind(t\mapsto \pi \cdot \partial_\theta u_n(e^{i2\pi t}), t\mapsto V(u_n(e^{i2\pi t})) , J) = 0.
\end{equation}

The section $V$ of $\xi|_{\Dcal_1}$ has a unique simple positive zero inside $\Dcal_1$. By Lemma~\ref{localclass}
the section $Z|_{\partial\Dcal_1}$ extends to a non-vanishing section of $\xi|_{\Dcal_1}$. It follows from standard
degree theory that
\begin{equation}\label{wind3}
 \wind(t\mapsto V(x(Tt)), t\mapsto Z(x(Tt)), J) = 1.
\end{equation}
Consequently $\wind \left( t\mapsto V \circ u_n \left( e^{i2\pi t} \right) , t\mapsto Z \circ u_n \left( e^{i2\pi t}
\right) , J \right) \to 1$, proving
\begin{equation}\label{wind4}
 \wind \left( t\mapsto V \circ u_n \left( e^{i2\pi t} \right) , t\mapsto Z \circ u_n \left( e^{i2\pi t} \right) , J \right) = 1 \ \text{if } n\gg 1,
\end{equation}
because winding numbers are $\Z$-valued and $u_n\left( e^{i2\pi t} \right)$ converges in $C^\infty$ to $x(Tt)$. Now
recall the non-vanishing sections $Z_n$ of $u_n^*\xi$ (\ref{sectionsZn}) and compute, for $n\gg1$,
\[
 \begin{aligned}
  \wind & \left( t\mapsto \pi \cdot \partial_x u_n \left( e^{i2\pi t} \right) , t\mapsto Z_n \left( e^{i2\pi t} \right) ,J \right) \\
  &= \wind \left( t\mapsto \pi \cdot \partial_x u_n \left( e^{i2\pi t} \right) , t\mapsto \pi \cdot \partial_\theta u_n \left( e^{i2\pi t} \right) ,J \right) \\
  &+ \wind \left( t\mapsto \pi \cdot \partial_\theta u_n \left( e^{i2\pi t} \right) , t\mapsto V \circ u_n \left( e^{i2\pi t} \right) , J \right) \\
  &+ \wind \left( t\mapsto V \circ u_n \left( e^{i2\pi t} \right) , t\mapsto Z \circ u_n \left( e^{i2\pi t} \right) , J \right) \\
  &= - 1 + 0 + 1 = 0.
 \end{aligned}
\]
The last line follows from (\ref{wind1}), (\ref{wind2}) and (\ref{wind4}). This proves the algebraic count of zeros
of the section $\pi \cdot \partial_xu_n$ on $\D$ is zero. Since zeros count positively (this follows from
(\ref{crpidun})) we conclude $\pi \cdot \partial_xu_n$ never vanishes on $\D$, if $n\gg1$.
\end{proof}

We write $\hat\Gamma = \{z_2,\dots,z_N\}$ and $z_1 = \infty$. The map $\vtil$ was obtained by bubbling-off analysis
of the disks $\util_n$, following a standard procedure described in~\cite{fols}. Further bubbling-off analysis would
reveal an entire bubbling-off tree. By Lemma~\ref{uniqueextension} the section $Z\circ v$ (defined for $|z|$ large) extends to a special non-vanishing section $B$ of $v^*\xi$. Let $\delta$ be the homotopy class of $B$.

%
%
%
%
%
%

\begin{lemma}\label{winfty}
$\wind_\infty(\vtil,\delta,\infty) = 1$.
\end{lemma}

\begin{proof}
We can assume $t\mapsto v(Re^{i2\pi t})$ converges to $t\mapsto x(Tt)$ in $C^\infty$, as $R\rightarrow+\infty$. Let
$z = re^{i2\pi \theta}$ be polar coordinates centered at $0$, and $z-z_n = \rho e^{i2\pi\varphi}$ be polar
coordinates centered at $z_n$. Here $z_n$ is the sequence defined in (\ref{sequencezn}). There exists $R_0 \gg 1$
such that $r\geq R_0$ implies $v(re^{i2\pi\theta}) \in U$, $\pi\cdot dv(re^{i2\pi\theta}) \not= 0$ and
$B(re^{i2\pi\theta}) = Z \circ v(re^{i2\pi\theta})$. This follows from Lemma~\ref{propsvtil}, Theorem~\ref{p1} and
from the asymptotic behavior described in Definition~\ref{behavior}. We compute for $R \gg R_0$
\[
 \begin{aligned}
  & \wind \left( \pi\cdot \left.\frac{d}{dr}\right|_{r=R} v \left( re^{i\theta} \right) , B \left( Re^{i\theta} \right) , J \right) \\
  & = \wind \left( \pi\cdot \left.\frac{d}{dr}\right|_{r=R} v \left( re^{i\theta} \right) , Z \circ v \left( Re^{i\theta} \right) , J \right) \\
  & = \lim_{n\rightarrow+\infty} \wind \left( \pi\cdot \left.\frac{d}{dr}\right|_{r=R} v_n \left( re^{i\theta} \right) , Z \circ v_n \left( Re^{i\theta} \right) , J \right) \\
  & = \lim_{n\rightarrow+\infty} \wind \left( \pi\cdot \left.\frac{d}{d\rho}\right|_{\rho=R\delta_n} u_n \left( z_n+\rho e^{i\varphi} \right) , Z \circ u_n \left( z_n+\delta_n Re^{i\varphi} \right) , J \right).
 \end{aligned}
\]
All windings are taken with respect to angular variables. \\

\noindent {\bf CLAIM:} If $n\gg1$ then
\[
 \begin{aligned}
  \wind & \left( \pi\cdot \left.\frac{d}{d\rho}\right|_{\rho=R\delta_n} u_n \left( z_n+\rho e^{i\varphi} \right) , Z \circ u_n \left( z_n+\delta_n Re^{i\varphi} \right) , J \right) \\
  & = \wind \left( \pi\cdot \left.\frac{d}{d\rho}\right|_{\rho=R\delta_n} u_n \left( z_n+\rho e^{i\varphi} \right) , Z_n \left( z_n+\delta_n Re^{i\varphi} \right) , J \right).
 \end{aligned}
\]
Here $Z_n$ are the sections (\ref{sectionsZn}).

\begin{proof}[Proof of CLAIM]
If $n\gg1$ then the loop $\varphi \mapsto u_n \left( z_n+R\delta_n e^{i\varphi} \right) = v_n \left( R e^{i\varphi}
\right)$ is arbitrarily $C^\infty$-close to the loop $\varphi \mapsto v \left( R e^{i\varphi} \right)$, which can be
made arbitrarily $C^\infty$-close to the loop $t\mapsto x(Tt)$ if $R$ is fixed large enough. By
Lemma~\ref{localclass} the section $\varphi \mapsto Z\left(u_n \left( z_n+R\delta_n e^{i\varphi} \right)\right)$
extends to a non-vanishing section of $u_n^*\xi|_{\{|z-z_n|\leq R\delta_n\}}$. Consequently it does not wind with
respect to $\varphi \mapsto Z_n \left( z_n+\delta_n Re^{i\varphi} \right)$.
\end{proof}

It follows from Lemma~\ref{pidunnotvanish} and from the above claim that
\[
 \begin{aligned}
  \lim_{n\rightarrow+\infty} \wind & \left( \pi\cdot \left.\frac{d}{d\rho}\right|_{\rho=R\delta_n} u_n \left( z_n+\rho e^{i\varphi} \right) , Z \circ u_n \left( z_n+\delta_n Re^{i\varphi} \right) , J \right) \\
  & = \lim_{n\rightarrow+\infty} \wind \left( \pi\cdot \left.\frac{d}{d\rho}\right|_{\rho=R\delta_n} u_n \left( z_n+\rho e^{i\varphi} \right) , Z_n \left( z_n+\delta_n Re^{i\varphi} \right) , J \right) \\
  & = \lim_{n\rightarrow+\infty} \wind \left( \pi\cdot \left.\frac{d}{dr}\right|_{r=1} u_n \left( re^{i\theta} \right) , Z_n \left( e^{i\theta} \right) , J \right) \\
  & = \lim_{n\rightarrow+\infty} \wind \left( \pi\cdot \left.\frac{d}{dr}\right|_{r=1} u_n \left( re^{i\theta} \right) , Z \circ u_n \left( e^{i\theta} \right) , J \right) \\
  & = \lim_{n\rightarrow+\infty} \wind \left( \pi\cdot \left.\frac{d}{dr}\right|_{r=1} u_n \left( re^{i\theta} \right) , \pi \cdot \partial_x u_n \left( e^{i\theta} \right) , J \right) \\
  & + \wind \left( \pi \cdot \partial_x u_n \left( e^{i\theta} \right) , Z \circ u_n \left( e^{i\theta} \right) , J \right) = 1+0 = 1.
 \end{aligned}
\]
\end{proof}

Let $\vtil$ be the finite-energy sphere obtained from Lemma~\ref{propsvtil}. If $\hat{\Gamma}\not=\emptyset$ then
Lemma~\ref{winfty} and Lemma~\ref{notdyn} provide a contradiction to $\mu_{CZ}(P_j) \geq 3 \ \forall j\geq 2$. We
showed $\hat{\Gamma}=\emptyset$. Consequently $\vtil$ is a finite-energy plane satisfying
$\wind_\infty(\vtil)\leq1$. By Lemma~\ref{gauss} we must have $\wind_\infty(\vtil)=1$. By Lemma~\ref{propsvtil},
$\vtil$ asymptotic to the orbit $P$ at the puncture $\infty$. We must now show that $\vtil$ is an embedding. To that
end we argue as in Subsection~\ref{endmain3}. The map $\vtil$ must be an immersion since $\wind_\pi(\vtil) =
\wind_\infty(\vtil) - 1 = 0$ implies $\pi\cdot dv$ does not vanish. Thus, if it is not an embedding, we find
self-intersections. Let $\Delta$ be the diagonal in $\C \times \C$ and consider
\[
 D = \{ (z_1,z_2) \in \C \times \C \setminus\Delta: \vtil(z_1) = \vtil(z_2)\}.
\]
If $D$ has a limit point in $\C\times\C \setminus \Delta$ then we find, using the similarity principle as
in~\cite{mcdsal}, a polynomial $p:\C\rightarrow\C$ of degree at least $2$ and a $\jtil$-holomorphic map
$f:\C\rightarrow \R\times M$ such that $\vtil = f\circ p$. This forces zeros of $d\vtil$, a contradiction. Thus $D$
is closed and discrete in $\C\times\C \setminus \Delta$. By stability and positivity of intersections of
pseudo-holomorphic immersions we find self-intersections of the maps $\util_n$ for large values of $n$ if $D \not=
\emptyset$. This is a contradiction since each $\util_n$ is an embedding. Theorem~\ref{existfast} is proved.


\section{Fredholm theory}\label{fredholm}

Our goal in this section is to prove Theorem~\ref{lemmafredholm}.


\subsection{The asymptotic analysis for $\bar\partial_0$}

In order to understand the asymptotic behavior of finite-energy surfaces near non-removable punctures we need the
fundamental analytical tools from~\cite{props1}. The next three statements can not be explicitly found in the
literature, but the proofs are completely contained in~\cite{props1}, see also~\cite{props4} and the appendix
of~\cite{sftcomp}. We include the arguments in the appendix for completeness.

We denote $Z = \R\times S^1$, $Z^+ = [0,+\infty)\times S^1$ and
\[
 J_0 = \begin{pmatrix} 0 & -I_{\R^n} \\ I_{\R^n} & 0 \end{pmatrix} \in \R^{2n \times 2n}
\]
where $n$ will be clear from the context. We always identify $S^1 = \R/\Z$.

\begin{theorem}\label{asympHWZ}
Fix $l\geq 3$. Let $S(s,t): Z^+ \rightarrow \R^{2k\times 2k}$ be $C^l$ and $N(t) : S^1 \rightarrow \R^{2k\times 2k}$
be smooth. Suppose $N(t)^T = N(t) \ \forall t$ and
\[
 \lim_{s\rightarrow +\infty} \sup_{t\in S^1} e^{ds} \norma{D^\gamma[S(s,t)-N(t)]} = 0 \ \forall |\gamma|\leq l
\]
for some $d>0$. Denote $L_N = -J_0\partial_t - N(t)$. If $\zeta(s,t) \in C^l(\Z^+,\R^{2k})$ satisfies
\[
  \begin{array}{ccc}
    \zeta_s+J_0\zeta_t+S\zeta = 0 & \text{and} & \lim_{s\rightarrow+\infty} \sup_{t\in S^1}e^{ds}\norma{\zeta(s,t)} = 0
  \end{array}
\]
then either $\zeta \equiv 0$ or the following holds: $\exists s_0>1$, $\mu<0$ and a smooth vector
$e:S^1\rightarrow\R^{2k}\setminus\{0\}$ satisfying $L_Ne = \mu e$ such that
\begin{equation}\label{asympbehaviorzeta}
 \zeta(s,t) = e^{\int_{s_0}^s\alpha(\tau)d\tau} \left( e(t) + R(s,t) \right)
\end{equation}
if $s\geq s_0$. The functions $\alpha$ and $R$ are $C^{l-2}$ and satisfy
\begin{equation}\label{remainder}
 \begin{aligned}
  & \lim_{s\rightarrow+\infty} \norma{D^j [\alpha(s)-\mu]} = 0 \ \forall j \leq l-2 \\
  & \lim_{s\rightarrow+\infty} \sup_t \norma{D^\gamma R(s,t)} = 0 \ \forall D^\gamma = \partial^{\gamma_1}_{s}\partial^{\gamma_2}_{t} \text{ with } |\gamma|\leq l-2.
 \end{aligned}
\end{equation}
\end{theorem}

\begin{lemma}\label{megatech}
Let $K : Z^+\rightarrow \R^{2k \times 2k}$ and $X : Z^+ \rightarrow \R^{2k}$ be smooth functions satisfying
\begin{equation}\label{perturbedcr}
 X_s + J_0X_t + K X = 0.
\end{equation}
If
\begin{equation}\label{goodassumptions}
\begin{array}{ccc}
  \sup_{Z^+} \norma{D^\gamma K(s,t)} < \infty, \ \forall \gamma & \text{and} & \lim_{s\rightarrow+\infty} \sup_{t\in S^1} e^{ds}\norma{X(s,t)} = 0
\end{array}
\end{equation}
for some $d\geq0$ then $$ \lim_{s\rightarrow+\infty} \sup_{t\in S^1} e^{ds}\norma{D^\gamma X(s,t)} = 0 \ \forall
\gamma. $$
\end{lemma}

\begin{lemma}\label{megatech2}
Let $X,h : Z^+ \rightarrow \R^{2k}$ be smooth functions satisfying
\begin{equation}\label{perturbedcr2}
 \begin{aligned}
  & X_s + J_0X_t = h \\
  \lim_{s\rightarrow+\infty} \sup_{t\in S^1} e^{ds}\norma{D^\gamma h(s,t)} & = 0, \ \forall \gamma \text{ and } \lim_{s\rightarrow+\infty} \sup_{t\in S^1} \norma{X(s,t)} = 0,
\end{aligned}
\end{equation}
for some $d>0$. Then $\exists r>0$ such that
\[
 \lim_{s\rightarrow+\infty} \sup_{t\in S^1} e^{rs}\norma{D^\gamma X(s,t)} = 0 \ \forall \gamma.
\]
\end{lemma}

\subsection{The asymptotic behavior}\label{abehav}

Let $P = \left( x_0 , T_0 \right)$ be a periodic Reeb orbit with minimal period $T_{min}>0$. Set $k := T_0/T_{min}
\in \Z^+$ and write $P_{min} = \left( x_0 , T_{min} \right)$. We choose a Martinet tube $$ \Psi : U
\stackrel{\sim}{\rightarrow} \R/\Z \times B $$ for $P$, as described in Definition~\ref{martinettube}. Here $U$ is
an open tubular neighborhood of $x_0(\R)$ in $M$ and $B \subset \R^2$ is an open ball centered at the origin. We
equip $\R/\Z \times \R^2$ with coordinates $(\theta,z=(x,y))$. The contact form is represented as $f\lambda_0$,
where $\lambda_0 = d\theta + xdy$ and $f$ satisfies $f|_{\R/\Z \times \{0\}} \equiv T_{min}$ and $df|_{\R/\Z \times
\{0\}} \equiv 0$. The bundle $\xi|_{U}$ is framed by $\partial_x$ and $-x\partial_\theta +
\partial_y$. Setting $e_1 = f^{-1/2}\partial_x$ and $e_2 = f^{-1/2}(-x\partial_\theta +
\partial_y)$ then $\{e_1,e_2\}$ is a $d\lambda$-symplectic for $\xi|_{P_{min}}$. We denote by
\[
 \beta^\Psi_P \in \mathcal{S}_P
\]
the homotopy class of $t\in\R/\Z \mapsto \partial_x|_{(kt,0)}$. The Reeb vector field is locally written as
\begin{equation}\label{reeblocalrep}
 R = (R_1,R_2,R_3) = \frac{1}{f^2} \left( f+xf_x , f_y - xf_\theta , -f_x \right).
\end{equation}
Moreover, $d\lambda = T_{min} dx \wedge dy$ on $\xi|_{P_{min}}$. We lift the variable $\theta$ to the universal
cover $\R$. Since the Reeb flow $\phi_t$ preserves $\lambda$ and $\xi|_{(\theta,0)} \simeq 0 \times \R^2 \ \forall
\theta$ we can write
\begin{equation}\label{dphi}
 D\phi_{T_{min}t} (0,0) \simeq \begin{pmatrix} 1 & 0 \\ 0 & \varphi(t) \end{pmatrix}
\end{equation}
in these coordinates. Here $\varphi : \R \rightarrow \Sp(1)$ satisfies $\varphi(0) = I$ and represents the
linearized Reeb flow restricted to $\xi$ along the curve $x_0(T_{min}t)$. The matrix $D\phi_t$ satisfies
\[
 \frac{d}{dt} D\phi_t(0,0) = DR(T_{min}^{-1}t,0) D\phi_t(0,0).
\]
If we set $Y=(R_2,R_3)$ then $T_{min}D_2Y(t,0) = \dot \varphi \varphi^{-1}$. The periodic orbit $P$ is
non-degenerate if, and only if, $\det \left[ \varphi(k)-I \right] \not= 0$, that is, $\varphi^{(k)} \in \Sigma^*$
where $\varphi^{(k)}(t) = \varphi(kt)$. Moreover, $\mu_{CZ}\left(P,\beta^{\Psi}_P\right) = \mu(\varphi^{(k)})$.

Let $J$ be any $d\lambda$-compatible complex structure on $\xi$. Then $J$ can be represented by a $2\times 2$
matrix
\begin{equation}\label{jlocalrep}
 J(\theta,x,y) = \begin{bmatrix}
                  J_{11} & J_{12} \\
                  J_{21} & J_{22}
                 \end{bmatrix}
\end{equation}
using the frame $\{e_1,e_2\}$. $J$ induces $\jtil$ on $\R \times M$ by (\ref{almcpxstr}).

Suppose $\util = (a,u) : \left( \C,i \right) \rightarrow (\R \times M,\jtil)$ is a finite-energy plane. We write
$\util(s,t) = \util \left( \est \right)$. For the moment we only assume $\exists\sigma,\hat d \in \R$ such that
\begin{equation}\label{minimalassump}
 \begin{aligned}
  & \lim_{s\rightarrow +\infty} u(s,t) = x_0(T_0(t+\hat d)) \text{ in } C^0(S^1,M) \\
  & \lim_{s\rightarrow +\infty} \sup_{t\in S^1} \norma{a(s,t) - T_0s - \sigma} = 0.
 \end{aligned}
\end{equation}
Then $u(s,t) \in U$ for $s$ large enough and we have well-defined maps
\begin{equation}\label{loccomposition}
 \begin{aligned}
 & w (s,t) := \Psi \circ u (s,t)  = (\theta(s,t),z(s,t)); \ z(s,t) := (x(s,t),y(s,t)) \\
 & \wtil(s,t) = (a(s,t),w(s,t)).
 \end{aligned}
\end{equation}
Theorem~\ref{thm93} and (\ref{minimalassump}) imply that $u(s,t) \rightarrow x_0(T_0(t+\hat d))$ in
$C^\infty(S^1,M)$ as $s\rightarrow +\infty$.

We can write $Y = Dz$ where $D(\theta,z) = \int_0^1 D_2Y(\theta,\tau z)d\tau$. The Cauchy-Riemann equations $d\util
\cdot i = ( \jtil\circ\util ) \cdot d\util$ become
\begin{equation}\label{crateta}
 \left\{
  \begin{aligned}
   & a_s - (f\circ w)(\theta_t + xy_t) = 0 \\
   & \theta_s + (f\circ w)^{-1}a_t + xy_s = 0
  \end{aligned}
 \right.
\end{equation}
and
\begin{equation}\label{crz}
 z_s + (J\circ w) z_t + Sz = 0
\end{equation}
where
\begin{equation}\label{matrixS}
 S(s,t) = [a_tI - a_s (J\circ w)]D\circ w.
\end{equation}
We refer to~\cite{props1} for more details.

\begin{lemma}\label{importantdecay}
If $\util$ satisfies (\ref{minimalassump}) then
\begin{equation}
 \lim_{s\rightarrow +\infty} \sup_{t\in S^1} \norma{D^\gamma[\wtil(s,t) - (T_0s+\sigma,kt+k\hat d,0,0)]} = 0 \ \forall D^\gamma = \partial_s^{\gamma_1}\partial_t^{\gamma_2}.
\end{equation}
\end{lemma}

The proof is exactly the same as that of Lemma 2.4 from~\cite{props1} and is omitted.

We need one more computation: $\pi\cdot \partial_s u (s,t) = U(s,t) \partial_x + V(s,t) (-x\partial_\theta +
\partial_y)$ where
\begin{equation}\label{functionsUV}
 \begin{bmatrix} U(s,t) \\ V(s,t) \end{bmatrix} = \begin{bmatrix} x_s \\ y_s \end{bmatrix} + \Delta(s,t)
\end{equation}
and
\[
 \Delta(s,t) = - \frac{\theta_s+xy_s}{f}
 \begin{bmatrix}
  f_y - x f_\theta \\
  -f_x
 \end{bmatrix}.
\]
The values of $f$ and its partial derivatives are evaluated at $w(s,t) = \Psi \circ u(s,t)$. Since $df \equiv 0$ on
$\R/\Z\times \{0\}$ we can use Lemma~\ref{importantdecay} to obtain $\norma{\Delta} = o\left(|z(s,t)|\right)$ as
$s\rightarrow+\infty$.

\begin{lemma}\label{generaldecay}
Suppose $\util$ satisfies (\ref{minimalassump}). Then $\util$ has non-degenerate asymptotics as in
Definition~\ref{nondegpunc} if, and only of, $\exists b>0$ such that $e^{bs}\norma{z(s,t)}$ is bounded. In this case
$\exists r > 0$ such that
\begin{equation}\label{expgeneraldecay}
 \lim_{s\rightarrow +\infty} \sup_{t\in S^1} e^{rs}\norma{D^\gamma[\wtil(s,t) - (T_0s+\sigma,kt+k\hat d,0,0)]} = 0\ \forall D^\gamma = \partial^{\gamma_1}_{s}\partial^{\gamma_2}_{t}.
\end{equation}
\end{lemma}

The technical proof is postponed to the appendix. The arguments are essentially found in section 4 of~\cite{props1}.
We include them here since Lemma~\ref{generaldecay} is not proved in~\cite{props1} and is crucial for our results.

For the remaining of this subsection we assume, for simplicity, that $\sigma=\hat d=0$ in (\ref{minimalassump}).
Write $N(t) = -T_0 J(kt,0)D_2Y(kt,0)$. Lemma~\ref{generaldecay} implies
\begin{equation}\label{SN}
 \lim_{s\rightarrow+\infty} \sup_{t\in \R/\Z} e^{rs}\norma{ D^\gamma[S(s,t)-N(t)] } = 0 \ \forall \gamma
\end{equation}
for some $r>0$. The identity
\[
 ( -J(kt,0)\partial_t\varphi^{(k)} - N(t)\varphi^{(k)} ) [ \varphi^{(k)} ]^{-1} = -T_0 JD_2Y + T_0 JD_2Y = 0
\]
shows that
\begin{equation}\label{repasympop}
 -J(kt,0)\partial_t\varphi^{(k)} - N(t)\varphi^{(k)} = 0.
\end{equation}
By (\ref{repasympop}) the operator
\begin{equation}\label{operatorLN}
 L_N := -J(kt,0)\partial_t - N(t)
\end{equation}
represents the asymptotic operator $A_P$ in the $d\lambda$-symplectic frame $\{e_1,e_2\}$ of $\xi|_P$. Moreover,
$N(t)$ is symmetric with respect to the inner-product $\left< \cdot,-J_0J(kt,0)\cdot \right>$. Here we denoted by
$\left< \cdot,\cdot \right>$ the standard euclidean inner-product on $\R^2$. In fact, it follows from
(\ref{reeblocalrep}) that $T_{min}^2J_0D_2Y(\theta,0) = \nabla^2f(\theta,0) \ \forall \theta$.

The proof of the following corollary of Theorem~\ref{asympHWZ} is found in~\cite{props1}, and is included in the appendix for completeness. 

\begin{cor}\label{oops}
If the finite-energy plane $\util$ has non-degenerate asymptotic behavior at $\infty$ either $z(s,t) \equiv 0$ or it
has the form
\begin{equation}\label{asympbehavior}
 z(s,t) = e^{\int_{s_0}^s\alpha(\tau)d\tau} \left( e(t) + R(s,t) \right)
\end{equation}
where $\mu<0$ and the smooth vector $e:\R/\Z\rightarrow\R^2\setminus\{0\}$ satisfies $L_Ne=\mu e$. Here $L_N$ is the
operator (\ref{operatorLN}). The functions $\alpha$ and $R$ satisfy (\ref{remainder}) for every $l$. In other words, $\mu \in
\sigma(A_P) \cap (-\infty,0)$ and $e(t)$ represents an eigenvector of $A_P$ in the frame $\{e_1,e_2\}$ discussed in
Remark~\ref{homtriv}.
\end{cor}

\begin{lemma}\label{c1}
Define $\varsigma(z) := e_1\circ u(z)$ for $\norma{z}\gg1$ and write $\varsigma(s,t) = \varsigma \left( \est
\right)$. $\pi\cdot du(s,t)$ does not vanish identically if, and only if, $z(s,t) \not =0$ for $s\gg1$. Moreover,
\[
 \lim_{s\rightarrow+\infty} \wind\left( \pi\cdot \partial_su(s,t),\varsigma(s,t),J \right) = \lim_{s\rightarrow+\infty} \wind \left( z(s,\cdot) \right).
\]
\end{lemma}

\begin{proof}
The lemma is an easy consequence of (\ref{functionsUV}) and (\ref{asympbehavior}) since $\norma{\Delta} =
o(|z(s,t)|)$ as $s\rightarrow+\infty$.
\end{proof}

Given any $\beta \in \mathcal{S}_P$ we recall the winding numbers $(\mu,\beta) \in \Z$ associated to eigenvalues of
$A_P$. Those were defined in~\cite{props2} and are discussed in section~\ref{compactness}.

\begin{lemma}\label{windinftyasymp}
If the finite-energy plane $\util = (a,u)$ has non-degenerate asymptotics and if $\beta^\Psi_P = \beta_{\util}$ then
$\wind_\infty(\util) = (\mu,\beta^\Psi_P)$. Here $\mu<0$ is the negative eigenvalue of $A_P$ given by an application
of Corollary~\ref{oops} to the function $z(s,t)$.
\end{lemma}

\begin{proof}
If $\beta^\Psi_P = \beta_{\util}$ then $\varsigma$ extends to a non-vanishing section of $u^*\xi$.
\end{proof}

\begin{lemma}\label{omegalimit}
Let $\util = (a,u)$ be a fast finite-energy plane asymptotic to $P=(x_0,T_0)$ at the (positive) puncture $\infty$.
Suppose $\mu(\util) \geq 3$. Denote by $\phi_t$ the Reeb flow and fix $q\in M\setminus x_0(\R)$. If the
$\omega$-limit ($\alpha$-limit) set of $q$ intersects $x_0(\R)$ then $\forall n\in\Z$ $\exists t>n$ ($t<n$) such
that $\phi_{t}(q) \in u(\C)$.
\end{lemma}

\begin{proof}
This argument can be found in section 5 of~\cite{convex}. Assume $T_{min} = 1$ for simplicity. We only prove the
first assertion. By Definition~\ref{fastplane} we must have $k=1$, $P = P_{min}$ and $\varphi^{(k)} = \varphi^{(1)}
= \varphi$. By Remark~\ref{homtriv} we can assume $\beta^{\Psi}_P = \beta_{\util}$, which implies $\mu(\varphi) \geq
3$. It follows from the geometrical characterization of the Conley-Zehnder index, discussed in
Appendix~\ref{geomindex}, that $\exists r>0$ such that $\arg(v(1)) - \arg(v(0)) \geq 2\pi + r$ for every non-trivial
solution $v$ of (\ref{ed}). The number $r$ is independent of $v(0)\not= 0$. If a non-vanishing smooth vector
$e:\R/\Z \rightarrow \C$ satisfies $\wind(e) \leq 1$ then
\[
 \arg \left(v\bar{e}|_{t=1}\right) - \arg\left(v\bar{e}|_{t=0}\right) \geq r.
\]
Iterating and using as initial condition the vectors $v(n)$, we obtain
\begin{equation}\label{globalsection1}
 \arg \left( v\bar{e}|_{t=k+n} \right) - \arg \left( v\bar{e}|_{t=k} \right) \geq nr
\end{equation}
whenever $k\in \Z$ and $n \in \Z^+$. Here $\arg$ is any continuous choice of the argument.

Lift the $\theta$-variable to the universal cover $\R$ and work with $(\theta,z) \in \R \times \R^2$. Since the Reeb
vector at $x_0(0) = (0,0,0)$ is transversal to $\xi|_{(0,0,0)} = 0 \times \R^2$ we can use the implicit function
theorem to find $t_n \rightarrow +\infty$, $\{k_n\} \subset \Z^+$ and $\{z_n\} \subset \R^2\setminus\{0\}$
satisfying $k_n \rightarrow +\infty$, $z_n \rightarrow 0$ and $\Psi\circ \phi_{t_n}(q) = (k_n,z_n)$. By
differentiability properties of the flow we know that $\forall L>0$ $\exists C>0, n_0\geq1$ such that $n\geq n_0$
implies
\begin{equation}\label{estdelta}
 \Psi \circ \phi_{t+t_n}(q) = (k_n + t,\varphi(t) z_n) + \Delta_n(t), \ \forall t\in[0,L]
\end{equation}
where $|\Delta_n|_{C^1\left([0,L],\R\times\R^2\right)} \leq C |z_n|^2$.

Let $L_N = -J(t,0) \partial_t - N(t)$ be the representation of $A_P$ in the frame $\{e_1,e_2\}$. Suppose
$e:\R/\Z\rightarrow \C\setminus\{0\}$ satisfies $L_Ne=\mu e$ and $\wind(e) \leq 1$, for some $\mu<0$. By
Corollary~\ref{oops}, the function $z(s,t)$ has the form (\ref{asympbehavior}) for such $e(t)$. If $n$ is large
define $s_n$ by $e^{-\int_{s_0}^{s_n}\alpha(r)dr} = |z_n|^{-1}$ and set
\[
 T_n : (t,\zeta) \mapsto (t - k_n , e^{-\int_{s_0}^{s_n}\alpha(r)dr}\zeta) = (t-k_n,|z_n|^{-1}\zeta).
\]
We can assume $z_n/|z_n| \rightarrow z^* \in S^1$. Consider the maps
\begin{gather*}
 h_n(s,t) := T_n \left( \theta(s+s_n,t)+k_n , z(s+s_n,t) \right) \\
 \gamma_n(t) := T_n \circ \Psi \circ \phi_{t+t_n}(q).
\end{gather*}
It follows from (\ref{asympbehavior}), Lemma~\ref{generaldecay} and (\ref{estdelta}) that
\begin{gather*}
 h_n \rightarrow h(s,t) := \left( t , e^{\mu s}e(t) \right) \text{ in } C^\infty_{loc} \\
 \gamma_n(t) \rightarrow \gamma_\infty(t) := \left(t,\varphi(t) z^* \right) \text{ in } C^1_{loc}.
\end{gather*}
It is an immediate consequence of (\ref{globalsection1}) that $\gamma_\infty$ intersects the embedded surface $h
\left( \R \times \R \right)$ at some time $t>0$. We now argue, following~\cite{convex}, that this intersection is
transverse.

For every $\theta \in \R$ the matrix $J(\theta,0)$ (\ref{jlocalrep}) is a complex multiplication in $\R^2$
compatible with the standard area form $dx\wedge dy$. This is so because the frame $\{e_1,e_2\}$ is
$d\lambda$-symplectic over $\R/\Z \times 0$. We find a $1$-periodic smooth map $\theta \mapsto L(\theta) \in \Sp(1)$
satisfying $J_0L(\theta) = L(\theta)J(\theta,0)$. Set $\zeta(t) = L(t) \varphi(t) z^*$ and $E(t) = L(t)e(t)$, $M(t)
= (L(t)N(t) - J_0L^\prime(t))L(t)^{-1}$. Viewing $\zeta$ and $E$ as complex numbers we set $u(t) =
\zeta(t)\cl{E(t)}$. Then $-J_0\zeta^\prime-M\zeta=0$, $-J_0E^\prime-ME=\mu E$ and
\[
 -J_0u^\prime = (M\zeta) \bar{E} - \zeta \cl{(ME)} - \mu u.
\]
If $\gamma_\infty$ intersects $h \left( \R \times \R \right)$ at time $t$ then $u(t) \in \R$, $(M\zeta) \bar{E} -
\zeta \cl{(ME)}$ is purely imaginary and the real part $r$ of $u^\prime(iu)^{-1}$ is equal to $-\mu>0$. However $r$
is the derivative of the argument of $u$. It follows that $\gamma_\infty$ intersects the embedded surface $h \left(
\R \times \R \right)$ transversely at time $t$.

As consequence we obtain (transverse) intersections of $\gamma_n$ with the surface $h_n$ for $n$ large enough,
proving that $\phi_t(q)$ intersects the surface $u(\C)$ for $t>t_n$.
\end{proof}


\subsection{Fredholm Theory}\label{fredtheory}

In this subsection we prove

\begin{lemma}\label{partialfredholm}
Let $\util = (a,u)$ be an embedded fast finite-energy plane with asymptotic limit $P_0 = (x_0,T_0)$. Suppose
$\mu(\util) \geq 3$. Then, for every $l\geq1$, there exists a $C^l$ embedding $f:\C\times B_r(0) \rightarrow
\R\times M$ satisfying properties (1) and (2) listed in Theorem~\ref{lemmafredholm}.
\end{lemma}

$T_0$ is the minimal positive period of $x_0$ since $\util$ is fast. We write $\util(s,t) = (a(s,t),u(s,t)) =
\util\left(\est\right)$ and assume, without loss of generality, that
\begin{equation}\label{normu0}
\begin{aligned}
& \lim_{s\rightarrow+\infty} u(s,t) = x_0(T_0t) \text{ in } C^0(S^1,M) \\
& \lim_{s\rightarrow+\infty} \sup_{t \in S^1} \norma{a(s,t)-T_0s} = 0.
\end{aligned}
\end{equation}

\subsubsection{Normal coordinates at $\util$}

We work on a Martinet tube $\Psi$ (\ref{phi}) defined on an open neighborhood $U$ of the set $x_0(\R)$. We assume
$\beta^\Psi_P = \beta_{\util}$ and use all the notation from Subsection~\ref{abehav}. Define the complex structure
$\hat{J}$ on $\R\times S^1\times B$ by
\[
 (id_\R\times\Psi)^*\hat{J} \equiv \jtil \text{ on } \R\times U.
\]
The basis $\{e_1, e_2\}$ is a $d\lambda$-symplectic frame for $\xi|_{U}$. $\hat J$ is represented by
\begin{equation}\label{jtillocalrep}
  \begin{bmatrix}
    0 & -f & 0 & -xf \\
    R_1 & xf(J_{21}R_2+J_{22}R_3) & -xJ_{21} & x^2f(J_{21}R_2+J_{22}R_3)-xJ_{22} \\
    R_2 & -f(J_{11}R_2+J_{12}R_3) & J_{11} & -xf(J_{11}R_2+J_{12}R_3)+J_{12} \\
    R_3 & -f(J_{21}R_2+J_{22}R_3) & J_{21} & -xf(J_{21}R_2+J_{22}R_3)+J_{22}
  \end{bmatrix}
\end{equation}
with respect to the basis $\{\partial_a,\partial_\theta,\partial_x,\partial_y\}$, where $J = (J_{ij})$ is the matrix
(\ref{jlocalrep}). Note that (\ref{jtillocalrep}) is independent of the first coordinate. For $s$ large enough we
have a well-defined map
\[
 \wtil(s,t) := \left( id_{\R} \times \Psi \right) \left( \util(s,t) \right) = (a(s,t),\theta(s,t),z(s,t) = (x(s,t),y(s,t))).
\]
It follows from Lemma~\ref{generaldecay} that $\hat{J}(s,t) := \hat{J}(\wtil(s,t))$ satisfies
\[
 \lim_{s\rightarrow+\infty} e^{rs} \left|D^\gamma\left[\hat{J}(s,t) - \begin{bmatrix}
                                        0 & -T_0 & 0 & 0 \\
                                        T_0^{-1} & 0 & 0 & 0 \\
                                        0 & 0 & J_{11}(x_0(T_0t)) & J_{12}(x_0(T_0t)) \\
                                        0 & 0 & J_{21}(x_0(T_0t)) & J_{22}(x_0(T_0t))
                                       \end{bmatrix}\right]\right| = 0 \ \forall \gamma
\]
for some $r>0$.

On $\R\times M$ we have projections $\pi_{\R} : \R \times M \rightarrow \R$ and $\pi_{M} : \R \times M \rightarrow
M$. The metric $g^0$ (\ref{norm0}) is $\R$-invariant and $\jtil$ is a pointwise isometry with respect to $g^0$. Fix
two non-vanishing smooth sections $n$ and $m$ of $\xi|_U$ such that $m = J n$. We can assume in addition that
$\{n,m\}$ is $d\lambda$-symplectic since $GL(1,\C)$ is homotopy equivalent to $U(1)$. Then $\tilde n = \pi_{M}^*n$
and $\tilde m = \pi_{M}^*m$ are smooth sections of $\pi_{M}^*\xi$ over $\R \times U$ satisfying $\jtil \tilde n =
\tilde m$. Define
\begin{equation}\label{sectionshat}
 \begin{aligned}
  \hat{n}(s,t) &= d(id_\R\times\Psi)_{\util(s,t)}\cdot \tilde n \circ \util (s,t) \in 0 \times \R^3 \subset \R^4 \\
  \hat{m}(s,t) &= d(id_\R\times\Psi)_{\util(s,t)}\cdot \tilde m \circ \util (s,t) \in 0 \times \R^3 \subset \R^4 \\
  \hat{n}_\infty(t) &= d\Psi|_{x_0(T_0t)} \cdot n(x_0(T_0t)) \in \R^3 \\
  \hat{m}_\infty(t) &= d\Psi|_{x_0(T_0t)} \cdot m(x_0(T_0t)) \in \R^3.
 \end{aligned}
\end{equation}
Lemma~\ref{generaldecay} gives $b>0$ such that
\begin{equation}\label{decaynmhat}
 \begin{aligned}
  & \lim_{s\rightarrow+\infty} \sup_{t \in S^1} e^{bs} \norma{ D^\gamma[\hat n (s,t) - (0,\hat{n}_\infty(t))]} = 0 \ \forall \gamma \\
  & \lim_{s\rightarrow+\infty} \sup_{t \in S^1} e^{bs} \norma{ D^\gamma[\hat m (s,t) - (0,\hat{m}_\infty(t))]} = 0 \ \forall \gamma
 \end{aligned}
\end{equation}

The bundle with fibers $E_{(s,t)} = \pi^*_M\xi|_{\util(s,t)}$ is smooth, $\jtil$-invariant and transverse to
$T\util$ over $[s_0,+\infty)\times S^1$ for a fixed $s_0>0$ large. Thus we can assume there exists a smooth
$\jtil$-invariant subbundle $L$ of $\util^*T(\R\times M)$ such that
\begin{enumerate}
 \item $L=E$ over points $z=\est$ with $s\geq s_0+1$.
 \item $L=N\util$ over points $z=\est$ with $s\leq s_0+1/2$.
\end{enumerate}
Here $N\util$ is the normal bundle to $\util$ with respect to the metric $g^0$. $N\util$ is also $\jtil$-invariant.
Standard degree theory shows that there is precisely one homotopy class $\beta_N \in \mathcal{S}_P$ with the
following property: the section $\tilde n$ extends to a smooth non-vanishing section of $L$ if, and only if, the
section $t\mapsto n(x_0(T_0t))$ is in class $\beta_N$. From now on we assume this is the case and that $\tilde n$
can be extended. We extend $\tilde m$ by $\tilde m = \jtil \tilde n$. The following identity is proved
in~\cite{props3}
\begin{equation}\label{windnp}
 \wind(\beta_N,\beta_{\util},J) = +1.
\end{equation}

Now fix a metric $g$ on $\R\times M$ agreeing with $\left(id_{\R}\times
\Psi\right)^*\left(da^2+d\theta^2+dx^2+dy^2\right)$ on $\R\times U$. Consider $\Phi$ defined by
\begin{equation}\label{mapPhi}
 \begin{aligned}
  \Phi:\C\times B^\prime &\rightarrow \R\times M \\
  (z,v) &\mapsto exp_{\util(z)}(v_1\tilde n(z)+v_2\tilde m(z))
 \end{aligned}
\end{equation}
where $exp$ is the exponential map associated to $g$ and $B^\prime$ is a small open ball around $0\in\R^2$. We need
to examine the map $\Phi$ in more detail. Let us define
\begin{equation}\label{mapF}
 F(s,t,v) = (id_\R\times\Psi) \circ \Phi \circ (\sigma\times id_{\R^2}) (s,t,v)
\end{equation}
where $\sigma(s,t) = \est$. By the asymptotic behavior of $\wtil$, $\hat{n}$ and $\hat{m}$ there exists $b>0$ such that
\begin{equation}\label{exponentialestimates}
 \lim_{s\rightarrow+\infty} e^{bs} |D^\gamma[F(s,t,v)-F_\infty(s,t,v)]| = 0\ \forall\gamma
\end{equation}
where $F_\infty(s,t,v) = (T_0s,t,v_1\hat{n}_\infty(t)+v_2\hat{m}_\infty(t))$. We used Lemma~\ref{generaldecay}. In
particular,
\begin{equation}\label{estDFinfty}
 \lim_{s\rightarrow+\infty} e^{bs} |D^\gamma[DF-DF_\infty]| = 0\ \forall\gamma
\end{equation}
where the smooth matrix
\begin{equation}\label{DFinfty}
 DF_\infty(t,v)       = \begin{bmatrix}
                    T_0 & 0 & 0 & 0 \\
                    0 & 1 & 0 & 0 \\
                    0 & v_1[\hat{n}_\infty^\prime(t)]_1 + v_2[\hat{m}_\infty^\prime(t)]_1 & [\hat{n}_\infty(t)]_1 & [\hat{m}_\infty(t)]_1 \\
                    0 & v_1[\hat{n}_\infty^\prime(t)]_2 + v_2[\hat{m}_\infty^\prime(t)]_2 & [\hat{n}_\infty(t)]_2 & [\hat{m}_\infty(t)]_2
                   \end{bmatrix}
\end{equation}
is independent of $s$.
Define $\bar J = \Phi^*\jtil$ and write
\begin{equation}\label{matrixJbar}
 \bar J(z,v) = \begin{bmatrix} j_1(z,v) & \Delta_1(z,v) \\ \Delta_2(z,v) & j_2(z,v) \end{bmatrix}.
\end{equation}
Note that
\begin{equation}\label{Jbaroverzero}
 \bar{J}(z,0) \equiv \begin{bmatrix} J_0 & 0 \\ 0 & J_0 \end{bmatrix}
\end{equation}
since $\util_t = \jtil\util_s$ and $\tilde m = \jtil \tilde n$.

\begin{lemma}\label{asympJbar}
Consider $\underbar J = F^*\hat J = (\sigma\times id_{\R^2})^*\bar J$. There exists a smooth map $\underbar J_\infty : S^1 \times B^\prime \rightarrow \R^{4\times 4}$ and $b>0$ such that $(\underbar J_\infty)^2 = -I$ and
\begin{equation}\label{estimatesJbar}
 \lim_{s\rightarrow+\infty} \sup_{(t,v) \in S^1\times K} e^{bs} |D^\alpha[\underbar{J}(s,t,v)-\underbar{J}_\infty(t,v)]| = 0 \ \forall \alpha
\end{equation}
for every compact set $K \subset B^\prime$.
\end{lemma}

\begin{proof}
The lemma follows easily from the formula
\[
 \underline J(s,t,v) = DF(s,t,v)^{-1} \cdot \hat J(F(s,t,v)) \cdot DF(s,t,v)
\]
and the estimates (\ref{exponentialestimates}) and (\ref{estDFinfty}), where $\hat J$ is the matrix
(\ref{jtillocalrep}). It is only important to note that $\hat J$ is defined on $\R \times S^1 \times B$ and is
independent of the first variable.
\end{proof}


\subsubsection{The non-linear Cauchy-Riemann equations in normal coordinates}

We look for smooth maps $z\mapsto v(z)$ such that $z\mapsto \Phi(z,v(z))$ has a $\jtil$-invariant tangent space.
This is equivalent to
\begin{equation}\label{badH}
 H(v) := \Delta_2(z,v) + j_2(z,v)\circ dv - dv \circ j_1(z,v) - dv\circ\Delta_1(z,v)\circ dv = 0.
\end{equation}
In view of (\ref{Jbaroverzero}) one can write
\begin{equation}\label{barJdecomp1}
 \begin{aligned}
  & j_1(z,v) = J_0 + \int_0^1 j_1'(z,rv)dr \cdot v  = J_0 + \delta j_1(z,v) \cdot v \\
  & j_2(z,v) = J_0 + \int_0^1 j_2'(z,rv)dr \cdot v  = J_0 + \delta j_2(z,v) \cdot v \\
  & \Delta_1(z,v) = \int_0^1 \Delta_1^\prime(z,rv)dr \cdot v = \delta \Delta_1(z,v) \cdot v \\
  & \Delta_2(z,v) = \int_0^1 \Delta_2^\prime(z,rv)dr \cdot v = \delta \Delta_2(z,v) \cdot v
 \end{aligned}
\end{equation}
where the prime denotes a derivative with respect to the $v$-variable. Analogously we write
\begin{equation}\label{barJdecomp2}
 \begin{aligned}
  \delta \Delta_2(z,v) &= \Delta_2^\prime(z,0) + \int_0^1 \int_0^1 \Delta_2^{\prime\prime}(z,r\tau v)rdrd\tau \cdot v \\
  &= C(z) + \delta \Delta_2^\prime(z,v) \cdot v.
 \end{aligned}
\end{equation}
Plugging (\ref{barJdecomp1}) and (\ref{barJdecomp2}) into (\ref{badH}) we find
\begin{equation}\label{goodH}
 J_0 \cdot dv - dv\cdot J_0 + C \cdot v + W(z,v,dv) \cdot v = 0
\end{equation}
where
\[
 \begin{aligned}
  W(z,v,L) \cdot u &= [ \delta \Delta_2^\prime(z,v) \cdot v ] \cdot u + [ \delta j_2(z,v) \cdot u ] L \\
  &- L [ \delta j_1(z,v) \cdot u ] - L [ \delta \Delta_1(z,v) \cdot u ] L.
 \end{aligned}
\]

Note that $C(z)$ and $W(z,v,L)$ are linear maps $\R^2\to\R^{2\times 2}$ for fixed $(z,v,L)$, and (\ref{goodH}) is an
identity on $2 \times 2$ matrices. Let $\tilde C(s,t)$ be the $2\times 2$-matrix given by $\tilde C(s,t)u = [C(\sigma(s,t))u]\partial_t\sigma(s,t)$. Then by Lemma~\ref{asympJbar}
\begin{equation}\label{propsmapW}
 \lim_{s\rightarrow+\infty} \sup_{t \in S^1} e^{bs} D^\gamma[\tilde C(s,t)-C_\infty(t)] = 0 \ \forall \  D^\gamma=\partial_s^{\gamma_1}\partial_t^{\gamma_2}
\end{equation}
where $b>0$ and $C_\infty(t)$ is some smooth matrix loop. Later we will need the following statement which is an easy consequence of the chain rule and Lemma~\ref{asympJbar}.

\begin{lemma}\label{asympvW}
Fix $m\geq 1$ and assume $v: \C \rightarrow B^\prime$ solves $H(v)=0$ . Suppose that for some $d>0$, $v(s,t) = v(\est)$ satisfies
\[
 \lim_{s\rightarrow+\infty} \sup_{t\in S^1} e^{ds} |D^\gamma v(s,t)| = 0 \ \forall |\gamma|\leq m.
\]
Then the matrix $\tilde W_v(s,t)$ defined by $\tilde W_v(s,t)u := [W(z,v(z),dv(z))\cdot u]\partial_t$, with $z=\sigma(s,t)$ and $\partial_t = \partial_t\sigma(s,t)$, satisfies
\[
 \lim_{s\rightarrow+\infty} \sup_{t\in S^1} e^{ds} |D^\gamma \tilde W_v(s,t)| = 0 \ \forall |\gamma|\leq m-1.
\]
\end{lemma}


\subsubsection{The asymptotic operator and the Fredholm index}

The frame $\{n,m\}$ of $\xi|_{U}$ is $(d\lambda,J)$-unitary, that is, $Jn=m$ and $d\lambda(n,m) \equiv 1$. The
following important lemma is proved in~\cite{props3}.

\begin{lemma}\label{Cinftysymm}
The matrix $C_\infty(t)$ is symmetric. The solution $t \mapsto \psi(t) \in Sp(1)$ of $-J_0 \psi^\prime -
C_\infty\psi = 0$, $\psi(0) = I$, is the linearized Reeb flow restricted to $\xi$ along $t \mapsto x_0(T_0t)$
represented in the basis $\{n,m\}$. In other words, the differential operator $L_{C_\infty} = -J_0\partial_t -
C_\infty(t)$ represents the asymptotic operator $A_P$ in the frame $\{n,m\}$.
\end{lemma}

$H$ is smooth in the following Banach space set-up defined in~\cite{props3}.

\begin{defn}
Fix $l\in\Z^+$, $\alpha\in(0,1)$ and $\delta<0$. The space $C^{l,\alpha,\delta}_0(\C,\R^2)$ is defined
in~\cite{props3}. Introducing cylindrical coordinates $z=e^{2\pi(s+it)}$ and writing $f(s,t)=f(z)$ for $z\not=0$ one
says that $f\in C^{l,\alpha,\delta}_0(\C,\R^2)$ if
\begin{enumerate}
 \item $f\in C^{l,\alpha}(\C,\R^2)$.
 \item $e^{-\delta s}D^\beta f(s,t)\in C^{0,\alpha}(\left[R,+\infty\right)\times S^1,\R^2)\ \forall
     \norma{\beta}\leq l\text{ and }R\in \R$.
 \item $\lim_{R\to+\infty} e^{-\delta s} \Vert D^\beta f(s,t)\Vert_{C^{0,\alpha}(\left[R,+\infty\right)\times
     S^1,\R^2)}=0 \ \forall \ |\beta| \leq l$.
\end{enumerate}
$C^{l,\alpha,\delta}_0(\C,\R^2)$ is a Banach space with the norm
\begin{align*}
 \norma{f}_{l,\alpha,\delta} = & \Vert z\mapsto f(z) \Vert_{C^{l,\alpha}(\mathbb{D},\R^2)} \\
 &+ \Vert (s,t)\mapsto e^{-\delta s}f(\est) \Vert_{C^{l,\alpha}(\left[-1,+\infty\right)\times S^1,\R^2)}.
 \end{align*}
The vector bundle $Y\to \C$ with fiber $Y_z = \{ \R\text{-linear maps } T_z\C \to \R^2\}$ admits a splitting $Y = Y^{1,0} \oplus Y^{0,1}$ into $\C$-linear and $\C$-anti-linear maps. The space $C^{l-1,\alpha,\delta}_0(Y)$ consists of sections $A:\C \to Y$ of class $C^{l-1,\alpha}$ such that $(s,t) \mapsto A(\sigma(s,t))\cdot \partial_t\sigma(s,t)$ belongs to $C^{l-1,\alpha,\delta}_0$ on $\R^+ \times S^1$. $C^{l-1,\alpha,\delta}_0(Y^{0,1})$ is defined analogously.
\end{defn}

By (\ref{estimatesJbar}), equation (\ref{goodH}) defines a smooth map $$ H:C^{l,\alpha,\delta}_0(\C,B^\prime) \to C^{l-1,\alpha,\delta}_0(Y) $$ satisfying $H(0)=0$ and $DH(0)\zeta = J_0  d\zeta - d\zeta  J_0 + C \zeta$. Differentiating the identity $\bar J^2=-I$ one checks that $DH(0)\zeta \in C^{l-1,\alpha,\delta}_0(Y^{0,1}) \subset C^{l-1,\alpha,\delta}_0(Y)$, $\forall \zeta$. There exists a Banach bundle over $C^{l,\alpha,\delta}_0(\C,B^\prime)$ (with fibers modeled on $C^{l-1,\alpha,\delta}_0(Y^{0,1})$), and a smooth Fredholm section $\eta$ such that $\eta(v) = 0 \Leftrightarrow H(v) = 0$. Moreover, $D\eta(0)\zeta = [DH(0)\zeta]^{0,1} = DH(0)\zeta$, $\forall \zeta$. The details are spelled out in section 5 from~\cite{props3}. We will fix $\delta$ later. The following theorem is proved in~\cite{props3}, see also~\cite{mschwarz}.

\begin{theorem}[Hofer, Wysocki and Zehnder]\label{indexformula}
If $\delta \not\in \sigma\left(L_{C_\infty}\right)$ and $\delta<0$ then the vertical derivative $D\eta(0)$ is a Fredholm operator with index $\tilde \mu^\delta \left(L_{C_\infty}\right) + 1 = \mu_{CZ}^\delta (P,\beta_N) + 1= \mu_{CZ}^\delta(P,\beta_{\util}) - 1$.
\end{theorem}

\subsubsection{The choice of $\delta < 0$}

By the inequality $\mu(\util) \geq 3$ we can choose $\delta$ so that
\begin{equation}\label{choice}
 \begin{array}{cccc}
   \delta \in (-\infty,0) \setminus \sigma(A_P), & (\nu^{pos}_\delta,\beta_{\util}) = 2 & \text{and} & (\nu^{neg}_\delta,\beta_{\util}) = 1.
 \end{array}
\end{equation}
This follows from the characterization of $\mu_{CZ}$ explained in Section~\ref{compactness}. For this choice we have
$\mu_{CZ}^\delta(P,\beta_{\util}) = 3$. The index of $D\eta(0)$ is $3 - 1 = 2$, by Theorem~\ref{indexformula}.


\subsubsection{Automatic transversality}

Let $\zeta \in \ker D\eta(0) = \ker DH(0)$. Then $$ 0 = [DH(0)\zeta]\partial_t = \zeta_s+J_0\zeta_t+\tilde C\zeta.
$$ It follows from Theorem~\ref{asympHWZ} that either $\zeta$ vanishes identically or
\[
 \zeta(s,t)=e^{\int_{s_0}^s \tilde\alpha(r)dr}(\tilde e(t)+\tilde R(s,t))
\]
holds for $s\geq s_0 \gg 1$. Here $\tilde e$ is the representation of an eigenvector for $\tilde
\eta\in\sigma(A_P)\cap(-\infty,0)$ in the (unitary) frame $\{n,m\}$. The functions $\tilde\alpha$
and $\tilde R$ satisfy
\[
 \begin{aligned}
  & \lim_{s\rightarrow+\infty} |\tilde\alpha(s)-\tilde\eta| = 0 \\
  & \lim_{s\rightarrow+\infty} \sup_t |\tilde R(s,t)| = 0.
 \end{aligned}
\]
The definition of the space $C^{l,\alpha,\delta}_0(\C,\R^2)$ forces $\tilde{\eta}<\delta$. The map
$\zeta$ satisfies a perturbed Cauchy-Riemann equation so, by the similarity principle, it has only
isolated zeros or vanishes identically, see~\cite{93}. Moreover, each zero counts positively in the
algebraic count of the intersection number with the zero map $z\mapsto(z,0)$. The above asymptotic
behavior tells us that if $\zeta$ does not vanish identically then it does not vanish near
$\infty$. Standard degree theory implies
\[
 0\leq\#\{\text{zeros of }\zeta\}=\lim_{R\rightarrow+\infty}\wind\left(\zeta(Re^{i2\pi t})\right)=\wind(\tilde{e}(t)).
\]
By the choice of $\delta<0$ in (\ref{choice})
\[
 0\leq\wind(\tilde{e}(t))=(\tilde{\eta},\beta_N)=(\tilde{\eta},\beta_{\util})+\wind(\beta_{\util},\beta_N)\leq1-1=0.
\]
Thus $\zeta$ never vanishes or vanishes identically. If there are 3 linearly independent sections
in $\ker D\eta(0)$ then a linear combination of them would have to vanish at some point, which is
impossible by the above discussion. This proves that $D\eta(0)$ is surjective.


\subsubsection{Consequences of the implicit function theorem}

We fix $\delta<0$ as in (\ref{choice}) and an integer $l\geq 4$.

Since the linearization $D\eta(0)$ is surjective the implicit function theorem yields an open neighborhood $\OO \subset
C^{l,\alpha,\delta}_0(\C,\R^2)$ of $0$, an open neighborhood $\mathcal{U} \subset \R^2$ of $0$ and a smooth
embedding $\tau \in \mathcal{U} \mapsto v(\tau) \in \OO$ satisfying
\begin{equation}\label{ift}
 \begin{aligned}
  & \left\{ v(\tau) : \tau \in \Ucal \right\} = \OO \cap H^{-1}(0) \\
  & D\eta(v(\tau)) \text{ is a surjective Fredholm map of index } 2.
 \end{aligned}
\end{equation}

Writing $v(\tau)(z)=v(\tau,z)$, the maps $z\mapsto \Phi(z,v(\tau,z))$ are not necessarily $\jtil$-holomorphic. This
is taken care of in the appendix of~\cite{props3}. Analyzing a suitable Fredholm problem, fixing $0<\epsilon<2\pi$
and possibly making $\Ucal$ smaller, it is possible to find a smooth function
\begin{equation}\label{goodrep}
 \tau \in \Ucal \mapsto (C_\tau,D_\tau,\phi_\tau(z))\in \C\times\C\times\C^{l,\alpha,-\epsilon}_0(\C,\R^2)
\end{equation}
with the following properties: $C_0=1$, $D_0=0$, $\phi_0(z)\equiv0$,
\begin{equation}\label{mappsitau}
 \psi_\tau(z)=C_\tau z+D_\tau+\phi_\tau(z)
\end{equation}
is a diffeomorphism of $\C$, and if we define
\begin{equation}\label{embf}
 \begin{aligned}
  f : \C \times \Ucal & \rightarrow \R \times M \\
  (z,\tau) & \mapsto \Phi(\psi_\tau(z),v(\tau,\psi_\tau(z)))
 \end{aligned}
\end{equation}
then the maps $f(\cdot,\tau)$ are $\jtil$-holomorphic.

For $\tau \in \Ucal$ small, $t\in S^1$ and $s\gg1$ we can define smooth functions $S = S(\tau,s,t) \in \R$ and $T =
T(\tau,s,t) \in S^1$ by
\begin{equation}\label{formulaST}
 e^{2\pi(S+iT)} = \psi_\tau (\est).
\end{equation}
Denoting $\sigma_\tau = T_0(2\pi)^{-1}\log |C_\tau|$ and $d_\tau = (2\pi)^{-1}\arg C_\tau$ one easily verifies
\begin{equation}\label{estST}
 \lim_{s\rightarrow+\infty} e^{-\hat{b}s} \left| D^\gamma[(S,T)-(s+\sigma_\tau/T_0,t+d_\tau)] \right| = 0 \ \forall |\gamma|\leq l
\end{equation}
uniformly in $(\tau,t)$, for $\max\{\delta,-\epsilon\}<\hat{b}<0$. Define
\begin{equation}\label{mapw}
 \begin{aligned}
  \wtil_\tau(s,t) &= (id_\R\times\Psi)\circ f(\est,\tau) \\
  &= (id_\R\times\Psi)\circ\Phi(e^{2\pi(S+iT)},v(\tau,e^{2\pi(S+iT)})) \\
  &= F(S,T,v(\tau,S,T))
 \end{aligned}
\end{equation}
where $F$ is the map (\ref{mapF}). We write
\[
 \wtil_\tau = (a_\tau,w_\tau) = (a_\tau,\theta_\tau,z_\tau = (x_\tau,y_\tau)).
\]
Note that $w_0(s,t) = \Psi\circ u \left( \est \right)$ and $a_0(s,t) = a\left( \est \right)$.

\begin{lemma}\label{regvtau}
The map $f$ is $C^l$, and $v(\tau)$ is smooth $\forall \tau \in \Ucal$.
\end{lemma}

\begin{proof}
$f$ is $C^l$ because of the definition of the topology of $C^{l,\alpha,\delta}_0(\C,\R^2)$ and from the fact that
$\tau \mapsto v(\tau)$ is smooth. By the regularity properties of the Cauchy-Riemann equations each $f(\cdot,\tau)$
is a smooth map. Thus the graph $z \in \C \mapsto (z,v(\tau)(z))$ is a smooth submanifold of $\C \times \R^2$. It
follows that $v(\tau)$ must be smooth.
\end{proof}

\begin{remark}\label{loopconv}
Each $f(\cdot,\tau)$ is a finite-energy plane with energy $E(f(\cdot,\tau)) = T_0$. We also have
\begin{gather*}
 \lim_{s\rightarrow +\infty} \pi_{M} \circ f (e^{2\pi(s+it)},\tau ) = x_0(T_0(t+d_\tau)) \text{ in } C^0(S^1,M) \\
 \lim_{s\rightarrow +\infty} \sup_{t\in S^1} \norma{a_\tau(s,t) - T_0s - \sigma_\tau} = 0.
\end{gather*}
Consequently we can apply Lemma~\ref{importantdecay} to obtain
\begin{equation}\label{estwtil}
 \lim_{s\rightarrow +\infty} \sup_{t\in S^1} \norma{D^\gamma[\wtil_\tau(s,t) - (T_0s+\sigma_\tau,t+d_\tau,0,0)]} = 0 \ \forall\gamma \ \forall \tau.
\end{equation}
\end{remark}


\subsubsection{The asymptotic analysis of fast planes}

It follows from Lemma~\ref{generaldecay} and Corollary~\ref{oops}, as explained in Subsection~\ref{abehav}, that
either $z_0$ vanishes or
\begin{equation}\label{asympz0}
 z_0(s,t) = e^{\int_{s_0}^s\alpha(\tau)d\tau}(e(t)+R(s,t))
\end{equation}
for $s\geq s_0\gg1$. Here $e(t)$ is an eigenvector corresponding to the eigenvalue $\mu<0$ of the asymptotic
operator $A_P$, represented in the trivializing frame $\{e_1,e_2\}$ for $\xi|_{P}$. The functions $\alpha$ and $R$
satisfy (\ref{remainder}).

The homotopy class of $t\mapsto e_1(x_0(T_0t))$ is $\beta_{\util}$. We must have $\mu<\delta$. In fact, if
$\mu>\delta$ then $\wind(e) = (\mu,\beta_{\util}) \geq (\nu^{pos}_\delta,\beta_{\util}) = 2$. Using
Lemma~\ref{windinftyasymp} we have $1 = \wind_\infty(\util) =  \wind(e) \geq 2$. This is absurd.

\begin{lemma}\label{eachtaufast}
Each $f(\cdot,\tau)$ is a fast finite-energy plane.
\end{lemma}

\begin{proof}
We can write
\[
 \begin{aligned}
  \wtil_\tau(s,t) &= F(S,T,v(\tau)(S,T)) \\
  &= \wtil_0(S,T) + \left[ \int_0^1 D_3F(S,T,av(\tau)(S,T)) da \right] \cdot v(\tau)(S,T)
 \end{aligned}
\]
where $\psi_\tau(\est) = e^{2\pi(S+iT)}$. If $z_\tau(s,t)$ is as defined in (\ref{mapw}) then
\[
 \begin{aligned}
  |z_\tau(s,t)-z_0(S,T)|&\leq|\wtil_\tau(s,t)-\wtil_0(S,T)| \\
  &\leq C|v(\tau)(S,T)| \\
  &\leq Ce^{\delta S} \leq C^\prime e^{\delta s}
 \end{aligned}
\]
in view of (\ref{exponentialestimates}), (\ref{estST}) and of the definition of $C^{l,\alpha,\delta}_0(\C,\R^2)$.
Also
\[
 \begin{aligned}
 z_0(S,T) &= z_0(s,t) \\
 &+ \left[ \int_0^1Dz_0((1-a)s+aS,(1-a)t+aT)da \right] \cdot(S-s,T-t).
 \end{aligned}
\]
$\forall \epsilon>0 \ \exists C>0$ such that $\norma{Dz_0(s,t)} \leq Ce^{(\mu+\epsilon)s}$ in view of formula
(\ref{asympz0}). Consequently
\[
 |z_0(S,T)-z_0(s,t)|\leq C\sup_{a\in[0,1]}e^{(\mu+\epsilon) (s+a(S-s))}\leq C^\prime e^{\mu s}
\]
in view of (\ref{estST}). We proved
\begin{equation}\label{est37}
 |z_\tau(s,t)-z_0(s,t)|\leq C^\prime e^{\max\{\mu+\epsilon,\delta\}s}= C^\prime e^{\delta s} \text{ for } s\gg 1,
\end{equation}
since $\mu+\epsilon<\delta$ if $\epsilon>0$ is small enough. Again by $\mu < \delta$ we estimate
\begin{equation}\label{decaydeltaz}
 \norma{z_\tau(s,t)} \leq \norma{z_\tau(s,t) - z_0(s,t)} + \norma{z_0(s,t)} \leq C^{\prime\prime}e^{\delta s}
\end{equation}
for some $C^{\prime\prime}>0$. (\ref{decaydeltaz}) and Lemma~\ref{generaldecay} prove that each plane
$f(\cdot,\tau)$ has non-degenerate asymptotics.

We can apply Corollary~\ref{oops} and conclude that either $z_\tau\equiv0$ or
\begin{equation}\label{asympztau}
 z_\tau(s-T_0^{-1}\sigma_\tau,t-d_\tau)=e^{\int_{s_0}^s \alpha_\tau(r)dr}(e_\tau(t)+R_\tau(s,t))
\end{equation}
for $s\geq s_0\gg1$. Here $e_\tau$ is the representation in the frame $\{e_1,e_2\}$ of an eigenvector of the
asymptotic operator $A_P$ and $\tilde\mu_\tau$ is the corresponding eigenvalue. The functions $\alpha_\tau$ and
$R_\tau$ satisfy
\[
  \lim_{s\rightarrow+\infty} \left[|D^j [\alpha_\tau(s)-\mu_\tau]| + \sup_t |D^\gamma R_\tau(s,t)|\right] = 0 \ \forall j \text{ and } \gamma.
\]
It follows from Lemma~\ref{windinftyasymp} that $\wind_\infty(f(\cdot,\tau))=(\tilde{\mu}_\tau,\beta_{\util})$.
(\ref{asympztau}) and (\ref{decaydeltaz}) imply $\tilde{\mu}_\tau<\delta$. Thus, by our choice of $\delta$ in
(\ref{choice}), we have $\wind_\infty(f(\cdot,\tau)) = (\tilde{\mu}_\tau,\beta_{\util}) \leq 1$. Lemma~\ref{gauss}
proves $\wind_\infty(f(\cdot,\tau))=1$.
\end{proof}


\subsubsection{Further consequences of the asymptotic analysis}

%
%
%
%
%
%
%

Fix $\tau \in \Ucal$ and $\zeta$ in the kernel of $DH(0)$. We can write
\begin{gather*}
 H(v(\tau)) = J_0dv(\tau) - dv(\tau)J_0 + [C + W_\tau] v(\tau) = 0 \\
 DH(0)\zeta = J_0 d\zeta - d\zeta J_0 + C \zeta = 0
\end{gather*}
for a suitable $C^{l-1}$ function $W_\tau(z)$. As a consequence of the definition of the space $C^{l,\alpha,\delta}_0(\C,\R^2)$ and Lemma~\ref{asympvW} we have
\[
 \lim_{s\rightarrow+\infty} \sup_{t \in S^1} e^{-\delta s} |D^\gamma \tilde W_\tau(s,t)| = 0 \ \forall |\gamma| \leq l-1
\]
where $\tilde W_\tau(s,t) u = [W_\tau\left(\est\right)u] \partial_t$, for some $d>0$. Since we chose $l\geq4$, these estimates and
Lemma~\ref{Cinftysymm} allow us to apply Theorem~\ref{asympHWZ} and find that either $v(\tau)$ vanishes identically
or
\[
 v(\tau)(s,t)=e^{\int_{s_0}^s \alpha(r)dr}(e(t)+R(s,t))
\]
for $s\geq s_0\gg1$. The functions $\alpha$ and $R$ are $C^{l-3}$ and satisfy
\[
 \begin{aligned}
  & \lim_{s\rightarrow+\infty} |D^j [\alpha(s)-\eta]| = 0 \ \forall j \leq l-3 \\
  & \lim_{s\rightarrow+\infty} \sup_t |D^\gamma R(s,t)| = 0 \ \forall |\gamma| \leq l-3
 \end{aligned}
\]
where $\eta<0$ is an eigenvalue of $A_P$ and $e$ is the representation of an eigenvector for $\eta$ in the (unitary)
frame $\{n,m\}$. Analogously, either $\zeta$ vanishes identically or
\[
 \zeta(s,t)=e^{\int_{s_0}^s \tilde\alpha(r)dr}(\tilde e(t)+\tilde R(s,t))
\]
where $\tilde e$ is the representation of an eigenvector for $\tilde \eta\in\sigma(A_P)\cap(-\infty,0)$ in the
(unitary) frame $\{n,m\}$. The functions $\tilde\alpha$ and $\tilde R$ satisfy
\[
 \begin{aligned}
  & \lim_{s\rightarrow+\infty} |D^j [\tilde\alpha(s)-\tilde\eta]| = 0 \ \forall j \leq l-3 \\
  & \lim_{s\rightarrow+\infty} \sup_t |D^\gamma \tilde R(s,t)| = 0 \ \forall |\gamma| \leq l-3.
 \end{aligned}
\]
The definition of the space $C^{l,\alpha,\delta}_0(\C,\R^2)$ forces $\eta<\delta$ and $\tilde{\eta}<\delta$.

The maps $v(\tau)$ and $\zeta$ satisfy perturbed Cauchy-Riemann equations. By the similarity principle, they have
only isolated zeros or vanish identically, see~\cite{93}. Moreover, each zero counts positively in the algebraic
count of the intersection number with the zero map $z\mapsto(z,0)$. The above asymptotic behavior tells us that if
$v_\tau$ or $\zeta$ do not vanish identically then they do not vanish near $\infty$. Standard degree theory implies
\[
 \begin{aligned}
  &0\leq\#\{\text{zeros of }v(\tau)\}=\lim_{R\rightarrow+\infty}\wind\left(t\mapsto v(\tau)(Re^{i2\pi t})\right)=\wind(e(t)) \\
  &0\leq\#\{\text{zeros of }\zeta\}=\lim_{R\rightarrow+\infty}\wind\left(t\mapsto \zeta(Re^{i2\pi t})\right)=\wind(\tilde{e}(t)).
 \end{aligned}
\]
By the choice of $\delta<0$ in (\ref{choice})
\[
 \begin{aligned}
  &0\leq\wind(e(t))=(\eta,\beta_N)=(\eta,\beta_{\util})-\wind(\beta_N,\beta_{\util})\leq1-1=0 \\
  &0\leq\wind(\tilde{e}(t))=(\tilde{\eta},\beta_N)=(\tilde{\eta},\beta_{\util})-\wind(\beta_N,\beta_{\util})\leq1-1=0.
 \end{aligned}
\]
Thus $v(\tau)$ and $\zeta$ never vanish or vanish identically. This has important consequences. The map $$ G :
(z,\tau) \in \C\times\Ucal \mapsto (z,v(\tau)(z)) \in \C\times B' $$ is an immersion. In fact, $DG(z,0)$ is
non-singular $\forall z\in\C$ since non-zero elements $\zeta\in\ker DH(0)$ never vanish. Now fix any $\tau
\in\Ucal$. Applying the implicit function theorem and doing the same analysis centered at a given fast plane with
image $\Pi_\tau := \{\Phi(z, v(\tau)(z)):z\in\C\}$ we conclude that $DG(z,\tau)$ is also non-singular $\forall
z\in\C$. Analogously one shows that $G$ is 1-1. In fact, if $\tau\not=0$ then $$ \{(z,v(\tau)(z)):z\in\C\} \cap
\C\times\{0\} = \emptyset $$ since $v(\tau)$ has no zeros. This shows that $v(\tau)(z_1) \not= v(0)(z_2)$ for all
$(z_1,z_2)\in\C^2$. Doing the same analysis by applying the implicit function theorem centered at a given map fast
plane with image $\Pi_{\tau_0}$ we conclude that $v(\tau_0)(z_1) \not= v(\tau_1)(z_2)$ for all $(z_1,z_2)\in\C^2$
when $\tau_0\not=\tau_1$.


We proved that the map $f$ in (\ref{embf}) is an embedding. This concludes the proof of Lemma~\ref{partialfredholm}.
We collect a useful lemma that follows from our arguments so far.

\begin{lemma}\label{embproj}
Let $f$ be the map (\ref{embf}). If $\delta\tau\in\R^2$ is such that $D_2f(z_0,\tau)\cdot\delta\tau=0$ for some
$(z_0,\tau)\in\C\times \Ucal$ then $\delta\tau=0$.
\end{lemma}


\subsection{Completeness}

First we need a particular instance of completeness which will be crucial in the proof of Lemma~\ref{int4} below.

\begin{lemma}\label{rtranslations}
Let $\util$ be an embedded fast finite-energy plane satisfying $\mu(\util)\geq 3$. Let $f$ be the map (\ref{embf})
constructed using the implicit function theorem. If the sequence $\{c_n\} \subset \R$ satisfies $c_n
\rightarrow 0$ then $\exists \tau_n \rightarrow 0$ such that $f(\C,\tau_n) = \left( c_n\cdot\util \right)(\C)$ if
$n\gg1$.
\end{lemma}

The proof is technical but straightforward, we only sketch it here.

\begin{proof}[Sketch of Proof]
Consider the map $\Phi$ in (\ref{mapPhi}). Clearly we can find $v_n:\C\to B'$ of class $C^\infty$ such that
$$ (c_n\cdot \util)(\C) = \{ \Phi(z,v_n(z)) : z\in\C \} $$ and $v_n \to 0$ in $C^\infty_{loc}$. It follows
from $\wind_\infty(c_n\cdot \util) = 1$ that $v_n \in C^{l,\alpha,\delta}_0(\C,\R^2)$. A long and technical
argument, again using the very particular fact that $c_n\cdot \util$ are translations of $\util$, shows that
$v_n\to0$ in the space $C^{l,\alpha,\delta}_0(\C,\R^2)$. The uniqueness statement in the implicit function theorem
concludes the proof.
\end{proof}

The notion of a somewhere injective pseudo-holomorphic curve is well known, we refer the reader to~\cite{props2},
\cite{props3} or~\cite{mcdsal}. The following theorem is proved in~\cite{props2}.

\begin{theorem}[Hofer, Wysocki and Zehnder]
Fix $P=(x,T)\in\p$. If $\util = (a,u) : \C \rightarrow \R\times M$ is a somewhere injective finite-energy plane
asymptotic to $P$ at the (positive) puncture $\infty$ and $u(\C) \cap x(\R) = \emptyset$ then $u : \C \rightarrow M
\setminus x(\R)$ is a proper embedding.
\end{theorem}

\begin{lemma}\label{int2}
Fix $P=(x,T)\in\p$. If $\util = (a,u) : \C \rightarrow \R\times M$ is a fast finite-energy plane asymptotic to $P$
at the (positive) puncture $\infty$ and $u(\C) \cap x(\R) = \emptyset$ then $u : \C \rightarrow M \setminus x(\R)$
is a proper embedding.
\end{lemma}

\begin{proof}
All we need to show is that $\util$ is somewhere injective. Suppose not. In the appendix of~\cite{props2} it is
proved that there exists a somewhere injective $\jtil$-holomorphic curve $f:\C\rightarrow\R\times M$ and a
polynomial $p$ of degree at least $2$ such that $\util = f\circ p$. This forces zeros of $\pi\cdot du$,
contradicting $\wind_\pi(\util) = \wind_\infty(\util) - 1 = 0$.
\end{proof}

\begin{lemma}\label{int3}
Let $P=(x,T)\in\p$ and suppose $\util = (a,u)$ and $\vtil = (b,v)$ are embedded fast finite-energy planes asymptotic
to $P$ at $\infty$. Assume $u(\C) \cap x(\R) = \emptyset$, $v(\C) \cap x(\R) = \emptyset$ and $\beta_{\util} =
\beta_{\vtil}$. Then $u(\C) = v(\C)$ or $u(\C) \cap v(\C) = \emptyset$.
\end{lemma}

The proof of the above lemma is the same as that of Theorem 4.11 in~\cite{props2}. One only needs to check that the assumptions $\wind_\infty(\util)=\wind_\infty(\vtil)=1$ and $\beta_{\util}=\beta_{\vtil}$ play the exact same role as the assumption $\mu(\util)=\mu(\vtil)\leq3$ made in~\cite{props2}. 

\begin{lemma}\label{int4}
Let $\util = (a,u)$ be an embedded fast finite-energy plane asymptotic to the periodic Reeb orbit $P=(x,T)$ at the
(positive) puncture $\infty$. Suppose $\mu(\util)\geq3$. Then $\util(\C) \cap \R \times x(\R) = \emptyset$.
\end{lemma}

\begin{proof}
It follows from $\wind_\pi(\util) = \wind_\infty(\util) - 1 = 0$ that $\pi\cdot du$ does not vanish. Hence $u$ and
$x$ only intersect transversely. Let us assume, by contradiction, that $u(\C) \cap x(\R) \not= \emptyset$. In view
of Definition~\ref{behavior}, the map $u$ has self-intersections and we find intersections of the planes $\util$ and
$c \cdot \util$ for $c \gg 1$. Here $c \cdot \util$ denotes the translation of the plane $\util$ in the
$\R$-direction of $\R \times M$ by $c$, as explained in Remark~\ref{realaction}. Consider, for each $c>0$, the
closed set
\[
 A_c = \left\{ (z,w) \in \C\times\C : \util(z) = c\cdot\util (w) \right\}.
\]
$A_c$ can not accumulate in $\C \times \C$ since otherwise one could argue, using the similarity principle as
in~\cite{mcdsal}, that $\util(\C) = (c \cdot \util)(\C)$. It is not hard to show $\exists R_0 \gg 1$ such that
\begin{equation}\label{asympemb}
 R \geq R_0 \Rightarrow
 \left\{
 \begin{aligned}
  & u^{-1} \left( u \left( \C \setminus B_R(0) \right) \right) = \C \setminus B_R(0) \\
  & u|_{\C \setminus B_R(0)} : \C \setminus B_R(0) \rightarrow M \setminus x(\R) \text{ is an embedding.}
 \end{aligned}
 \right.
\end{equation}
This follows essentially from Lemma~\ref{generaldecay} since $P$ is simply covered.

It follows from (\ref{asympemb}) that given any $\epsilon>0$ one finds a compact $K\subset\C$ such that $A_c\subset
K\times K$ for every $c\geq\epsilon$. This is so since intersections of the plane $\util$ with any of its
translations $c \cdot \util$ induce self-intersections of the map $u$. We can now use the homotopy invariance of
intersection numbers together with positivity of intersections of pseudo-holomorphic maps to conclude that $\util$
intersects $c\cdot\util$ for any $c>0$. Let
\[
 f: \C \times B_r(0) \rightarrow \R\times M
\]
be the embedding (\ref{embf}) obtained by the implicit function theorem. Choose $c_n\rightarrow 0^+$. By
Lemma~\ref{rtranslations} $\exists \tau_n\rightarrow 0$ satisfying $\tau_n \not= 0$ and $\left(c_n\cdot\util\right)
(\C) = f(\C,\tau_n)$ for $n$ large enough. This is an absurd because $f$ is 1-1 and $c_n \cdot \util$ intersects
$\util$ for every $n$.
\end{proof}

\begin{lemma}\label{rep}
Suppose $\util = (a,u)$ and $\vtil = (b,v)$ are embedded fast finite-energy planes asymptotic to the periodic Reeb
orbit $P=(x,T)$ at the (positive) puncture $\infty$. Suppose $\beta_{\util} = \beta_{\vtil}$ and denote $\mu =
\mu(\util) = \mu(\vtil)$. Suppose further that $\mu\geq3$. Then either $u(\C) \cap v(\C) = \emptyset$ or $u(\C) =
v(\C)$. If $u(\C)=v(\C)$ then we find $c\in\R$ and $A,B\in\C$, $A\not=0$, such that
\[
 \util(Az+B) = (b(z)+c,v(z)) = c\cdot\vtil(z) \ \forall z\in\C.
\]
\end{lemma}

\begin{proof}
It is a straightforward consequence of lemmas~\ref{int3} and~\ref{int4} that either $u(\C) \cap v(\C) = \emptyset$
or $u(\C) = v(\C)$. Suppose $u(\C) = v(\C)$. One finds a diffeomorphism $\varphi:\C\rightarrow\C$ such that
$v=u\circ\varphi$ since both $u$ and $v$ are embeddings of $\C$ into $M\setminus x(\R)$. Let $s+it$ be a complex
parameter on $\C$ and compute
\[
 \begin{aligned}
  \pi \cdot (du\circ\varphi) \cdot \varphi_s &= \pi \cdot v_s \\
  &= - (J\circ v) \cdot \pi \cdot v_t \\
  &= - (J\circ u \circ \varphi) \cdot \pi \cdot (du\circ\varphi) \cdot \varphi_t \\
  &= \pi \cdot (du\circ\varphi) \cdot (- i \varphi_t).
 \end{aligned}
\]
The condition $\wind_\pi(\util)=0$ implies that $\pi \cdot du$ is a linear isomorphism from $T_z\C$ to $\xi|_{u(z)}$
for every $z\in\C$. Hence so is $\pi \cdot (du\circ\varphi)$. We conclude from the above computation that
$\varphi_s+i\varphi_t=0$, that is, $\varphi$ is a biholomorphism. It must have the form $\varphi(z)=Az+B$. The
Cauchy-Riemann equations $d\util\cdot i = (\jtil\circ\util) \cdot d\util$ imply $(\lambda\circ u) \cdot du \cdot i =
da$. We compute
\[
 \begin{aligned}
  b_s &= (\lambda\circ v) \cdot v_t \\
  &= (\lambda\circ u\circ\varphi) \cdot (du\circ\varphi) \cdot \varphi_t \\
  &= (\lambda\circ u\circ\varphi) \cdot (du\circ\varphi) \cdot i \cdot \varphi_s \\
  &= (da\circ \varphi) \cdot \varphi_s = (a\circ\varphi)_s.
 \end{aligned}
\]
Analogously, $b_t = (a\circ\varphi)_t$. Hence $\exists c\in\R$ such that $b + c = a\circ\varphi$.
\end{proof}

We now complete the proof of Theorem~\ref{lemmafredholm}. Let $f$ be the $C^l$-embedding (\ref{embf}) and write $f =
(h,g) \in \R \times M$. Fix $\tau_0 \in B_r(0)$. Let $\vtil_n = (b_n,v_n)$ be a sequence of embedded fast
finite-energy planes with $\mu(\vtil_n) = \mu$ and $\vtil_n \rightarrow f(\cdot,\tau_0)$ in $C^\infty_{loc}$. We
find $(B_n,\tau_n) \rightarrow (0,\tau_0)$ such that $\vtil_n(0) = f(B_n,\tau_n)$. By Lemma~\ref{rep} we find $c_n
\in \R$, $A_n\rightarrow 1$ such that $(c_n\cdot \vtil_n)(z) = f(A_nz+B_n,\tau_n) \ \forall z\in\C$ since $v_n(0) =
g(B_n,\tau_n)$. Note that $c_n = h(B_n,\tau_n) - b_n(0) = 0$. The conclusion follows.



\section{Open book decompositions}\label{cons}

In this section we prove Theorem~\ref{ob2}. Recall the families $\Lambda(H,P)$ in (\ref{setLambda}) and the
$C^\infty_{loc}$-closed subfamilies $\Lambda_k(H,P)$ defined in (\ref{setLambdak}). From now on we assume the
contact form $\lambda$ and the periodic orbit $P$ satisfy the hypotheses of Theorem~\ref{ob2}.


\subsection{Local foliations}\label{localfols}

\begin{lemma}\label{unique}
Fix a point $(r,q)\in\R\times M$. If $\{\util,\vtil\} \subset \Lambda_k(\{(r,q)\},P)$ for some $k\geq 3$ then there
exists $\theta\in\R$ such that $\util(e^{i\theta}z) = \vtil(z) \ \forall z\in\C$.
\end{lemma}

\begin{proof}
The sets $u(\C)$ and $v(\C)$ intersect at the point $q\in M$. Lemma~\ref{rep} provides constants $A\in\C^*$ and
$B\in\C$ such that $\util(Az+B) = \vtil(z)$ for all $z\in\C$. The definition of $\Lambda(\{(r,q)\},P)$ implies $B=0$
and $A\in S^1$.
\end{proof}

Let $\hat x : \R \rightarrow M$ be a Reeb trajectory. We denote
\begin{equation}\label{setLambdat}
 \Lambda_t^k = \Lambda_k \left( {\{(0,\hat x(t))\}},P \right) \ \forall t \in \R.
\end{equation}
It follows from Lemma~\ref{int4} that $\hat x (\R) \cap x(\R) = \emptyset$ if $\Lambda^k_t \not= \emptyset$ for some
$t\in\R$ and $k\geq3$. In the following $I_r(a)$ denotes the open interval $(a-r,a+r)$.

\begin{lemma}\label{l}
Suppose $I$ is an open real interval, $k\geq3$ and $\util = (a,u) : I \times \C \rightarrow \R \times M$ is a
$C^1$-map such that $\util(t,\cdot) \in \Lambda_t^k \ \forall t \in I$. Then $\forall t_0 \in I \ \exists \epsilon >
0$ such that $u$ is 1-1 on $I_\epsilon(t_0) \times \C$.
\end{lemma}

\begin{proof}
By the inverse function theorem we can find $\rho > 0$ and $\epsilon > 0$ small so that
\begin{equation}\label{difloc}
 (t,z) \in I_\epsilon(t_0) \times B_\rho(0) \mapsto u(t,z) \in M
\end{equation}
is a $C^1$-embedding onto an open set of $M$ since $\wind_\pi(\util(t_0,\cdot)) = 0$ and $u(t,0) = \hat x(t)$. We
now argue indirectly. Take sequences $t_n,t^*_n\rightarrow t_0$ and $z_n,z^*_n \in \C$ such that $(t_n,z_n) \not=
(t^*_n,z^*_n)$ and $u(t_n,z_n) = u(t^*_n,z^*_n)$. By Lemma~\ref{int2} $t_n = t^*_n$ implies $z_n = z^*_n$. Consequently we must have $t_n \not= t^*_n$. By Lemma~\ref{rep} we find a sequence $\{\zeta_n\} \subset \C\setminus\{0\}$ such that $u(t_n,\zeta_n) = \hat x(t^*_n)$ since $u(t_n,\C) = u(t^*_n,\C) \ \forall n$. The sequence $\{\zeta_n\}$ is bounded since the compact sets $x(\R)$ and $\hat x \left( \cl{I_\epsilon(t_0)} \right)$ do not intersect. Suppose $\liminf|\zeta_n|=0$. Then we find a subsequence $\zeta_{n_j} \rightarrow 0$ as $j\rightarrow+\infty$. For $j$ large enough the points $(t_{n_j},\zeta_{n_j})$ and $(t^*_{n_j},0)$ are distinct points in $I_\epsilon(t_0) \times B_\rho(0)$ and satisfy $u(t_{n_j},\zeta_{n_j}) = \hat x(t^*_{n_j}) = u(t^*_{n_j},0)$. This is in contradiction to the injectivity of the map (\ref{difloc}). This proves $\liminf_n \norma{\zeta_n} > 0$. After selecting a subsequence we can assume $\zeta_n\rightarrow \zeta^* \not= 0$. Consequently $u(t_0,\zeta^*) = u(t_0,0) = \hat x(t_0)$. However, $u(t_0,\cdot)$ is an embedding by Lemma~\ref{int2}. This is a contradiction.
\end{proof}

\begin{lemma}\label{localemb}
Suppose $k\geq3$ and $\util_0 \in \Lambda_{t_0}^k$. $\forall l\geq1 \ \exists \epsilon>0$ and a $C^l$-map
\[
 \util = (a,u) : I_\epsilon(t_0) \times \C \rightarrow \R\times M
\]
satisfying the following properties:
\begin{enumerate}
 \item $\util(t_0,\cdot) \equiv \util_0$ and $\util(t,\cdot)\in \Lambda_t^k\ \forall t\in I_\epsilon(t_0)$.
 \item $u : I_\epsilon(t_0) \times \C \rightarrow M \setminus x(\R)$ is an embedding.
 \item Suppose $\vtil_n$ are embedded fast finite-energy planes asymptotic to $P$ at $\infty$, as in
     definition~\ref{behavior}. If $\vtil_n \rightarrow \util_0$ in $C^\infty_{loc}$ and $\mu(\vtil_n) = k \
     \forall n$ then, for $n$ large enough, one finds $\{A_n,B_n\}\subset\C$, $\{r_n\}\subset\R$ and $\{t_n\}\subset\R$ such that $A_n\rightarrow1$, $B_n\rightarrow0$, $r_n\rightarrow0$, $t_n\rightarrow t_0$ and
 \[
  \util(t_n,z) = \left( r_n \cdot \vtil_n \right) \left( A_nz+B_n \right) \ \forall z\in\C.
 \]
\end{enumerate}
\end{lemma}

\begin{proof}
We split the proof into two steps.
\\

\noindent \textbf{STEP 1:} A $C^l$-map $\util = (a,u) : I_\epsilon(t_0) \times \C \rightarrow \R\times M$ satisfying
(1) exists.

\begin{proof}[Proof of STEP 1]
Consider the $C^l$-embedding $f:\C\times B_r(0) \hookrightarrow \R\times M$ given by Theorem~\ref{lemmafredholm}
such that $f(\cdot,0)\equiv\util_0$. We find $\epsilon>0$ small and unique $C^l$-curves
$\tau:I_\epsilon(t_0)\rightarrow B_r(0)$ and $\zeta:I_\epsilon(t_0)\rightarrow\C$ satisfying $\tau(t_0)=0$,
$\zeta(t_0)=0$ and
\[
 (0,\hat x(t)) = f(\zeta(t),\tau(t)) \ \forall t\in I_\epsilon(t_0).
\]
We must have $\tau^\prime(t_0)\not=0$ since $\wind_\pi(f(\cdot,\tau(t_0)) = 0$. Write $f=(h,g)\in\R \times M$ and
define
\[
 F(r,t) := \int_{|z-\zeta(t)|\leq r} g(\cdot,\tau(t))^*d\lambda.
\]
Then $F(1,t_0)=T-\gamma$ and $\partial_rF>0$. By the implicit function theorem we find, possibly after making
$\epsilon$ smaller, a $C^l$-function $t\in I_\epsilon(t_0)\mapsto r(t)$ such that $r(t_0)=1$ and $F(r(t),t)\equiv
T-\gamma$. Define $\util : I_\epsilon(t_0) \times \C \rightarrow \R \times M$ by the formula
\[
 \util(t,z) = (a(t,z),u(t,z)) = f(r(t)z+\zeta(t),\tau(t)).
\]
If $\epsilon$ is small enough then $\util$ is a $C^l$-embedding. Clearly $\util(t,\cdot) \in \Lambda_t^k \ \forall
t\in I_\epsilon(t_0)$.
\end{proof}

By Lemma~\ref{l} we can assume $u$ is injective.
\\

\noindent \textbf{STEP 2:} The map $u$ is an immersion.

\begin{proof}[Proof of STEP 2]
Suppose not. We find $(z^*,t^*) \in \C\times I_\epsilon(t_0)$ and a non-zero vector $(\delta z^*,\delta t^*) \in
\C\times \R$ such that
\[
 D\util(t^*,z^*) \cdot (\delta t^*,\delta z^*) = (c,0) \in \R\times 0 \subset T_{\util(t^*,z^*)}\left(\R\times M\right)
\]
We must have $c\not=0$ since $\util$ is an immersion. Denote $\util(t^*,\cdot)$ by $\util_{t^*} =
(a_{t^*},u_{t^*})$. We claim that $\delta t^* \not= 0$. If not then $\R\times 0$ is a (real) line in
$T_{z^*}\util_{t^*}$. The Cauchy-Riemann equations $d\util_{t^*} \cdot i = \jtil \cdot d\util_{t^*}$ imply that
$z^*$ is a zero of $\pi \cdot du_{t^*}$ since $\jtil$ maps $\R\times 0$ onto $0 \times \R R$. This contradicts
$\wind_\pi(\util_{t^*}) = 0$. From now on we assume $\delta t^* = 1$ and denote $Q(z,t) = (Q_1,Q_2) =
(r(t)z+\zeta(t),\tau(t))$. Setting $Q^* = Q(z^*,t^*)$ we compute
\begin{equation}\label{vec1}
 \begin{aligned}
  (c,0) &= D\util(t^*,z^*) \cdot (1,\delta z^*) \\
  &= D \left( f \circ Q \right)|_{(z^*,t^*)} \cdot (\delta z^*,1) \\
  &= D_1f|_{Q^*} \cdot \left[ D_1Q_1|_{(z^*,t^*)} \cdot \delta z^* + D_2Q_1|_{(z^*,t^*)} \right] \\
  &+ D_2f|_{Q^*} \cdot \left[ D_1Q_2|_{(z^*,t^*)} \cdot \delta z^* + D_2Q_2|_{(z^*,t^*)} \right] \\
  &= D_1f|_{Q^*} \cdot \left[ r(t^*) \delta z^* + r^\prime(t^*)z^* + \zeta^\prime(t^*) \right] + D_2f|_{Q^*} \cdot \left[ \tau^\prime(t^*) \right]
 \end{aligned}
\end{equation}

By Theorem~\ref{lemmafredholm} we find $C^l$-curves $A(s)\in\C^*$, $B(s)\in\C$ and $\sigma(s)\in\R^2$ defined on
$I_\delta(0)$ ($\delta>0$ small), satisfying $\sigma(0) = \tau(t^*)$, $A(0) = r(t^*)$, $B(0) = \zeta(t^*)$ and
\[
 (a_{t^*}(z) + cs , u_{t^*}(z)) = f(A(s)z+B(s) , \sigma(s)) \ \forall z\in\C.
\]
We compute
\begin{equation}\label{vec2}
 \begin{aligned}
  (c,0) &= \left. \frac{d}{ds} \right|_{s=0} f \left( A(s)z^*+B(s),\sigma(s) \right) \\
  &= D_1f|_{Q^*} \cdot \left[ A^\prime(0)z^* + B^\prime(0) \right] + D_2f|_{Q^*} \cdot \sigma^\prime(0).
 \end{aligned}
\end{equation}
Subtracting (\ref{vec2}) from (\ref{vec1}) we obtain
\[
 D_1f|_{Q^*} \cdot \Delta + D_2f|_{Q^*} \cdot \left[ \tau^\prime(t^*) - \sigma^\prime(0) \right] = 0
\]
where $\Delta = r(t^*) \delta z^* + r^\prime(t^*)z^* + \zeta^\prime(t^*) - A^\prime(0)z^* - B^\prime(0)$. Since the
images of $D_1f$ and of $D_2f$ are transversal, the section
\[
 z \mapsto D_2f(z,\tau(t^*)) \cdot \left[ \tau^\prime(t^*) - \sigma^\prime(0) \right]
\]
has a zero at $z = r(t^*)z^* + \zeta(t^*)$. By Lemma~\ref{embproj} we have $\tau^\prime(t^*) - \sigma^\prime(0) =
0$. We compute
\begin{equation}\label{vec3}
 \begin{aligned}
  (0,R(\hat x(t^*))) &= \left. \frac{d}{dt} \right|_{t=t^*} f \left( \zeta(t),\tau(t) \right) \\
  &= D_1f|_{(\zeta(t^*),\tau(t^*))} \cdot \zeta^\prime(t^*) + D_2f|_{(\zeta(t^*),\tau(t^*))} \cdot \tau^\prime(t^*)
 \end{aligned}
\end{equation}
and
\begin{equation}\label{vec4}
 \begin{aligned}
  (c,0) &= \left. \frac{d}{ds} \right|_{s=0} f \left( B(s),\sigma(0) \right) \\
  &= D_1f|_{(B(0),\sigma(0))} \cdot B^\prime(0) + D_2f|_{(B(0),\sigma(0))} \cdot \sigma^\prime(0).
 \end{aligned}
\end{equation}
Subtracting (\ref{vec4}) from (\ref{vec3}) we obtain
\[
 (-c,R(\hat x(t^*))) \in \text{image} \left( D_1f|_{(\zeta(t^*),\tau(t^*))} \right) = T_{\util_{t^*}(\zeta(t^*))}\util_{t^*}
\]
and this is again in contradiction to $\wind_\pi(\util_{t^*}) = 0$.
\end{proof}

We proved $\util$ satisfies (1) and (2). Let us write $\util_0 = (a_0,u_0)$ and $\vtil_n = (b_n,v_n)$. We find
$\tau_n\rightarrow0$ such that $f(\C,\tau_n) = \vtil_n(\C)$ in view of Theorem~\ref{lemmafredholm}. Since $v_n
\rightarrow u_0$ in $C^\infty_{loc}$ there exists $t_n \rightarrow t_0$ such that $\hat x(t_n) \in v_n(\C) \ \forall
n$. We used that $\hat x$ intersects the embedded surface $u_0(\C)$ transversally at $\hat x(t_0)$. Consequently
$\hat x(t_n) \in v_n(\C) \cap u(t_n,\C)$ and $v_n(\C) = u(t_n,\C)$ by Lemma~\ref{int3}. Lemma~\ref{rep} provides
$\{r_n\}\subset\R$, $\{A_n\}\subset\C^*$, $\{B_n\}\in\C$ such that $(r_n \cdot \vtil_n)(A_nz+B_n) = \util(t_n,z) \
\forall z\in\C$.
\end{proof}

\begin{lemma}\label{bundle}
Let $I$ and $J$ be two open real intervals. Fix $k\geq3$ and $l\geq 1$. Let $\util : I \times \C \rightarrow \R
\times M$ and $\vtil : I \times \C \rightarrow \R \times M$ be $C^l$ maps such that $\util(t,\cdot) \in \Lambda_t^k
\ \forall t \in I$ and $\vtil(t,\cdot) \in \Lambda_t^k \ \forall t \in J$. If $I \cap J \not= \emptyset$ then there
exists a $C^l$ function $t \in I \cap J \mapsto \theta(t) \in \R$ such that
\[
 \util(t,z) = \vtil(t,e^{i2\pi\theta(t)}z)
\]
for every $(t,z) \in I \cap J \times \C$.
\end{lemma}

\begin{proof}
Note that $\hat x(t) = u(t,0) = v(t,0) \ \forall t \in I \cap J$. By Lemma~\ref{unique} we can find a function
$\theta(t)$ defined on $I \cap J$ satisfying the desired equation. It remains only to show that we can arrange
$\theta$ to be $C^l$. By Lemma~\ref{l} $\exists \eta > 0$ small such that $v : I_\eta(t_0) \times \C \rightarrow M
\setminus x(\R)$ is injective. Since $\wind_\pi(\vtil(t_0,\cdot)) = 0$ we can invoke the inverse function theorem
and assume, without loss of generality, that $\exists \rho > 0$ such that $v: I_\eta(t_0) \times B_\rho(0)
\rightarrow M$ is an embedding onto an open neighborhood of $\hat x(t_0)$. If $0<\epsilon\ll\eta$ then we find $C^l$
functions $\alpha:I_\epsilon(t_0)\rightarrow I_\eta(t_0)$ and $\beta:I_\epsilon(t_0)\rightarrow\C$ such that
$u(t,\epsilon) = v(\alpha(t),\beta(t))$. By lemmas~\ref{int4} and~\ref{int3} we know that $u(t,\C) = v(t,\C) \
\forall t\in I_\epsilon(t_0)$. Thus $\alpha(t) \equiv t$ since $v$ is 1-1 on $I_\eta(t_0)\times \C$ and $\beta(t)
\not= 0$ since $v(t,0) = u(t,0) \not= u(t,\epsilon)$ and each $u(t,\cdot)$ is 1-1. Let $\theta(t) :=
\frac{1}{2\pi}\arg \beta(t)$ be a $C^l$ choice of argument. It follows from Lemma~\ref{unique} that
$\vtil(t,e^{i2\pi\theta(t)}z) = \util(t,z)$ for every $(t,z) \in I_\epsilon(t_0)\times \C$. We proved $\theta$ can
be locally chosen of class $C^l$.
\end{proof}

The following statement is an easy consequence of lemmas~\ref{localemb} and~\ref{bundle}.

\begin{lemma}\label{localinj}
Suppose $I$ is an open real interval and $l\geq1$. Let $\util = (a,u) : I \times \C \rightarrow \R \times M$ be a
$C^l$ map such that $\util(t,\cdot) \in \Lambda_t^k \ \forall t \in I$ for some $k\geq3$. Then $\forall t^\prime \in
I \ \exists \epsilon>0$ such that $u : I_\epsilon(t^\prime) \times \C \rightarrow M \setminus x(\R)$ is a
$C^l$-embedding.
\end{lemma}

Later we will need the following claim.

\begin{claim}\label{claimob}
Suppose $I$ and $J$ are two open real intervals. Let $\util : I \times \C \rightarrow \R \times M$ and $\vtil : J
\times \C \rightarrow \R \times M$ be $C^l$ maps such that $\util(t,\cdot) \in \Lambda^k_t \ \forall t \in I$ and
$\vtil(t,\cdot) \in \Lambda^k_t \ \forall t \in J$. Suppose $v(s_0,0) = u(t_0,z_0)$. Then there exists $\epsilon>0$
and $C^l$ functions $\alpha:I_\epsilon(s_0) \rightarrow I$ and $\zeta:I_\epsilon(s_0) \rightarrow \C$ such that
$I_\epsilon(s_0) \subset J$, $\alpha^\prime > 0$, $\alpha(s_0) = t_0$, $\zeta(s_0) = z_0$ and $v(s,0) =
u(\alpha(s),\zeta(s)) \ \forall s\in I_\epsilon(s_0)$.
\end{claim}

\begin{proof}
By Lemma~\ref{localinj} $\exists\epsilon_0 > 0$ such that $u:I_{\epsilon_0}(t_0) \times \C \rightarrow M$ is an
embedding onto an open subset of $M$. If $\epsilon>0$ is small enough then $v \left( I_\epsilon(s_0) \times \{0\}
\right) \subset u \left( I_{\epsilon_0}(t_0) \times \C \right)$. Hence there are unique $\alpha$ and $\zeta$ as in
the statement satisfying $v(s,0) = u(\alpha(s),\zeta(s)) \ \forall s\in I_\epsilon(s_0)$. The inequality
$\alpha^\prime>0$ is trivial to check.
\end{proof}


\subsection{Proof of Theorem~\ref{ob2}}

From now on we denote $k = \mu(\util_0) \geq 3$ and proceed in three steps.

\subsubsection{Foliating $M \setminus x(\R)$}

The Reeb flow preserves the volume form $\lambda\wedge d\lambda$. By Poincar\'e's recurrence almost every point of
$M$ is a recurrent point, that is, it belongs to its own $\alpha$ and $\omega$ limit sets.

Let $\util_0 = (a_0,u_0)$ be the embedded fast finite-energy plane asymptotic to $P = (x,T)$ as in
Theorem~\ref{ob2}. After translation in the $\R$-direction we can assume $\util_0(0) = (0,q) \in \{0\} \times M$.
Let $\hat x : \R \rightarrow M$ be the Reeb trajectory satisfying $\hat x(0) = q$. As in Subsection~\ref{localfols}
we denote
\[
 \Lambda_t^k = \Lambda_k(\{(0,\hat x(t))\},P) \ \forall t \in \R.
\]
We can reparametrize $\util_0$ in order to achieve $\util_0 \in \Lambda^k_0$. In view of Lemma~\ref{localemb} we can
assume, without loss of generality, that $q$ is a recurrent point. We will prove

\begin{lemma}\label{foliating}
There exists $L>0$, $\delta>0$ and a $C^l$ map
\[
 \util = (a,u) : (-\delta,L+\delta) \times \C \rightarrow \R \times M
\]
such that
\begin{enumerate}
 \item $\util(0,\cdot) \equiv \util_0$ and $\util(t,\cdot) \in \Lambda^k_t \ \forall t\in (-\delta,L+\delta)$.
 \item $u(0,\C) = u(L,\C)$ and each $u(t,\cdot) : \C \rightarrow M \setminus x(\R)$ is an embedding transversal
     to the Reeb vector.
 \item $u:(0,L) \times \C \rightarrow M \setminus \left( u_0(\C) \cup x(\R) \right)$ is a $C^l$ orientation
     preserving diffeomorphism.
\end{enumerate}
\end{lemma}

By Lemma~\ref{localemb} $\exists\delta > 0$ and $\util : I_\delta(0) \times \C \rightarrow \R \times M$ of class
$C^l$ satisfying $\util(0,\cdot) \equiv \util_0$ and $\util(t,\cdot) \in \Lambda^k_t \ \forall t \in I_\delta(0)$.
In view of Lemma~\ref{int4} we must have $q \not\in x(\R)$, implying $x(\R) \cap \hat x(\R) = \emptyset$. We define
a set
\begin{equation}\label{setB}
 B \subset (0,+\infty)
\end{equation}
by requiring that $t \in B$ if, and only if, there exists $\delta>0$ and a $C^l$ map $\util : (-\delta,t) \times \C
\rightarrow \R \times M$ such that $\util(0,\cdot) \equiv \util_0$ and $\util(s,\cdot) \in \Lambda^k_s \ \forall s
\in (-\delta,t)$. $B\not=\emptyset$ by our remarks so far.

\begin{lemma}\label{continuation}
$\sup B = +\infty$.
\end{lemma}

\begin{proof}
We argue by contradiction and assume $\tau = \sup B < +\infty$. Fix an increasing sequence $t_n \rightarrow \tau^-$
and fast planes $\util_n \in \Lambda^k_{t_n}$. We can assume, in view of the assumptions of Theorem~\ref{ob2}, that
$\util_n \rightarrow \wtil$ in $C^\infty_{loc}$ for some $\wtil \in \Lambda^k_\tau$. We used that the compact set $H
= \{0\} \times \hat x \left( [0,\tau] \right)$ satisfies $H \cap \R \times x(\R) = \emptyset$. Applying
Lemma~\ref{localemb} to the plane $\wtil$ we find $\eta>0$ and a $C^l$ map
\[
 \vtil = (b,v) : I_\eta(\tau) \times \C \rightarrow \R \times M
\]
such that $\vtil(\tau,\cdot) \equiv \wtil$ and $\vtil(t,\cdot) \in \Lambda^k_t$ for every $t \in I_\eta(\tau)$. Now
we fix $t_1 \in (\tau-\eta,\tau) \cap B$ and a $C^l$ map $\util : (-\delta,t_1) \times \C \rightarrow \R \times M$
such that $\util(0,\cdot) \equiv \util_0$ and that $\util(t,\cdot) \in \Lambda^k_t$ for every $t\in(-\delta,t_1)$.
By Lemma~\ref{bundle} we can find a $C^l$ function $\theta : (\tau-\eta,t_1) \rightarrow \R$ such that $\tau-\eta <
t < t_1$ implies $\util(t,z) = \vtil\left(t,e^{i2\pi\theta(t)}z\right) \ \forall z \in \C$. Fix a number $0 < \rho
\ll t_1 - \tau + \eta$ and a smooth function $\phi \in C^\infty(\R)$ such that $\phi \equiv 1$ on
$(-\infty,\tau-\eta+\rho)$ and $\phi \equiv 0$ on $(t_1 - \rho,+\infty)$. The map $(t,z) \mapsto \vtil\left(
t,e^{i2\pi\phi(t)\theta(t)}z \right)$ defined on $(\tau-\eta,t_1)\times \C$ agrees with $\util(t,z)$ on
$(\tau-\eta,\tau-\eta+\rho) \times \C$ and with $\vtil(t,z)$ on $(t_1-\rho,t_1) \times \C$. Thus it can be used to
glue $\util$ and $\vtil$ and to provide a map $\tilde f : (-\delta,\tau+\eta) \times \C \rightarrow \R \times M$
satisfying $\tilde f(t,\cdot) \in \Lambda^k_t \ \forall t\in(-\delta,\tau+\eta)$ and $\tilde f(0,\cdot) \equiv
\util_0$. This is a contradiction to the definition of $\tau$.
\end{proof}

The point $q\in u_0(\C)$ is a recurrent point. We know that $\hat x$ intersects $u_0$ transversely at $q$ since
$\wind_\pi(\util_0) = 0$. Hence $\exists\tau > 0$ such that $\hat x (\tau) \in u_0(\C)$. By Lemma~\ref{continuation}
we find $\delta>0$ and a $C^l$ map $\util : (-\delta,\tau+\delta)\times\C \rightarrow \R\times M$ such that
$\util(t,\cdot) \in \Lambda^k_t \ \forall t\in(-\delta,\tau+\delta)$ and $\util(0,\cdot) \equiv \util_0$. It follows
from lemmas~\ref{int4} and~\ref{int3} that $u(\tau,\C) = u_0(\C)$. Consider the set
\[
 E = \left\{ t \in (0,\tau] : u(t,\C) = u_0(\C) \right\}.
\]
It follows easily from Lemma~\ref{localinj} that
\[
 L := \inf E > 0
\]
Moreover $E$ is closed in $(0,\tau]$. In fact, let $\{t_n\} \subset E$ satisfy $t_n \rightarrow t \in (0,\tau]$.
There exists a unique sequence $\{z_n\} \subset \C$ such that $\hat x(t_n) = u(t_n,0) = u_0(z_n) \ \forall n$. We
must have $\sup_n |z_n| < \infty$ since $\util_0$ is asymptotic to $P$ and $\hat x\left( [0,\tau] \right) \cap x(\R)
= \emptyset$. Hence we may assume $z_n \rightarrow z^*$. It follows that $\hat x(t) = u(t,0) = u_0(z^*)$. Thus $t
\in E$ in view of lemmas~\ref{int4} and~\ref{int3}, concluding the proof that $E$ is closed. Consequently $L \in E$.

\begin{lemma}\label{surj}
$u \left( (0,L) \times \C \right) = M \setminus \left( u_0(\C) \cup x(\R) \right)$.
\end{lemma}

\begin{proof}
Set $\mathcal{U} := M \setminus \left( u_0(\C) \cup x(\R) \right)$. Clearly $\mathcal{U}$ is open and connected. It
follows from lemmas~\ref{int4} and~\ref{int3} and from the definition of $L$ that $u \left( (0,L) \times \C \right)
\subset \mathcal{U}$. We claim $u \left( (0,L) \times \C \right)$ is closed in $\mathcal{U}$. In fact, suppose the
sequence $y_n$ satisfies $\forall n \exists t_n \in (0,L)$ such that $y_n \in u(t_n,\C)$ and $y_n \rightarrow y$ for
some $y \in \mathcal{U}$. One finds a unique sequence $B_n \in \C$ such that $u(t_n,B_n) = y_n$.

Let $\W$ be a $S^1$-invariant neighborhood of the (discrete) set of $S^1$-orbits
\[
 \{c \in C^\infty(S^1,M) : \exists \tilde P  = (\tilde x,\tilde T) \in \p^* \text{ such that } c \in \tilde P \text{ and } \tilde T\leq T \}
\]
We can assume no connected component of $\W$ contains loops in distinct classes of $C^\infty(S^1,M)/S^1$.
Let $\W_P$ be the component containing the class $P$. We can assume, without loss of generality, that $c \in \W_P
\Rightarrow c(S^1) \cap \cl{\{y_n\}} = \emptyset$.

Consider the cylinders $Z_t (s,\vartheta) := \util \left( t,e^{2\pi(s+i\vartheta)} \right)$ for $t \in [0,L]$ and
set $s_0 = (2\pi)^{-1}\log 2$. The definition of $\Lambda^k_t$ implies that
\begin{enumerate}
 \item $\int_{\{s\}\times S^1} Z_t^*\lambda \geq T-\gamma \ \forall t\in[0,L], \ \forall s\geq s_0$.
 \item $E(Z_t) = T \ \forall t\in[0,L]$.
 \item $\int_{[s_0,+\infty)\times S^1} Z_t^*d\lambda \leq \gamma \ \forall t\in[0,L]$.
\end{enumerate}
Applying Lemma~\ref{longcyl} we find $r_0 \gg 1$ such that $|z| \geq r_0 \Rightarrow \util(t,z) \not\in \cl{\{y_n\}}
\ \forall t\in [0,L]$. It follows that $\sup_n |B_n| < \infty$. We can assume $B_n \rightarrow B$ for some $B \in
\C$ and $t_n \rightarrow t^*$ for some $t^* \in [0,L]$. If $t^* = 0$ then $y_n = u(t_n,B_n) \rightarrow u(0,B)$
contradicting $y \not\in u_0(\C)$. Hence $t^* \not= 0$. Analogously $t^* \not= L$ and we have $t^* \in (0,L)$. It
follows that $y_n = u(t_n,B_n) \rightarrow u(t^*,B)$ and $(t^*,B) \in (0,L) \times \C$. We proved $y \in u \left(
(0,L) \times \C \right)$. Thus $u \left( (0,L) \times \C \right)$ is closed in $\mathcal{U}$. That $u \left( (0,L)
\times \C \right)$ is open in $\mathcal{U}$ follows easily from Lemma~\ref{localinj}.
\end{proof}

By the previous lemma we have $u \left( [0,L] \times \C \right) = M \setminus x(\R)$.

\begin{lemma}\label{inj}
$u$ is 1-1 on $(0,L) \times \C$.
\end{lemma}

\begin{proof}
Suppose $(t_0,z_0) \not= (t_1,z_1)$ are points of $(0,L) \times \C$ satisfying $u(t_0,z_0) = u(t_1,z_1)$. It follows
from lemmas~\ref{int4} and~\ref{int3} that $u(t_0,\C) = u(t_1,\C)$. In view of Lemma~\ref{int2} every $u(t,\cdot)$
is 1-1. Thus $t_0 = t_1$ implies $z_0 = z_1$. Consequently we may assume $0 < t_0 < t_1 < L$. In view of
Lemma~\ref{localinj}, we can cover $[0,L]$ by finitely many open intervals $\{I_k\}$ such that $u:I_k\times \C
\rightarrow M$ is an embedding for each $k$. Examining the Lebesgue number of this cover we find $\eta>0$ such that
$u:F\times\C \rightarrow M$ is an embedding whenever $F$ is a subinterval of $[0,L]$ of length at most $\eta$. It
follows that $t_0 \leq t_1 - \eta$.

Fix $r>0$ very small. Applying Claim~\ref{claimob} to the maps $\util : (t_0-r,L] \times \C \rightarrow M$ and
$\util : (t_1-r,L] \times \C \rightarrow M$ we find $\epsilon>0$ small, $C^l$ functions
$\alpha:I_\epsilon(t_1)\rightarrow (t_0-r,L]$ and $\zeta: I_\epsilon(t_1) \rightarrow \C$ such that
$\alpha^\prime>0$, $\alpha(t_1) = t_0$ and $u(t,0) = u(\alpha(t),\zeta(t)) \ \forall t\in I_\epsilon(t_1)$. By the
properties of the number $\eta$ we must have $\alpha(t) \leq t-\eta \ \forall t\in I_\epsilon(t_1)$. Now fix
$t_1<b_0<t_1+\epsilon$, $t_0<e<\alpha(b_0)$ and consider the set
\[
 D = \left\{ b \in (t_1,L] : \exists d \in [e,b-\eta] \text{ such that } u(b,0) \in u(d,\C) \right\}.
\]
Clearly $b_0 \in D$ since $d=\alpha(b_0)$ works for $b_0 \in (t_1,L]$. It follows easily from Claim~\ref{claimob}
that $D$ is open in $(t_1,L]$.

Let $B := \sup D$. We claim $B\in D$ and $B = L$. In fact, take $\{b_n\} \subset D$, $b_n \rightarrow B^-$. We find
$d_n \in [e,B-\eta]$ and $z_n\in\C$ such that $u(b_n,0) = u(d_n,z_n)$. Arguing as in the proof of Lemma~\ref{surj}
we find $r_0 \gg 1$ such that $|z| \geq r_0 \Rightarrow \util(t,z) \not\in \hat x([0,L]) \ \forall t\in [0,L]$. Thus
$\sup |z_n| < \infty$ since $u(b_n,0) = \hat x(b_n)$. Consequently we can assume, after selecting a subsequence,
that $d_n \rightarrow d \in [e,B-\eta]$ and $z_n \rightarrow z^* \in \C$. We conclude $u(B,0) = u(d,z^*)$, proving
$B \in D$. Suppose $B < L$. Since $D$ is open in $(t_1,L]$ we find an element of $D$ larger than $B$, contradicting
the definition of $B$. Thus $L = \sup D \in D$ and $\exists d \in [e,L-\eta]$ such that $u(L,0) \in u(d,\C)$. By
lemmas~\ref{int3} and~\ref{int4} we have $u(d,\C) = u(L,\C) = u(0,\C) = u_0(\C)$. This contradicts the definition of
$L$.
\end{proof}

The proof of Lemma~\ref{foliating} is complete.


\subsubsection{Reparametrizing the foliation}

We now glue the foliation $$ \util : (-\delta,L+\delta) \times \C \rightarrow \R \times M $$ given by
Lemma~\ref{foliating} near the ends to produce a closed $S^1$-family of planes.

Let $z_0 \in \C$ be defined by $u(0,z_0) = u(L,0)$. Using Claim~\ref{claimob} we find $\epsilon>0$ small and $C^l$
functions $\alpha:I_\epsilon(L) \rightarrow \R$, $\zeta:I_\epsilon(L) \rightarrow \C$ satisfying $\alpha^\prime>0$,
$\alpha(L) = 0$, $\zeta(L) = z_0$ and $u(t,0) = u(\alpha(t),\zeta(t)) \ \forall t\in I_\epsilon(L)$.
\\

\noindent {\bf Fact:} There exist $r>0$, $L-r<c<d<L$ and an increasing diffeomorphism
\[
 F : (L-r,L) \rightarrow (\alpha(L-r),0)
\]
such that $F(t) = \alpha(t)$ on $(L-r,c)$ and $F(t) = t-L$ on $(d,L)$.

\begin{proof}[Proof of Fact]
Fix $0<\sigma<\alpha^\prime(L)$. We make some \emph{a priori} arguments. For fixed $r,c$ and $d$ as above, choose
$\phi:\R\rightarrow[0,1]$ smooth such that $\phi \equiv 0$ on $(-\infty,c)$ and $\phi \equiv 1$ on $(d,+\infty)$.
The function $G(t) := \alpha(t)(1-\phi(t)) + \sigma(t-L)\phi(t)$ satisfies the desired properties with $t-L$
replaced by $\sigma(t-L)$ if $r>0$ is small enough. Clearly $G$ can be modified to obtain $F$ as required.
\end{proof}

Let $\rho(t) > 0$ be the $C^l$ function defined on $(L-r,c)$ characterized by the identity
\[
 \int_{B_{\rho(t)}(\zeta(t))} u(\alpha(t),\cdot)^*d\lambda \equiv T-\gamma.
\]
It follows that the $C^l$ family of (fast) finite-energy planes $\vtil(t,z) = (b(t,z),v(t,z))$ defined on $(L-r,c)
\times \C$ by
\[
 \begin{aligned}
  & b(t,z) := a(\alpha(t),\rho(t)z+\zeta(t)) - a(\alpha(t),\zeta(t)) \\
  & v(t,z) := u(\alpha(t),\rho(t)z+\zeta(t))
 \end{aligned}
\]
satisfies $\vtil(t,\cdot) \in \Lambda^k_t \ \forall t\in(L-r,c)$. By Lemma~\ref{bundle} we find
$\theta:(L-r,c)\rightarrow\R$ of class $C^l$ such that $\vtil (t,e^{i2\pi\theta(t)}z ) = \util(t,z)$ $\forall (t,z)
\in (L-r,c) \times \C$. Fix $L-r<x<y<c$ and choose $C^l$ functions $t\mapsto A(t) \in \C \setminus\{0\}$, $t\mapsto
B(t) \in\C$ and $t\mapsto\tau(t) \in \R$ defined on $(L-r,c)$ satisfying the following properties. If $\psi_t(z) :=
A(t)z+B(t)$ then
\begin{enumerate}
 \item $\psi_t(z) = \rho(t)e^{i2\pi\theta(t)}z+ \zeta(t)$ and $\tau(t) = a(F(t),\zeta(t))$ if $t\in(L-r,x]$.
 \item $\psi_t(z) = z$ and $\tau(t) = 0$ if $t\in[y,c)$.
\end{enumerate}

We finally define $\tilde U : \R / L\Z \times \C \rightarrow \R \times M$ by
\begin{equation}\label{mapU}
\tilde U (t,z) =
\left\{
\begin{aligned}
 & \util (t,z) \text{ on } [0,L-r] \times \C \\
 & \left( -\tau(t) \cdot \util \right) (\alpha(t),\psi_t(z)) \text{ on } (L-r,c] \times \C \\
 & \util (F(t),z) \text{ on } (c,L].
\end{aligned}
\right.
\end{equation}
By construction $\tilde U$ is $C^l$ since $F(t) = t-L$ near $L$. Each $\tilde U (t,\cdot)$ is an embedded fast
finite-energy plane asymptotic to $P$ at the positive puncture $\infty$. It follows that $\tilde U$ satisfies (1)
and (2) of Theorem~\ref{ob2}.


\subsubsection{Each page is a global surface of section}

Let $\tilde U = (a,u)$ be the map (\ref{mapU}). We claim that $\forall s\in \R/L\Z$ and $\forall p \in M \setminus
x(\R)$ we find sequences $t^\pm_n$ such that $t^+_n \rightarrow +\infty$, $t^-_n \rightarrow -\infty$ and
$\phi_{t^\pm_n}(p) \in u(s,\C)$.

Let $p \in M \setminus x(\R)$ be arbitrary, denote by $y_0:\R\rightarrow M$ be the Reeb trajectory satisfying
$y_0(0) = p$. We have to prove that the sets
\begin{gather*}
 C^+_n = \{s \in \R/L\Z : u(s,\C) \cap y_0([n,+\infty)) \not= \emptyset \} \\
 C^-_n = \{s \in \R/L\Z : u(s,\C) \cap y_0((-\infty,n]) \not= \emptyset \}
\end{gather*}
are equal to $\R/L\Z$ for every $n\in\Z$. We only prove $C^+_n = \R/L\Z \ \forall n$, the arguments for $C^-_n$ are
analogous. Let $\omega(p)$ be the $\omega$-limit set of $p$. If $\omega(p) \cap x(\R) = \emptyset$ then we find a
neighborhood $\OO$ of $x(\R)$ in $M$ such that $y_0([n,+\infty)) \cap \OO = \emptyset$. Using Lemma~\ref{longcyl}
exactly as in the proof of Lemma~\ref{surj} we find $r > 0$ such that $|z| > r \Rightarrow u(s,z) \in \OO \ \forall
s$. It follows that $y_0([n,+\infty)) \subset u( \R/L\Z \times \cl{B_r(0)} )$. It follows easily that $C^+_n$ is
closed $\forall n$. It is also open (each $u(s,\cdot) : \C \rightarrow M\setminus x(\R)$ is an embedding transversal
to the Reeb vector) and clearly non-empty. Thus $C^+_n = \R/L\Z \ \forall n\in\Z$ in this case. Now suppose
$\omega(p) \cap x(\R) \not= \emptyset$ and fix $s\in\R/L\Z$. Since $\tilde U(s,\cdot)$ is an embedded fast
finie-energy plane asymptotic the the orbit $P$ at the positive puncture $\infty$ then we conclude from
Lemma~\ref{omegalimit} that $s\in C^+_n$ for every $n\in\Z$. The conclusion follows.

It is proved in~\cite{convex} that the Poincar\'e return map $\psi:u(s,\C)\rightarrow u(s,\C)$ is an area-preserving
diffeomorphism with respect to the smooth area form $\omega = d\lambda|_{u(s,\C)}$, for every $s$. Clearly $\int
\omega = T$. They also show that $\psi$ is conjugated to a diffeomorphism of the open unit disk $\interior{\D}$
preserving the measure $\frac{T}{\pi}dx\wedge dy$. The proof of Theorem~\ref{ob2} can now be completed as explained
in the introduction.


\appendix

\section{A geometrical characterization of the index}\label{geomindex}

For a closed interval $I$ of length less than $\pi$ such that $2\pi\Z \cap \partial I = \emptyset$ consider the
integer $\hat\mu(I)$ defined by
\[
 \begin{aligned}
  & \hat \mu(I) = 2k \text{ if } 2\pi k\in I, \\
  & \hat \mu(I) = 2k+1 \text{ if } I \subset (2\pi k,2\pi(k+1)).
 \end{aligned}
\]
This function can be extended to the set of all intervals of length less that $\pi$ by
\[
 \hat \mu(I) = \lim_{\epsilon\rightarrow 0^+} \hat \mu(I-\epsilon).
\]

Fix a smooth $\varphi:[0,1]\rightarrow \Sp(1)$ with $\varphi(0) = I$. Let $S=-J_0\varphi^\prime\varphi^{-1}$. To
any $z_0 \in \C\setminus\{0\}$ we can associate the real number
\[
 \Delta(z_0) := \Delta\arg(z(t)) = \arg(z(1)) - \arg(z(0))
\]
where $z(t)$ solves
\begin{equation}\label{ed}
\left\{
\begin{aligned}
 & -J_0 \dot z - Sz = 0 \\
 & z(0) = z_0
\end{aligned}
\right.
\end{equation}
and $\arg$ is a continuous choice of argument. The interval $I_\varphi = \{ \Delta(z_0) : z_0 \not=0\}$ has length
less than $\pi$. This follows easily from the linearity of (\ref{ed}) and is explained in~\cite{fols}. Moreover, $\det [\varphi(1)-I] = 0$
if, and only if, $2\pi\Z \cap \partial I_\varphi \not= \emptyset$. It turns out, see~\cite{hkriener}, that the
$\mu$-index given by Theorem~\ref{axiomscz} satisfies
\[
 \mu(\varphi) = \hat \mu (I_\varphi) \ \forall \varphi \in \Sigma^*.
\]
This discussion provides an extension of the $\mu$-index to paths $\varphi \not\in \Sigma^*$, which coincides with
the extension explained in Section~\ref{compactness}.

\section{Asymptotical analysis for $\bar\partial_0$}

\subsection{Proof of Theorem~\ref{asympHWZ}}

\begin{proof}
We follow~\cite{props1} and proceed in four steps.
\\

\noindent {\bf STEP 1:} There exists $\mu \in \sigma(L_N) \cap (-\infty,0)$ and a $C^{l-1}$ function
$\alpha:[1,+\infty) \rightarrow \R$ satisfying $\lim_{s\rightarrow+\infty} \alpha(s) = \mu$ and
\[
 \Vert \zeta(s,\cdot) \Vert_{L^2(S^1)} = e^{\int_1^s \alpha(\tau)d\tau} \Vert \zeta(1,\cdot) \Vert_{L^2(S^1)} \text{ for } s\geq 1.
\]

\begin{proof}[Proof of STEP 1]
We write $L^2 = L^2(S^1,\R^{2k})$ and abbreviate $\norma{\cdot}_2 = \Vert\cdot\Vert_{L^2(S^1)}$.
$\escp{\cdot}{\cdot}_2$ denotes the inner-product on $L^2$ induced by the standard euclidean structure
$\escp{\cdot}{\cdot}$ on $\R^{2k}$. If $\exists s_0>0$ such that $\norma{\zeta(s,\cdot)}_2 = 0$ then the zero set of
$\zeta(s,t)$ has a limit point in $\R^+ \times S^1$. It follows from usual arguments using the similarity principle
that $\zeta \equiv 0$ on $\R^+\times S^1$, see~\cite{props1}. Thus we assume $\norma{\zeta(s,\cdot)}_2 \not= 0 \
\forall s>0$. Define $\xi(s,t) := \zeta(s,t)\norma{\zeta(s,\cdot)}_2^{-1}$ and
\begin{equation}\label{defnalpha}
 \alpha(s) := \escp{-J_0\xi_t(s,\cdot)-S(s,\cdot)\xi(s,\cdot)}{\xi(s,\cdot)}_2 = \escp{-J_0\xi_t-S\xi}{\xi}_2.
\end{equation}
It is not hard to see that
\begin{equation}\label{eqnnormzeta}
 \norma{\zeta(s,\cdot)}_2 = e^{\int_1^s \alpha(\tau)d\tau} \norma{\zeta(1,\cdot)}_2
\end{equation}
for $s\geq 1$, and that $\xi$ satisfies
\begin{align}\label{diffeqnxi}
	\xi_s = (L_N-\alpha+\epsilon)\xi
\end{align}
where $\epsilon(s,t) := N(t) - S(s,t)$ and $L_N = -J_0\partial_t - N(t)$. By our assumptions we have estimates
$\norma{\escp{\epsilon\xi}{\xi_s}_2} \leq o(s) \norma{\xi_s}_2$ and $\norma{\escp{\epsilon_s\xi}{\xi}_2} \leq o(s)$
for some $o(s) \to 0$ as $s\to +\infty$. Since $\escp{\xi_s}{\xi}_2 \equiv 0$ and $L_N$ is self-adjoint we obtain
\begin{equation}\label{estalphaprime}
  \alpha^\prime = 2\norma{\xi_s}^2_2 - \escp{\epsilon\xi}{\xi_s}_2 + \escp{\xi}{\epsilon\xi_s} + \escp{\epsilon_s\xi}{\xi}_2 \geq 2\norma{\xi_s}_2 \left[ \norma{\xi_s}_2 - o(s) \right] - o(s).
\end{equation}

We claim $\alpha(s)$ has a limit as $s\rightarrow+\infty$.
%
To see this define
\[
 \begin{array}{ccc}
   A = \liminf_{s \rightarrow +\infty} \alpha(s) & \text{and} & B = \limsup_{s \rightarrow +\infty} \alpha(s).
 \end{array}
\]
By contradiction, suppose $A<B$. This establishes an oscillatory behavior for $\alpha$. By Kato's perturbation theory of unbounded self-adjoint operators with compact resolvent, see~\cite{kato}, $\sigma(L_N)\subset\R$ is discrete and accumulates only at $\pm\infty$. Thus 
we find $r \in (A,B) \setminus \sigma(L_N)$ and  $s^\prime_n \rightarrow +\infty$ satisfying $\alpha(s^\prime_n) =
r$ and $\alpha^\prime(s^\prime_n) \leq 0$. Denote $a = \text{dist}(r,\sigma(L_N)) > 0$. Recall that for any
(possibly unbounded, closed and densely defined) self-adjoint operator $T$ on $L^2$, if $x\not\in\sigma(T)$ then $\|
(T-x)^{-1} \|^{-1} = \text{dist}(x,\sigma(T))$ where $\left\|\cdot\right\|$ is the operator norm on
$\Lcal\left(L^2\right)$. Using this fact with $L_N$, $\xi(s,\cdot)$ and $\alpha(s)$ we obtain
\begin{equation}\label{inversebounds}
 \begin{aligned}
  \text{dist}(\alpha(s),\sigma(L_N)) &= \text{dist}(\alpha(s),\sigma(L_N))|\xi|_2 \\
  &= \text{dist}(\alpha(s),\sigma(L_N))|(L_N-\alpha(s))^{-1}(L_N-\alpha(s))\xi|_2 \\
  &\leq \text{dist}(\alpha(s),\sigma(L_N))\left\|(L_N-\alpha(s))^{-1}\right\| |(L_N-\alpha(s))\xi|_2 \\
  &= \norma{(L_N-\alpha(s))\xi}_2
 \end{aligned}
\end{equation}
whenever $\alpha(s) \not\in \sigma(L_N)$. By (\ref{diffeqnxi}) and (\ref{inversebounds}) we can estimate
\begin{equation}\label{distspectrum}
 \norma{\xi_s}_2 \geq \norma{(L_N-\alpha(s))\xi}_2 - o(s) \geq \text{dist}(\alpha(s),\sigma(L_N)) - o(s),
\end{equation}
proving that $\norma{\xi_s(s^\prime_n,\cdot)}_2 \geq a/3$ if $n$ is large enough. Equation (\ref{estalphaprime}) implies that $\alpha^\prime(s^\prime_n) \geq a^2/9$ if $n$ is large enough, contradicting $\alpha^\prime(s^\prime_n)\leq 0 \ \forall n$. 

This discussion shows that $\lim_{s\rightarrow +\infty} \alpha(s) = \mu$ exists in $[-\infty,+\infty]$. Equation
(\ref{eqnnormzeta}) proves $\mu\leq0$ since $\norma{\zeta(s,\cdot)}_2$ is bounded. One can argue similarly as before
to prove $\mu>-\infty$. Note that $\liminf_{s\rightarrow+\infty} \norma{\xi_s}_2 = 0$ since, otherwise, we could use
(\ref{estalphaprime}) to get an estimate $\alpha^\prime(s)\geq \delta>0$, for $s$ large enough, and use
(\ref{eqnnormzeta}) to prove $\norma{\zeta(s,\cdot)}_2 \rightarrow +\infty$, an absurd. Pick $s_n\rightarrow\infty$
such that $\norma{\xi_s(s_n,\cdot)}_2 \rightarrow 0$. Then (\ref{distspectrum}) proves $\mu\in\sigma(L_N)$. By
assumption we know that $\lim_{s\rightarrow+\infty} \sup_t e^{rs}\norma{\zeta(s,t)} = 0$ for some $r>0$. Again
(\ref{eqnnormzeta}) proves $\mu<0$.
\end{proof}

In view of step 1 we can write
\begin{equation}\label{formulawithxi}
 \zeta(s,t) = \Vert \zeta(1,\cdot) \Vert_{L^2(S^1)} e^{\int_1^s \alpha(\tau)d\tau} \xi(s,t)
\end{equation}
where $\Vert \xi(s,\cdot) \Vert_{L^2(S^1)} \equiv 1$ for $s\geq 1$. In the following we denote $I_r(\tau) =
[\tau-r,\tau+r]$ and $Q_r(\tau) = I_r(\tau)\times S^1$, for $r>0$.
\\

\noindent {\bf STEP 2:} For every $1<p<\infty$ the functions $\xi$ and $\alpha$ satisfy
\[
  \limsup_{\tau\rightarrow+\infty} \left[ \norma{\xi}_{W^{l,p}(Q_1(\tau))} + \norma{\alpha}_{W^{l,p}(I_1(\tau))} \right] < \infty.
\]

\begin{proof}[Proof of STEP 2]
The argument relies on the elliptic estimates (\ref{ellipticestimates}) for the $\bar\partial_0$-operator. Equation
(\ref{diffeqnxi}) can be rewritten as
\begin{equation}\label{diffeqnxi2}
 \bar\partial_0\xi + S\xi + \alpha\xi = 0
\end{equation}
where $\alpha$ is the function (\ref{defnalpha}). We will prove by induction on $m$ that
\begin{equation}\label{indstepfinal}
 \begin{aligned}
  & \limsup_{\tau\rightarrow+\infty} \norma{\xi}_{W^{m,p}(Q_1(\tau))} < \infty \ \forall m\leq l \\
  & \limsup_{\tau\rightarrow+\infty} \norma{\alpha}_{W^{m,p}(I_1(\tau))} < \infty \ \forall m\leq l
 \end{aligned}
\end{equation}
for every $1<p<\infty$. These estimates for $m=0$ and $p=2$ follow from STEP 1 and from $\norma{\xi(s,\cdot)}_2
\equiv 1$.

Choose $\beta \in C^\infty(\R,[0,1])$ such that $\beta \equiv 1$ on $[-1,1]$ and $\beta \equiv 0$ on
$\R\setminus[-2,2]$. For each $\tau \in \R$ we define $\beta^\tau(s) := \beta(s-\tau)$. Assume $m\geq1$ and
$1<p<\infty$ are arbitrary. Using (\ref{diffeqnxi2}) and (\ref{ellipticestimates}) we estimate
\begin{equation}\label{useofellest}
 \begin{aligned}
  \norma{\xi}_{W^{m,p}(Q_1(\tau))} &\leq \norma{\beta^\tau \xi}_{W^{m,p}(Q_2(\tau))} \\
  &\leq c_m \left( \norma{\bar \partial_0 \left( \beta^\tau \xi \right)}_{W^{m-1,p}(Q_2(\tau))} + \norma{\beta^\tau \xi}_{W^{m-1,p}(Q_2(\tau))} \right) \\
  &\leq c_m^\prime \left( \norma{S\xi + \alpha\xi}_{W^{m-1,p}(Q_2(\tau))} + \norma{\xi}_{W^{m-1,p}(Q_2(\tau))} \right).
 \end{aligned}
\end{equation}
The constant $c_m^{\prime}>0$ depends only on $c_m$ from (\ref{ellipticestimates}) and on the derivatives of $\beta$
up to order $m$. Thus $c_m^\prime$ is independent of $\xi$, $\alpha$ and $\tau$.

We have to distinguish two cases: $m=1$ and $m>1$. If $m=1$ then
\[
 \norma{\alpha\xi}_{L^p(Q_2(\tau))} \leq \left(\sup_{s\geq\tau-2} \norma{\alpha(s)}\right) \norma{\xi}_{L^p(Q_2(\tau))}.
\]
STEP 1 implies $\alpha$ is bounded. Since $S$ is uniformly bounded on $Q_2(\tau)$, independently of $\tau$, we
conclude from (\ref{useofellest}) that $1<p<\infty$ implies
\[
 \limsup_{\tau\rightarrow+\infty} \norma{\xi}_{W^{1,p}(Q_1(\tau))} \leq c \limsup_{\tau\rightarrow+\infty} \norma{\xi}_{L^p(Q_2(\tau))}
\]
for some $c>0$ independent of $\tau$. The case $p=2$ proves $|\xi|_{W^{1,2}(Q_1(\tau))}$ is bounded in $\tau$. By
the Sobolev embedding theorem $W^{1,2}(Q_1(\tau)) \hookrightarrow L^p(Q_1(\tau))$ $\forall 1<p<\infty$, the
embedding constant being independent of $\tau$. This proves $|\xi|_{L^p(Q_1(\tau))}$ is bounded in $\tau$, for every
$1<p<\infty$. Repeating the above argument we conclude that
\begin{equation}\label{estxi1p}
 \sup_\tau |\xi|_{W^{1,p}(Q_1(\tau))} < \infty
\end{equation}
holds for every $1<p<\infty$. As a consequence
\[
  \begin{array}{ccc}
    s \mapsto |\xi_s(s,\cdot)|_{L^p(S^1)} & \text{and} & s \mapsto |\xi_t(s,\cdot)|_{L^p(S^1)}
  \end{array}
\]
belong to $L^p(I_1(\tau))$, for every $1<p<\infty$, their norms being bounded uniformly on $\tau$. Combining the
case $p=2$ with (\ref{estalphaprime}) we have
\[
 |\alpha^\prime(s)| \leq C(|\xi_s(s,\cdot)|^2_{L^2(S^1)} + |\xi_s(s,\cdot)|_{L^2(S^1)} + 1).
\]
for some $C>0$. Now fix any $1<q<\infty$. Recall that there is a linear embedding $L^{2q}(S^1) \hookrightarrow
L^2(S^1)$ with norm $\leq1$. We estimate the first term.
\[
 \begin{aligned}
  ||\xi_s(s,\cdot)|^2_{L^2(S^1)}|^q_{L^q(I_1(\tau))} & \leq ||\xi_s(s,\cdot)|^2_{L^{2q}(S^1)}|^q_{L^q(I_1(\tau))} \\
  & = \int_{\tau-1}^{\tau+1} \int_{S^1} |\xi_s(s,t)|^{2q} dtds \\
  & = |\xi_s|^{2q}_{L^{2q}(Q_1(\tau))}.
 \end{aligned}
\]
The other terms are easier. We conclude from (\ref{estxi1p}) that
$$ \limsup_{\tau\rightarrow\infty} |\alpha|_{W^{1,q}(I_1(\tau))} < \infty. $$
This proves (\ref{indstepfinal}) whenever $1<p<\infty$ and $m=1$.

Now fix $m\geq2$ and assume that (\ref{indstepfinal}) holds for $m-1$ and every $2<p<\infty$. Note that, since
$m-1\geq1$, pointwise multiplication yields a bilinear continuous form
\[
 W^{m-1,p}(I_2(\tau)) \times W^{m-1,p}(Q_2(\tau)) \rightarrow W^{m-1,p}(Q_2(\tau)).
\]
The norm of this bilinear map in independent of $\tau$. In other words, $\exists \hat c_m>0$ independent of $\tau$,
$\alpha$ and $\xi$ such that
\[
 \norma{\alpha\xi}_{W^{m-1,p}(Q_2(\tau))} \leq \hat c_m \norma{\alpha}_{W^{m-1,p}(I_2(\tau))}\norma{\xi}_{W^{m-1,p}(Q_2(\tau))}.
\]
Now we argue as before to prove $\sup_\tau |\xi|_{W^{m,p}(Q_1(\tau))} < \infty$ whenever $2<p<\infty$. The case
$1<p\leq2$ follows easily. In particular, $$ \sup_\tau ||D^\gamma\xi(s,\cdot)|_{L^p(S^1)}|_{L^p(I_1(\tau))} < \infty
\ \forall |\gamma|\leq m, \ \forall 1<p<\infty.
$$ Differentiating (\ref{estalphaprime}) and arguing as before, using the linear embedding $L^{2p}(S^1)
\rightarrow L^2(S^1)$ for $p>1$, one proves
$$ \sup_{\tau\geq 1} |\alpha|_{W^{m,p}(I_1(\tau))} < \infty \ \text{ for arbitrary }1<p<\infty. $$
This concludes the induction argument.
\end{proof}

STEP $2$ shows that $\alpha$ is not only $C^{l-1}$ but is also $W^{l,p}_{loc}$, for every $1<p<\infty$.
\\

\noindent {\bf STEP 3:} Let $E = \ker (L_N - \mu)$. Then
\[
 \lim_{s\rightarrow+\infty} \text{dist}_{W^{1,2}}(\xi(s,\cdot),E) = 0.
\]

\begin{proof}[Proof of STEP 3]
Fix $L>0$ and let $I_n$ be a sequence of closed intervals of length greater or equal than $L$ satisfying $\inf I_n
\rightarrow +\infty$. We claim that
\begin{equation}\label{lowerboundxis}
 A_n := \inf_{s\in I_n} \norma{\xi_s(s,\cdot)}_{L^2(S^1)} \rightarrow 0
\end{equation}
If not we can assume, after taking a subsequence, that $A_n \geq \delta$ for some $\delta>0$. It follows from
estimates (\ref{estalphaprime}) that $\alpha^\prime(s) \geq \hat\delta$ on $I_n$ if $n$ is large enough, for some
$\hat\delta>0$. Denoting $I_n = [a_n,b_n]$ then we obtain $\alpha(b_n) \geq \alpha(a_n) + \hat\delta L \ \forall n$,
contradicting $\alpha(s) \rightarrow \mu$ as $s\rightarrow+\infty$. This proves (\ref{lowerboundxis}).

By STEP 2 there exists $c>0$ such that
\[
 \norma{\xi(b,\cdot)-\xi(a,\cdot)}_{W^{1,2}(S^1)} \leq c\norma{b-a}
\]
if $\min\{a,b\}\geq1$. Here we used that the assumption $l\geq3$ gives us a linear embedding $W^{l,p}(Q_1(\tau))
\hookrightarrow C^2(Q_1(\tau))$ when $p>2$, the embedding constant being independent of $\tau$. Suppose $\exists s_n
\rightarrow +\infty$ such that $\inf_n \text{dist}_{W^{1,2}}(\xi(s_n,\cdot),E)>0$. The above estimate provides $L>0$
and closed intervals $I_n$ of length $\geq L$ such that $\inf_{s\in I_n} \text{dist}_{W^{1,2}}(\xi(s,\cdot),E) \geq
\epsilon$ for some $ \epsilon>0$, and $\inf I_n \rightarrow +\infty$. Let $\tau_n \in I_n$ satisfy
$\norma{\xi_s(\tau_n,\cdot)}_2 = \inf_{s\in I_n} \norma{\xi_s(s,\cdot)}_2$. By (\ref{lowerboundxis}) we know that
$\norma{\xi_s(\tau_n,\cdot)}_2 \rightarrow 0$. Equation (\ref{distspectrum}) proves
$\norma{(L_N-\mu)\xi(\tau_n,\cdot)}_2 \rightarrow 0$. We can assume $\exists e\in W^{1,2}$ such that
\[
 \norma{\xi(\tau_n,\cdot) - e}_{W^{1,2}} \stackrel{n\rightarrow\infty}{\longrightarrow} 0
\]
in view of the $C^2$-bounds obtained from STEP 2 when $l\geq3$ and $p>2$. Thus
\[
 (L_N-\mu)e = \lim_{n\rightarrow\infty} (L_N-\mu)\xi(\tau_n,\cdot) = 0 \text{ in }L^2
\]
proving $e\in E$. This contradiction concludes the argument.
\end{proof}

\noindent {\bf STEP 4:} The functions $\xi(s,t)$ and $\alpha(s)$ satisfy
\[
 \begin{aligned}
  & \lim_{s\rightarrow +\infty} \sup_{t\in S^1} \norma{D^\gamma[\xi(s,t)-e(t)]} = 0 \ \forall |\gamma|\leq l-2 \\
  & \lim_{s\rightarrow +\infty} \norma{D^j[\alpha(s)-\mu]} = 0 \ \forall j\leq l-2
 \end{aligned}
\]
for some $e(t) \in E$.

\begin{proof}[Proof of STEP 4]
First we claim $\exists e \in E$ such that
\begin{equation}\label{existenceofevector}
 \lim_{s\rightarrow+\infty} \norma{\xi(s,\cdot)- e }_{W^{1,2}} = 0
\end{equation}

\begin{proof}[Proof of (\ref{existenceofevector})]
By STEP 2 we have $\limsup_{s\rightarrow+\infty} \norma{\xi(s,\cdot)}_{W^{2,2}} < \infty$. Again we used $p>2$ and
$l\geq 3$. Consequently, for every sequence $s_n\rightarrow+\infty$ one finds $e\in W^{1,2}$, $\norma{e}_2 = 1$, and
a subsequence $s_{n_j}$ such that $\lim_j |\xi(s_{n_j},\cdot)-e|_{W^{1,2}}=0$. Fix $e \in W^{1,2}$ obtained by
taking such a limit. STEP 3 implies $e\in E$.

We claim $\xi(s,\cdot) \rightarrow e$ in $W^{1,2}$ as $s\rightarrow+\infty$. Let $P \in \Lcal(L^2)$ denote the
orthogonal projection onto $E$. Set $\hat\xi := P\xi$ and $\eta := \hat\xi|\hat\xi|^{-1}_2$. We have
$|\hat\xi(s,\cdot)|_2 \geq \frac{1}{2}$ if $s$ is large, in view of STEP 3 and of $|\xi|_2 \equiv 1$. As noted
in~\cite{props1}, $\eta$ satisfies
\[
 \eta_s = \frac{P\epsilon\xi}{|\hat\xi|_2} - \frac{\escp{\eta}{P\epsilon\xi}_2}{|\hat\xi|_2}\eta
\]
where $\epsilon(s,t) = N(t) - S(s,t)$. Since $\exists r>0$ such that $\sup_t e^{rs}|D^\gamma\epsilon(s,t)|
\rightarrow 0 \ \forall |\gamma|\leq l$ as $s\rightarrow +\infty$, we find $C>0$ satisfying $|\eta_s|_2 \leq
Ce^{-rs}$ for $s\gg1$. Consequently we estimate using H\"older's inequality
\[
 \begin{aligned}
  \norma{\eta(\tau_n,\cdot)-\eta(\hat\tau_n,\cdot)}_2 &\leq \sqrt{|\tau_n-\hat\tau_n|\int_{\hat\tau_n}^{\tau_n}|\eta_s(s,\cdot)|^2_{L^2(S^1)}ds} \\
  & \leq \sqrt{|\tau_n-\hat\tau_n| C \int_{\hat\tau_n}^{\tau_n} e^{-2ry}dy} \stackrel{n}{\longrightarrow} 0
 \end{aligned}
\]
for every pair of sequences $\hat\tau_n\leq\tau_n$ with $\hat \tau_n\rightarrow+\infty$. This proves
$\lim_{s\rightarrow+\infty} \eta(s,\cdot)$ exists in $L^2$. By STEP 3 we know that $|\eta(s,\cdot)-\xi(s,\cdot)|_2
\rightarrow 0$ as $s\rightarrow+\infty$. We proved $\lim_{s\rightarrow+\infty} \xi(s,\cdot)$ exists in $L^2$. The
conclusion follows.
\end{proof}

Since we have $C^0$-convergence of $\xi(s,\cdot)$ to $e$, the bounds obtained by STEP 2 and the Arzel\`a-Ascoli
theorem show that $$ \xi(s,\cdot) \rightarrow e \text{ in } C^{l-2}(S^1,\R^{2k}) \text{, as } s\rightarrow+\infty.
$$ STEP 4 follows from an easy induction argument using equations (\ref{diffeqnxi}) and (\ref{estalphaprime}).
\end{proof}

Formula (\ref{formulawithxi}) and STEP 4 imply Theorem~\ref{asympHWZ}.
\end{proof}

\subsection{Proof of Lemma~\ref{megatech}}

\begin{proof}
This proof can be found in~\cite{props1}, however the statement can not. Fix $1<p<\infty$ and $\beta \in
C^\infty(\R,[0,1])$ such that $\beta \equiv 1$ on $[-1,1]$ and $\beta \equiv 0$ on $\R\setminus[-2,2]$. For each
$\tau \in \R$ we define $\beta^\tau(s) := \beta(s-\tau)$ and denote $Q_r(\tau) = [\tau-r,\tau+r]\times S^1$. We will
prove
\begin{equation}\label{indstep+}
 \lim_{\tau\rightarrow+\infty} e^{d\tau} \norma{X}_{W^{k,p}(Q_1(\tau))} = 0 \ \forall k\geq 0
\end{equation}
by induction on $k\geq0$. The case $k=0$ is a direct consequence of (\ref{goodassumptions}). Now assume
(\ref{indstep+}) for $k-1\geq 0$. There is a semi-Fredholm estimate for the $\bar
\partial_0$-operator
\begin{equation}\label{ellipticestimates}
 \norma{f}_{W^{k,p}(Z)} \leq c_k \left( \norma{\bar \partial_0 f}_{W^{k-1,p}(Z)} + \norma{f}_{W^{k-1,p}(Z)} \right).
\end{equation}
This holds for every smooth $f$ with compact support on $Z$. Here $\bar\partial_0 = \partial_s + J_0\partial_t$.
Using (\ref{perturbedcr}) we can estimate for $\tau \gg 1$
\begin{equation}\label{apriori}
 \begin{aligned}
  \norma{X}_{W^{k,p}(Q_1(\tau))} &\leq \norma{\beta^\tau X}_{W^{k,p}(Q_2(\tau))} \\
  &\leq c_k \left( \norma{\bar \partial_0 \left( \beta^\tau X \right)}_{W^{k-1,p}(Q_2(\tau))} + \norma{\beta^\tau X}_{W^{k-1,p}(Q_2(\tau))} \right) \\
  &\leq c_k^\prime \left( \norma{KX}_{W^{k-1,p}(Q_2(\tau))} + \norma{X}_{W^{k-1,p}(Q_2(\tau))} \right).
 \end{aligned}
\end{equation}
The constant $c_k^{\prime}$ depends only on $c_k$ and on the derivatives of $\beta$. By (\ref{goodassumptions}) we
can estimate
\begin{equation}\label{apriori2}
 \norma{KX}_{W^{k-1,p}(Q_2(\tau))} \leq \hat c \norma{X}_{W^{k-1,p}(Q_2(\tau))}.
\end{equation}
The constant $\hat c \geq 0$ depends on $K$ but is independent of $\tau$. By (\ref{apriori}) and (\ref{apriori2}) we
have an estimate
\begin{equation}\label{apriori3}
 \norma{X}_{W^{k,p}(Q_1(\tau))} \leq c_k^{\prime\prime} \norma{X}_{W^{k-1,p}(Q_2(\tau))}
\end{equation}
for some constant $c_k^{\prime\prime} > 0$ independent of $\tau$. The induction hypothesis implies
\begin{equation}\label{apriori4}
 \lim_{s\rightarrow+\infty} e^{d\tau} \norma{X}_{W^{k-1,p}(Q_2(\tau))} = 0.
\end{equation}
Equations (\ref{apriori3}) and (\ref{apriori4}) complete the induction step. The conclusion now follows from the
Sobolev embedding theorem.
\end{proof}

\subsection{Proof of Lemma~\ref{megatech2}}

\begin{proof}
The argument can be found in section 4 of~\cite{props1}. We only include it here because we are making a more general statement. We proceed in three steps. \\

\noindent {\bf STEP 1:} If $\int_0^1 X(s,t)dt = 0, \ \forall s$ and $0<d<\frac{1}{2}$ then $e^{\rho s} \Vert X(s,\cdot)
\Vert_{L^2([0,1])} \rightarrow 0$ as $s\rightarrow +\infty$, for any $0<\rho<d$.

\begin{proof}[Proof of STEP 1]
We abbreviate $\norma{\cdot}_2 = \Vert\cdot\Vert_{L^2([0,1])}$ and denote by $\left<\cdot,\cdot\right>$ the $L^2$
inner-product on $L^2([0,1])$. Set $g(s) := \frac{1}{2}\norma{X(s,\cdot)}^2_2$. Then
\[
 g^\prime(s) = \left< X_s,X \right> = \left< -J_0X_t+h,X \right>
\]
and
\[
 \begin{aligned}
  g^{\prime\prime}(s) &= \escp{-J_0X_{ts}}{X} + \escp{-J_0X_t}{X_s} + \escp{h_s}{X} + \escp{h}{X_s} \\
  &= 2\escp{-J_0X_t}{X_s} + \escp{h_s}{X} + \escp{h}{X_s} \\
  &= 2\norma{X_t}_2^2 + 2\escp{-J_0X_t}{h} + \escp{h_s}{X} + \escp{h}{-J_0X_t+h} \\
  &\geq 2\norma{X_t}_2^2 - 2\norma{X_t}_2\norma{h}_2 - \norma{h_s}_2\norma{X}_2 - \norma{h}_2\norma{X_t}_2 - \norma{h}_2^2 \\
  &\geq 2\norma{X_t}_2^2 - \epsilon \left( \norma{X}^2_2 + \norma{X_t}^2_2 \right) - \frac{C}{\epsilon} \left( \norma{h}^2_2 + \norma{h_s}^2_2 \right).
 \end{aligned}
\]
Here $\epsilon>0$ is arbitrarily small, $C>0$ depends only on $\epsilon$, and they are both independent of $X$, $h$
and $s$. Our hypotheses imply that $\norma{X}_2 = \norma{X_t}_2$. If $\epsilon$ is small enough we obtain
\[
 g^{\prime\prime}(s) \geq g(s) - c e^{-2ds}
\]
for some $c > 0$. Choose $0<\nu<2d$ and set $L = \frac{c}{4d^2-\nu^2}$. Consider $f(s) = g(s) + Le^{-2ds}$. We have
a differential inequality $f^{\prime\prime} \geq \nu^2f$. We used that $4d^2<1$. Since $f(s) \rightarrow 0$ as
$s\rightarrow+\infty$ we must have $g(s) \leq f(s) \leq f(s_0)e^{-\nu (s-s_0)}$. This gives the desired conclusion
since $\nu/2 \in (0,d)$ can be taken arbitrarily.
\end{proof}

\noindent {\bf STEP 2:} If $0<d<\frac{1}{2}$ then $e^{\rho s} \Vert X(s,\cdot)\Vert_{L^2([0,1])} \rightarrow 0$ as
$s\rightarrow +\infty$ for every $0<\rho<d$.

\begin{proof}[Proof of STEP 2]
Set $\alpha(s) = \int_0^1 X(s,\tau)d\tau$ and $\bar X = X - \alpha$. Then
\[
 \bar X_s + J_0 \bar X_t = h(s,t) - \int_0^1h(s,\tau)d\tau =: \bar h.
\]
By STEP 1 we have
\[
 \lim_{s\rightarrow +\infty} e^{\rho s} \Vert \bar X (s,\cdot)\Vert_{L^2([0,1])} = 0, \ \forall \ 0<\rho<d.
\]
Since
\[
 \norma{\alpha^\prime(s)} = \norma{\int_0^1 h_s(s,\tau) d\tau}  \leq Ke^{-ds}
\]
for some $K>0$ we find $M_*>0$ such that $\norma{\alpha(s)} \leq M_* e^{-ds}$. If $0<\rho<d$ then
\[
 e^{\rho s} \Vert X(s,\cdot) \Vert_{L^2([0,1])} \leq e^{\rho s} \Vert \bar X(s,\cdot) \Vert_{L^2([0,1])} + e^{(\rho-d)s}e^{ds}\norma{\alpha(s)}  \stackrel{s\rightarrow+\infty}{\longrightarrow} 0.
\]
\end{proof}

\noindent {\bf STEP 3:} If $0<d<\frac{1}{2}$ then $e^{\rho s} \Vert D^\gamma X(s,\cdot)\Vert_{L^2([0,1])} \rightarrow
0$ as $s\rightarrow +\infty$ for any $\gamma$ and $0<\rho<d$.

\begin{proof}[Proof of STEP 3]
Fix $\gamma$. We have equations
\[
 \partial_s D^\gamma X + J_0 \partial_t D^\gamma X = D^\gamma h =: h_\gamma
\]
and $\sup_t e^{ds}\norma{D^\beta h_\gamma} \rightarrow 0$ as $s\rightarrow+\infty$, for any $\beta$. The conclusion
follows from STEP~2.
\end{proof}

Lemma~\ref{megatech2} follows from STEP 3 by the Sobolev embedding theorem, since we can assume, possibly after
making $d$ smaller, that $0<d<\frac{1}{2}$.
\end{proof}

\section{Proofs of technical lemmas}\label{techproofs}

\subsection{Proofs of Lemma~\ref{generaldecay} and Corollary~\ref{oops}}

Suppose $e^{bs}\norma{z(s,t)}$ is bounded for some $b>0$. Then we can assume, possibly after making $b>0$ smaller,
that
\begin{equation}\label{decayonz}
 \lim_{s\rightarrow+\infty} \sup_{t\in S^1} e^{bs} \norma{z(s,t)} = 0.
\end{equation}
We can find a smooth $Sp(1)$-valued function $L = L(\theta,z)$ defined on a neighborhood of $\R/\Z \times 0$ such
that $LJ = J_0L$ where $J$ is the matrix (\ref{jlocalrep}). Recall that $z$ satisfies (\ref{crz}) where $S$ is given
in (\ref{matrixS}) and $J(s,t) = J\circ w(s,t)$. Denote $L(s,t) = L \circ w(s,t)$. Then $\zeta(s,t) := L(s,t)z(s,t)$
satisfies
\[
 \zeta_s+J_0\zeta_t + \Lambda\zeta = 0
\]
where
\[
 \Lambda(s,t) = \left( LS-L_s-J_0L_t \right) L^{-1}.
\]
By Lemma~\ref{importantdecay} we know that $\norma{D^\gamma \Lambda(s,t)}$, $\norma{D^\gamma L(s,t)}$ and
$\norma{D^\gamma L^{-1}(s,t)}$ are bounded $\forall \gamma$. By (\ref{decayonz}) we can estimate
\begin{equation}\label{expdecayzeta}
 \lim_{s\rightarrow+\infty} \sup_{t\in S^1} e^{bs}\norma{\zeta(s,t)} = 0.
\end{equation}
Lemma~\ref{megatech} implies
\[
 \lim_{s\rightarrow+\infty} \sup_{t\in S^1} e^{bs}\norma{D^\gamma\zeta(s,t)} = 0 \ \forall \gamma.
\]
This proves
\begin{equation}\label{bla}
 \lim_{s\rightarrow+\infty} \sup_{t\in S^1} e^{bs}\norma{D^\gamma z(s,t)} = 0 \ \forall \gamma.
\end{equation}
Equations (\ref{crateta}) can be written as
\[
 \begin{pmatrix} a_s \\ \theta_s \end{pmatrix} + \begin{pmatrix} 0 & -T_{min} \\ T_{min}^{-1} & 0 \end{pmatrix} \begin{pmatrix} a_t \\ \theta_t \end{pmatrix} + B z = 0
\]
for some smooth matrix $B(s,t)$ satisfying $\sup \norma{D^\gamma B} < \infty \ \forall \gamma$. We used $f\equiv
T_{min}$ and $df \equiv 0$ on $\R/\Z \times 0$. The exponential decay of $z$ and its derivatives prove that if $h :=
Bz$ then
\[
 \lim_{s\rightarrow+\infty} \sup_{t\in S^1} e^{bs}\norma{D^\gamma h(s,t)} = 0 \ \forall \gamma.
\]
For simplicity assume $T_{min}=1$. Then $T_0 = k$. Denoting $\Delta = \begin{pmatrix} a - ks - \sigma \\
\theta - kt - k\hat d \end{pmatrix}$ we have $\Delta_s + J_0 \Delta_t + h = 0$. Definition~\ref{behavior} implies
$\sup_{t\in S^1} \norma{\Delta(s,t)} \rightarrow 0$ as $s\rightarrow+\infty$. Lemma~\ref{megatech2} provides $r>0$
such that
\begin{equation}\label{decayonateta}
 \lim_{s\rightarrow+\infty} \sup_{t\in S^1} e^{rs}\norma{D^\gamma \Delta(s,t)} = 0 \ \forall \gamma.
\end{equation}
Equations (\ref{decayonateta}) and (\ref{bla}) prove (\ref{expgeneraldecay}). Obviously if $\util$ has
non-degenerate asymptotics then (\ref{decayonz}) is true. This completes the proof of Lemma~\ref{generaldecay}.

We now turn to the proof of Corollary~\ref{oops}. For simplicity we assume $\sigma=\hat d=0$. Define $L_\infty(t) =
L(kt,0)$. The definition of $L(s,t)$ and Lemma~\ref{generaldecay} together imply
\[
 \lim_{s\rightarrow+\infty} \sup_t e^{rs} \norma{D^\gamma[L(s,t) - L_\infty(t)]} = 0 \ \forall \gamma
\]
for some $r>0$. Again Lemma~\ref{generaldecay} and the definition of $S(s,t)$ in (\ref{matrixS}) imply
\[
 \lim_{s\rightarrow+\infty} \sup_t e^{rs} \norma{D^\gamma[\Lambda(s,t) - \Lambda_\infty(t)]} = 0 \ \forall \gamma
\]
where $\Lambda_\infty(\theta) = [L(k\theta,0)N(\theta)-kJ_0L_\theta(k\theta,0)]L(k\theta,0)^{-1}$. We compute
\[
 \begin{aligned}
  (-J_0\partial_t-\Lambda_\infty)L_\infty f &= -J_0(\partial_tL_\infty) f - J_0 L_\infty f_t \\
  &- L_\infty NL_\infty^{-1}L_\infty f + J_0(\partial_tL_\infty) L_\infty^{-1}L_\infty f \\
  &= L_\infty(-J(kt,0)f_t - Nf).
 \end{aligned}
\]
This proves $-J_0\partial_t-\Lambda_\infty$ is just a representation of the asymptotic operator $A_P$ in a different
symplectic frame. Moreover, $\Lambda(\theta)^T = \Lambda(\theta) \ \forall \theta$ since $L$ is $Sp(1)$-valued. The
asymptotic behavior of $\Lambda$ and (\ref{expdecayzeta}) allow us to apply Theorem~\ref{asympHWZ} to $\zeta(s,t)$.
The conclusion follows immediately from the formula $z(s,t)=L^{-1}(s,t)\zeta(s,t)$ and from the asymptotic behavior
of $L(s,t)$.

\end{document}